\documentclass[a4paper,11pt]{amsart}

\usepackage[top=30truemm,bottom=25truemm,left=35truemm,right=35truemm]{geometry}

\usepackage{amsmath}
\usepackage{amssymb}
\usepackage{amsthm}
\usepackage{graphics}
\usepackage[abbrev,alphabetic]{amsrefs}
\usepackage{amscd}
\usepackage[dvipdfmx]{graphicx}

\usepackage[all,cmtip]{xy}
\usepackage{xypic}

\title[Inversion of adjunction for quotient singularities]
{Inversion of adjunction for quotient singularities}

\author{Yusuke Nakamura}
\address{Graduate School of Mathematical Sciences, 
the University of Tokyo, 3-8-1 Komaba, Meguro-ku, Tokyo 153-8914, Japan.}
\email{nakamura@ms.u-tokyo.ac.jp}

\author{Kohsuke Shibata}
\address{Department of Mathematics, College of Humanities and Sciences, 
Nihon University, 3-25-40 Sakurajosui, Setagaya-ku, Tokyo 156-8550, Japan}
\email{shibata.kohsuke@gmail.com}



\newtheorem{thm}{Theorem}[section]
\newtheorem{lem}[thm]{Lemma}

\newtheorem{cor}[thm]{Corollary}
\newtheorem{prop}[thm]{Proposition}

\newtheorem{claim}[thm]{Claim}

\theoremstyle{definition}
\newtheorem{defi}[thm]{Definition}
\newtheorem{eg}[thm]{Example}
\newtheorem{conj}[thm]{Conjecture}

\theoremstyle{remark}
\newtheorem{rmk}[thm]{Remark}
\newtheorem*{ackn}{Acknowledgements}

\begin{document}
\begin{abstract}
We prove the precise inversion of adjunction formula for quotient singularities and klt Cartier divisors. 
As an application, we prove the semi-continuity of minimal log discrepancies for klt hyperquotient singularities. 
\end{abstract}

\maketitle

\tableofcontents

\section{Introduction}
The minimal log discrepancy is an invariant of singularities defined in birational geometry. 
Shokurov proved that two conjectures on the minimal log discrepancies, the LSC (lower semi-continuity) conjecture 
and the ACC (ascending chain condition) conjecture, imply the conjecture of termination of flips \cite{Sho04}. 
We refer the reader \cite{Kaw14, Kaw15, Nak16b, MN18, Kaw21, HLS} for the recent development related to the ACC conjecture. 
In this paper, we focus on the LSC conjecture. 
We always work over an algebraically closed field of characteristic zero unless otherwise stated.

\begin{conj}[LSC conjecture]\label{conj:LSC}
Let $(X, \mathfrak{a})$ be a log pair, and let $|X|$ be the set of all closed points of $X$ with the Zariski topology. 
Then the function 
\[
|X| \to \mathbb{R}_{\ge 0} \cup \{ - \infty \}; \quad x \mapsto \operatorname{mld}_x (X,\mathfrak{a})
\]
is lower semi-continuous. 
\end{conj}
The LSC conjecture is known to be true in the following cases: 
\begin{itemize}
\item[(\ref{conj:LSC}.1)] The case when $\dim X \le 3$ \cite{Amb99}. 
\item[(\ref{conj:LSC}.2)] The case when $X$ is smooth \cite{EMY03}. 
\item[(\ref{conj:LSC}.3)] More generally, the case when $X$ is a normal local complete intersection variety \cite{EM04}. 
\item[(\ref{conj:LSC}.4)] The case when $X$ has only quotient singularities \cite{Nak16}. 
\item[(\ref{conj:LSC}.5)] The case when $X$ is smooth variety over  an algebraically closed field of arbitrary characteristic under some  condition  \cite{Shi}.
\end{itemize}
The main purpose of this paper is to prove the LSC conjecture for varieties with hyperquotient singularities, more generally, 
the quotient of a complete intersection variety by a finite linear group action. 
\begin{thm}[{$=$\ Theorem \ref{thm:LSC_general}}]\label{thm:LSC}
Suppose a finite subgroup $G \subset {\rm GL}_N(k)$ acts on $\mathbb{A}_k^N$ freely in codimension one. 
Let $X := \mathbb{A}_k^N / G$ be the quotient variety. 
Let $Y$ be a subvariety of $X$ of codimension $c$ which has only klt singularities, and 
let $\mathfrak{a}$ be an $\mathbb{R}$-ideal sheaf on $Y$. 
Suppose that $Y$ is locally defined by $c$ equations in $X$. 
Then the function 
\[
|Y| \to \mathbb{R}_{\ge 0} \cup \{ - \infty \}; \quad y \mapsto \operatorname{mld}_y(Y,\mathfrak{a})
\]
is lower semi-continuous, where we denote by $|Y|$ the set of all closed points of $Y$ with the Zariski topology. 
\end{thm}

In this paper, we also treat the PIA (precise inversion of adjunction) conjecture. 

\begin{conj}[PIA conjecture, {\cite[17.3.1]{92}}]\label{conj:PIA}
Let $(X, \mathfrak{a})$ be a log pair and let $D$ be a normal Cartier prime divisor. 
Let $x \in D$ be a closed point. 
Suppose that $D$ is not contained in the cosupport of the $\mathbb{R}$-ideal sheaf $\mathfrak{a}$. 
Then 
\[
\operatorname{mld}_x \bigl( X, \mathfrak{a} \mathcal{O}_X(-D) \bigr) = \operatorname{mld}_x (D, \mathfrak{a} \mathcal{O}_D)
\]
holds. 
\end{conj}

\noindent
The PIA conjecture is known to be true in the following cases: 
\begin{itemize}
\item[(\ref{conj:PIA}.1)] The case when $X$ is smooth \cite{EMY03}. 
\item[(\ref{conj:PIA}.2)] More generally, the case when $X$ is a normal local complete intersection variety \cite{EM04}. 
\end{itemize}

In this paper, we study the PIA conjecture for varieties with quotient singularities.

\begin{thm}[{$=$\ Corollary \ref{cor:PIA_general}}]\label{thm:PIA_intro}
Suppose that a finite subgroup $G \subset {\rm GL}_N(k)$ acts on $\mathbb{A}_k^N$ freely in codimension one. 
Let $X := \mathbb{A}_k^N / G$ be the quotient variety and let $x \in X$ be the image of the origin of $\mathbb{A}_k^N$. 
Let $Y$ be a subvariety of $X$ through $x$ of codimension $c$, and 
let $\mathfrak{a}$ be an $\mathbb{R}$-ideal sheaf on $Y$. 
Suppose that $Y$ is locally defined by $c$ equations at $x$ in $X$. 
Let $D$ be a Cartier prime divisor on $Y$ through $x$ with a klt singularity at $x \in D$. 
Suppose that $D$ is not contained in the cosupport of the $\mathbb{R}$-ideal sheaf $\mathfrak{a}$. 
Then it follows that
\[
\operatorname{mld}_x \bigl( Y, \mathfrak{a} \mathcal{O}_Y (-D) \bigr) = 
\operatorname{mld}_x (D, \mathfrak{a} \mathcal{O}_D). 
\]
\end{thm}

\noindent
By Theorem \ref{thm:PIA_intro}, 
Theorem \ref{thm:LSC} can be reduced to the case when 
$X$ has quotient singularities (\ref{conj:LSC}.4). 
Hence, this paper is mainly devoted to proving Theorem \ref{thm:PIA_intro}.

The main tools of this paper involve 
the theory of the arc space of a quotient singularity established by Denef and Loeser in \cite{DL02} and 
the technique on arc spaces for proving (\ref{conj:PIA}.1) established by Ein, Musta\c{t}\u{a} and Yasuda in \cite{EMY03}. 

By the theory of Denef and Loeser, the arc space of quotient variety can be studied by those of certain $k[t]$-schemes. 
We review their theory here briefly. 
Suppose that a finite group $G \subset GL_{N}(k)$ of order $d$ acts on $\overline{X} = \operatorname{Spec} k[x_1, \ldots , x_N]$. 
Let $\overline{Y} \subset \overline{X}$ be a $G$-invariant closed subvariety and $I \subset k[x_1, \ldots , x_N]$ its defining ideal. 
We denote by $Y := \overline{Y}/G$ its quotient. 
For each $\gamma \in G$, a $k[t]$-scheme $\overline{Y}^{(\gamma)}$ is defined as follows. 
By changing the basis $x_1, \ldots, x_N$ linearly, we may assume that $\gamma$ is a diagonal matrix 
with entries $\xi^{e_1}, \ldots , \xi^{e_N}$ ($0 \le e_i \le d-1$) where $\xi$ is a primitive $d$-th root of unity in $k$. 
Let $\overline{\lambda}_{\gamma}^*$ be the ring homomorphism defined by
\[
\overline{\lambda}_{\gamma}^*: k[x_1, \ldots , x_N]^G \to k[t][x_1, \ldots , x_N]; \quad x_i \mapsto t^{\frac{e_i}{d}}x_i. 
\] 
Then the $k[t]$-scheme $\overline{Y}^{(\gamma)}$ is defined by
\[
\overline{Y}^{(\gamma)} = \operatorname{Spec} k[t][x_1, \ldots , x_N]/ \overline{I}^{(\gamma)}, 
\]
where $\overline{I}^{(\gamma)}$ is the ideal generated by elements of $\overline{\lambda}_{\gamma}^* (I)$. 
We denote by $Y_{\infty}$ the arc space of $Y$, which parametrizes $k$-morphisms $\operatorname{Spec} k[[t]] \to Y$. 
We also denote by $\overline{Y}^{(\gamma)}_{\infty}$ 
the arc space of $\overline{Y}^{(\gamma)}$ which parametrizes $k[t]$-morphisms $\operatorname{Spec} k[[t]] \to \overline{Y}^{(\gamma)}$. 
Denef and Loeser in \cite{DL02} investigate the change of variables formula on the map $\overline{Y}^{(\gamma)}_{\infty} \to Y_{\infty}$, 
and their theory allows us to compare the two spaces $Y_{\infty}$ and $\bigsqcup _{\gamma \in G} \overline{Y}^{(\gamma)}_{\infty}$. 
In \cite{DL02}, they basically work in the case where $I = 0$ since they are interested in the quotient singularity. 
In Sections \ref{section:prelimi} and \ref{section:DL2}, we explain their theory in a more general setting in detail. 
We will give self-contained proofs to most of the propositions. However, we emphasize that they are not our original. 
Most of the propositions in Section \ref{section:prelimi} follow from the existing work by Sebag \cite{Seb04} and Yasuda \cite{Yas} 
(cf.\ \cite{Yas04, Yas06}). 
In \cite{Seb04}, Sebag develops the theory of motivic integration for formal schemes over a complete DVR. 
In \cite{Yas}, Yasuda develops that for formal Deligne-Mumford stacks over $k[[t]]$ of arbitrary characteristic. 
The construction in Section \ref{section:DL2} is a special case of the construction in \cite{Yas16}. 
The construction in \cite{Yas16} is intrinsic and more general, and it works even in positive characteristics. 
See Remarks \ref{rmk:Sebag} and \ref{rmk:Yas16} for more detail.

By the result of Ein, Musta\c{t}\u{a} and Yasuda in \cite{EMY03} (and \cite{EM09}), 
the minimal log discrepancy of $Y$ can be described by the codimension of certain contact loci in $Y_{\infty}$. 
Then by applying the theory of Denef and Loeser above, it can be described 
by the codimension of the corresponding contact loci in $\overline{Y}^{(\gamma)}_{\infty}$. 
This description (Theorem \ref{thm:mld_hyperquot}) is one of the key steps to prove Theorem \ref{thm:PIA_intro}. 

In Section \ref{section:mld_hyperquot}, to prove Theorem \ref{thm:PIA_intro}, we apply the technique by Ein, Musta\c{t}\u{a} and Yasuda in \cite{EMY03}  (and \cite{EM09})
to the contact loci in $\overline{Y}^{(\gamma)}_{\infty}$. 
Their argument works basically well even in our setting 
because $\overline{I}^{(\gamma)}$ is generated by a regular sequence outside $t=0$ when $I$ is generated by a regular sequence. 
However, there are two main difficulties in this step as we discuss below. 

The first difficulty is that $\overline{Y}^{(\gamma)}$ is neither normal nor a complete intersection in general (see Remark \ref{rmk:nonlci}). 
Therefore we have no standard definition of the relative canonical sheaf $\omega _{\overline{Y}^{(\gamma)}/k[t]}$ on $\overline{Y}^{(\gamma)}$. 
We overcome this difficulty by defining an invertible sheaf $L_{\overline{Y}^{(\gamma)}}$ instead. 
Furthermore, Lemma \ref{lem:age2}, which relates the age of $\gamma$ and certain orders of arcs, 
is a key lemma for the argument in \cite{EMY03} to work in our setting. 

The second difficulty is that there may be very few arcs on $\overline{Y}^{(\gamma)}$. 
More precisely, the arc space $\overline{Y}^{(\gamma)}_{\infty}$ may be a thin set of $\overline{Y}^{(\gamma)}_{\infty}$ itself 
(see Definition \ref{def:thin} for the definition of a thin set). For example, 
if $(d,e_1,e_2,e_3)=(3,0,1,2)$ and $I = (x_1^3 + x_2^3 + x_3^3)$, then we have 
\[
\overline{Y}^{(\gamma)} = \operatorname{Spec} k[t][x_1, x_2, x_3]/ (x_1^3 + tx_2^3 + t^2x_3^3). 
\]
In this case, $\overline{Y}^{(\gamma)}_{\infty}$ consists of only one arc, and its order of the Jacobian ideal 
$\operatorname{Jac}_{\overline{Y}^{(\gamma)}/k[t]}$ is infinity. 
If the arc space $\overline{Y}^{(\gamma)}_{\infty}$ is a thin set, 
then any arc $\alpha \in \overline{Y}^{(\gamma)}_{\infty}$ has order $\operatorname{ord}_{\alpha} (\operatorname{Jac}_{\overline{Y}^{(\gamma)}/k[t]}) = \infty$, 
and because of it, the argument in \cite{EMY03} does not work. 
In Claim \ref{claim:not_thin}, we prove that $\overline{Y}^{(\gamma)}_{\infty}$ is not a thin set if $Y$ is klt and show 
that the argument in \cite{EMY03} really works. 
The key idea in the proof of Claim \ref{claim:not_thin} is to apply the result by Hacon and Mckernan \cite{HM07}, 
which states the rational chain connectedness of the fibers of the resolution 
$W \to \overline{Y}^{(\gamma)}$ of singularities of $\overline{Y}^{(\gamma)}$, 
and prove that there are actually many arcs on $W$ using the result by Graber, Harris, and Starr \cite{GHS03}. 
The klt assumption in Theorem \ref{thm:PIA_intro} is essentially used in this argument.

We also prove the Reid-Tai type formula on minimal log discrepancies (Corollary \ref{cor:qmld=cqmld}). 
This gives the affirmative answer to a question by Borisov in \cite{Bor97} whether 
the set of minimal discrepancies of quotient singularities 
with respect to arbitrary groups coincides with that of cyclic quotients of the same dimension. 
We prove Corollary \ref{cor:qmld=cqmld} using the description of 
the minimal log discrepancy in terms of arc spaces of the $k[t]$-schemes. 
Moreover, we give another proof without the theory of arc spaces.
As an application of this result, we prove the ACC conjecture for quotient singularities (Theorem \ref{thm:acc_quot}).

The paper is organized as follows. 
In Section \ref{section:prelimi}, we review some definitions and facts on pairs and arc spaces. 
We also prove some basic results on the arc spaces of $k[t]$-schemes, which are necessary for this paper. 
Especially, we discuss the theory of contact loci and their codimension for the arc spaces of $k[t]$-schemes following \cite{EM09}, 
where the arc spaces of $k$-schemes are dealt with. 
In Section \ref{section:DL2}, we review the theory of arc spaces of quotient varieties established by Denef and Loeser in \cite{DL02}. 
As mentioned previously, most of the propositions in Sections \ref{section:prelimi} and \ref{section:DL2} follow 
from the existing works by Sebag \cite{Seb04} and Yasuda \cite{Yas, Yas16}. 
Readers who are familiar with these papers could skip these sections. 
In Section \ref{section:mld_hyperquot}, we discuss the minimal log discrepancy of quotient singularities of linear action, and 
describe them by the codimension of cylinders in arc spaces of 
the $k[t]$-schemes defined in the previous section (Theorem \ref{thm:mld_hyperquot}). 
In Section \ref{section:PIA}, we prove the PIA conjecture for hyperquotient singularities (Theorem \ref{thm:PIA}). 
In Section \ref{section:maintheorem}, we prove the main theorems Corollary \ref{cor:PIA_general} and Theorem \ref{thm:LSC_general} with some generalizations.

\begin{ackn} 
We would like to thank Professor Takehiko Yasuda for the discussion and many suggestions. 
We also thank the referee, whose comments and suggestions have greatly improved the article. 
The referee pointed out the relevance to the papers \cite{Seb04}, \cite{Yas16}, and \cite{Yas}. 
Remarks \ref{rmk:Sebag} and \ref{rmk:Yas16} are largely owing to the referee's comments. 
The first author is partially supported by the Grant-in-Aid for Young Scientists (KAKENHI No.\ 18K13384).
The second author is partially supported by the Grant-in-Aid for Young Scientists (KAKENHI No.\ 19K14496).
\end{ackn}

\section{Preliminaries}\label{section:prelimi}

\subsection{Notation}\label{subsection:notation}

\begin{itemize}
\item 
We basically follow the notations and the terminologies in \cite{Har77} and \cite{Kol13}.

\item 
Throughout this paper, $k$ is an algebraically closed field  of characteristic zero. 
We say that $X$ is a \textit{variety over} $k$ or a \textit{$k$-variety} if 
$X$ is an integral scheme that is separated and of finite type over $k$. 
\end{itemize}

\subsection{Log pairs}
A \textit{log pair} $(X, \mathfrak{a})$ is a normal $\mathbb{Q}$-Gorenstein variety $X$ and 
an $\mathbb{R}$-ideal sheaf $\mathfrak{a}$ on $X$. 
Here, an $\mathbb{R}$-\textit{ideal sheaf} $\mathfrak{a}$ on $X$ is a formal product 
$\mathfrak{a} = \prod _{i = 1} ^s \mathfrak{a}_i ^{r_i}$, where $\mathfrak{a}_1, \ldots, \mathfrak{a}_s$ are 
non-zero coherent ideal sheaves on $X$ 
and $r_1, \ldots , r_s$ are positive real numbers. 
For a morphism $Y \to X$ and an $\mathbb{R}$-ideal sheaf $\mathfrak{a} = \prod _{i = 1} ^s \mathfrak{a}_i ^{r_i}$, 
we denote by $\mathfrak{a} \mathcal{O}_Y$ the $\mathbb{R}$-ideal sheaf $\prod _{i = 1} ^s (\mathfrak{a}_i \mathcal{O}_Y)  ^{r_i}$ on $Y$. 

Let $(X, \mathfrak{a} = \prod _{i = 1} ^s \mathfrak{a}_i ^{r_i})$ be a log pair. 
For a proper birational morphism $f: X' \to X$ from a normal variety $X'$ and a prime divisor $E$ on $X'$, 
the \textit{log discrepancy} of $(X, \mathfrak{a})$ at $E$ is defined as 
\[
a_E(X, \mathfrak{a}) := 1 + \operatorname{ord}_E (K_{X'} - f^* K_X) - \operatorname{ord}_E ( \mathfrak{a} ), 
\]
where we denote $\operatorname{ord}_E ( \mathfrak{a} ) = \sum _{i=1} ^s r_i \operatorname{ord}_E ( \mathfrak{a}_i )$. 
The image $f(E)$ is called the \textit{center} of $E$ on $X$ and we denote it by $c_X(E)$. 
For a closed point $x \in X$, we define \textit{the minimal log discrepancy} at $x$ as 
\[
\operatorname{mld}_x (X, \mathfrak{a}) := \inf _{c_X(E) = \{ x \}} a_E (X, \mathfrak{a})
\]
if $\dim X \ge 2$, where the infimum is taken over all prime divisors $E$ over $X$ with center $c_X(E) = \{ x \}$. 
It is known that $\operatorname{mld}_x (X, \mathfrak{a}) \in \mathbb{R}_{\ge 0} \cup \{ - \infty \}$ in this case (cf.\ \cite[Corollary 2.31]{KM98}). 
When $\dim X = 1$, we define $\operatorname{mld}_x (X, \mathfrak{a}) := \inf _{c_X(E) = \{ x \}} a_E (X, \mathfrak{a})$ 
if the infimum is non-negative and 
$\operatorname{mld}_x (X, \mathfrak{a}) := - \infty$ otherwise.

\subsection{Jet schemes and arc spaces for $k$-schemes}\label{subsection:jet}
In this subsection, we briefly review the definition and some properties of jet schemes and arc spaces.
The reader is referred to \cite{EM09} for details.

Let $X$ be a scheme of finite type over $k$, let $({\sf Sch}/k)$ be the category of $k$-schemes  
 and $({\sf Sets})$ the category of sets.
Define a contravariant functor  $F_{m}: ({\sf Sch}/k) \to ({\sf Sets})$ by 
\[
 F_{m}(Y) = \operatorname{Hom} _{k}\left( Y\times_{\operatorname{Spec} k} \operatorname{Spec} k[t]/(t^{m+1}), X \right).
\]
Then, the functor $F_{m}$ is representable by a scheme $X_m$ of finite type over $k$, and 
the scheme $X_m$ is called the $m$-th \textit{jet scheme} of $X$.
For $m \ge n \ge 0$, the canonical surjective homomorphism $k[t]/(t^{m+1}) \to k[t]/(t^{n+1})$ induces a morphism $\pi_{mn}:X_m \to X_n$.
There exists the projective limit and projections
\[
X_\infty := \mathop{\varprojlim}\limits_{m} X_m, \qquad \psi_{m}:X_{\infty} \to X_m
\]
and $X_{\infty}$ is called the {\it arc space} of $X$. 
Then there is a bijective map
\[
\operatorname{Hom} _{k}(\operatorname{Spec} K, X_{\infty}) \simeq \operatorname{Hom} _{k}(\operatorname{Spec} K[[t]], X)
\]
for any field $K$ with $k\subset K$. 

For $m \in \mathbb{Z}_{\ge 0} \cup \{ \infty \}$, we denote by $\pi_m: X_m \to X$ the canonical truncation morphism. 
For $m\in\mathbb Z_{\ge 0}\cup\{\infty\}$ and a morphism $f:Y\to X$ of  schemes  of  finite type over $k$,
we denote by $f_m : Y_m \to X_m$ the morphism induced by $f$.

A subset $C \subset X_{\infty}$ is called a \textit{cylinder} if $C = \psi_{m} ^{-1}(S)$ holds for some $m \ge 0$ and 
a constructible subset $S \subset X_m$. Typical examples of cylinders appearing in this paper are the \textit{contact loci} 
$\operatorname{Cont}^{m}(\mathfrak{a})$ and $\operatorname{Cont}^{\geq m}(\mathfrak{a})$ defined as follows. 
\begin{defi}
\begin{enumerate}
\item For an arc $\gamma\in X_\infty$ and an ideal sheaf $\mathfrak{a} \subset \mathcal O_X$, 
the \textit{order} of $\mathfrak{a}$  measured by $\gamma$
is defined as follows:
\[
\operatorname{ord}_\gamma(\mathfrak a)=\sup \{ r\in \mathbb Z_{\geq 0}\mid \gamma^*(\mathfrak a)\subset (t^r)\}, 
\]
where $\gamma^*: \mathcal{O}_{X} \to k[[t]]$ is the induced ring homomorphism by $\gamma$.

\item For $m \in \mathbb{Z}_{\ge 0}$, we define $\operatorname{Cont}^{m}(\mathfrak{a}), \operatorname{Cont}^{\geq m}(\mathfrak{a}) \subset X_{\infty}$ as follows:
\begin{align*}
\operatorname{Cont}^{m}(\mathfrak{a}) &= \{\gamma \in X_\infty \mid \operatorname{ord}_\gamma(\mathfrak a)= m\}, \\
\operatorname{Cont}^{\geq m}(\mathfrak{a}) &= \{\gamma \in X_\infty \mid \operatorname{ord}_\gamma(\mathfrak a)\geq m\}.
\end{align*}
\end{enumerate}
\end{defi}

\noindent
By the definition, we can see that
\[
\operatorname{Cont}^{\geq m}(\mathfrak{a})=\psi_{m-1}^{-1}(Z(\mathfrak{a})_{m-1}),
\]
where $Z(\mathfrak{a})$ is the closed subscheme of $X$ defined by the ideal sheaf $\mathfrak{a}$. 
Therefore $\operatorname{Cont}^{m}(\mathfrak{a})$ and $\operatorname{Cont}^{\geq m}(\mathfrak{a})$ become cylinders. 

For $m \le n+1$, we also define the subsets 
$\operatorname{Cont}^{m}(\mathfrak a)_n$ and $\operatorname{Cont}^{\geq m}(\mathfrak a)_n$ of $X_n$ in the same way.

We shall define the codimension for cylinders. 
For a variety $X$ of dimension $n$, we denote by 
$\operatorname{Jac}_X := \operatorname{Fitt} ^n (\Omega _X)$ the \textit{Jacobian ideal} of $X$, and 
by $X_{\rm{sing}}$ the singular locus of $X$ (see \cite{Eis95} for the definition of the Fitting ideal). 
\begin{defi}\label{defi:codim}
Let $X$ be a variety and let $C \subset X_{\infty}$ be a cylinder. 
\begin{enumerate}
\item 
Assume that $C \subset \operatorname{Cont}^e(\operatorname{Jac}_X)$ for some $e \in \mathbb{Z}_{\ge 0}$.
Then we define the codimension of $C$ in $X_\infty$ as 
\[
\operatorname{codim}(C):=(m+1)\operatorname{dim}X-\operatorname{dim}(\psi_m(C))
\]
for any sufficiently large $m$. This definition is well-defined by \cite[Proposition 4.1]{EM09}.

\item 
In general, we define the codimension of $C$ in $X_\infty$ as follows:
\[
\operatorname{codim}(C):=\min_{e\in\mathbb Z_{\ge 0}} {\operatorname{codim}(C\cap \operatorname{Cont}^e(\operatorname{Jac}_X))}. 
\]
By convention, $\operatorname{codim}(C) = \infty$ if $C \subset (X_{\rm{sing}})_{\infty}$. 
\end{enumerate}
\end{defi}

We recall the definition of the Nash ideals of varieties and morphisms. 

\begin{defi}\label{defi:nash}
\begin{enumerate}
\item 
Let $X$ be a normal $\mathbb{Q}$-Gorenstein variety over $k$ of dimension $n$ and 
let $r$ be a positive integer such that the reflexive power $\omega _X ^{[r]} := (\omega _X ^{\otimes r})^{**}$ is an invertible sheaf. 
Then we have a canonical map
\[
\eta_r \colon (\Omega_X^n)^{\otimes r} \to \omega _X ^{[r]}.
\] 
Since $\omega _X ^{[r]}$ is an invertible sheaf, 
an ideal sheaf $\mathfrak{n}_{r,X} \subset \mathcal{O}_X$ is uniquely determined by 
${\rm Im}(\eta_r) = \mathfrak{n}_{r,X} \otimes \omega _X ^{[r]}$.
The ideal sheaf $\mathfrak n_{r,X}$ is called the $r$-th \textit{Nash ideal} of $X$.

\item Furthermore, let $f:X \to Y$ be a morphism to a variety $Y$ over $k$. Then
we have a canonical map
\[
\theta_r \colon f^*(\Omega_Y ^n)^{\otimes r} \to \omega _X ^{[r]}, 
\] 
and an ideal sheaf $\mathfrak{n}_{r,f} \subset \mathcal{O}_X$ such that 
${\rm Im}(\theta_r)=\mathfrak{n}_{r,f}\otimes \omega _X ^{[r]}$.
The ideal sheaf $\mathfrak n_{r,f}$ is called the $r$-th \textit{Nash ideal} of $f$.
\end{enumerate}
\end{defi}

\subsection{Jet schemes and arc spaces for $k[t]$-schemes}\label{subsection:ktjet}
Following \cite{DL02}, we extend the definition of the arc spaces of $k$-schemes in Subsection \ref{subsection:jet} to 
the case where $X$ is a $k[t]$-scheme, namely a scheme over $\operatorname{Spec} k[t]$. 

Let $X$ be a scheme of finite type over $\operatorname{Spec} k[t]$. 
Define a contravariant functor $F_{m}: ({\sf Sch}/k) \to ({\sf Sets})$ by 
\[
F_{m}(Y)=\operatorname{Hom} _{k[t]} \left( Y \times_{\operatorname{Spec} k} \operatorname{Spec} k[t]/(t^{m+1}), X \right).
\]
Then, $F_{m}$ is representable by a scheme $X_m$ of finite type over $k$, and the scheme $X_m$ is called the \textit{$m$-th jet scheme} of $X$. 
We shall denote by the same symbols $X_{\infty}$, $\pi_{mn}$, $\psi_m$, $\pi_m$ also for this setting. 
Cylinders and the contact loci $\operatorname{Cont}^{m}(\mathfrak{a})$ and $\operatorname{Cont}^{\ge m}(\mathfrak{a})$ are also defined by the same way for this setting.

\begin{rmk}\label{rmk:bc}
Note that $X_0 \simeq X$ holds if $X$ is a scheme over $k$. 
However, this is not true for $k[t]$-schemes. 
Indeed, if $X = \mathbb{A}^1 _{k[t]}$, then $X_0 \simeq \mathbb{A}^1 _{k}$ holds. 
More generally, $X_m \simeq Y_m$ holds for a $k$-scheme $X$ and $Y=X\times_{\operatorname{Spec} k}\operatorname{Spec} k[t]$.
\end{rmk}

In this paper, we basically treat $k[t]$-schemes with the following conditions: 
\begin{quote}
$(\star)_n$ \ $X$ is a scheme of finite type over $\operatorname{Spec} k[t]$. 
Any irreducible component of $X$ has dimension at least $n+1$. 
Furthermore, any irreducible component dominating $\operatorname{Spec} k[t]$ is exactly $(n+1)$-dimensional. 
\end{quote}
\begin{quote}
$(\star\star)_n$ \ $X$ is a $k[t]$-scheme with condition $(\star)_n$. 
Furthermore, any irreducible components of $X$ dominating $\operatorname{Spec} k[t]$ is reduced outside $t=0$. 
\end{quote}
These categories are suitable for defining the Jacobian ideal and the codimension of cylinders. 

For a $k[t]$-schemes $X$ with the condition $(\star)_n$, we denote by 
$\operatorname{Jac}_{X/k[t]} := \operatorname{Fitt} ^{n} (\Omega _{X/k[t]})$ the \textit{Jacobian ideal} of $X$ over $k[t]$. 
Under the condition $(\star\star)_n$, we see in Subsection \ref{subsection:codim} that the codimension of a cylinder 
is also defined by the same way as in Definition \ref{defi:codim}. 

\begin{rmk}\label{rmk:component}
Let $X$ be a $k[t]$-scheme with the condition $(\star)_n$. 
Let $X^{(1)}, \ldots, X^{(\ell)}$ be the irreducible components of $X$ dominating $\operatorname{Spec} k[t]$. 
Then we have $X_{\infty} = \bigcup _{i =1} ^{\ell} X^{(i)}_{\infty}$. 
Therefore, we can reduce some problems on the arc space of $X$ to that for its irreducible component dominating $\operatorname{Spec} k[t]$. 
However, we can not do such a reduction on problems relating to the order of the Jacobian ideal (cf.\ Example \ref{eg:eg1}). 
\end{rmk}

\begin{eg}\label{eg:eg1}
The $k[t]$-schemes 
\[
X = \operatorname{Spec} \bigl( k[t][x,y,z]/(tx,ty) \bigr), \quad
Y = \operatorname{Spec} \bigl( k[t][x,y,z]/(x,y) \bigr)
\]
satisfy the condition $(\star\star)_1$. 
We have a canonical isomorphism $X_{\infty} \simeq Y_{\infty}$ on the arc spaces, 
but corresponding arcs have different orders of the Jacobian ideals because 
$\operatorname{Jac}_{X/k[t]} = (t^2,tx,ty)/(tx,ty)$ and $\operatorname{Jac}_{Y/k[t]} = (1)$. 
\end{eg}

We also define the order of the Jacobian for a morphism. 

\begin{defi}\label{defi:jac_kt}
Let $X$ and $Y$ be $k[t]$-schemes of finite type, and let $f:X \to Y$ be a morphism over $k[t]$. 
Let $\gamma : \operatorname{Spec} k[[t]] \to X$ be an arc, and let $\gamma ' := f_{\infty} (\gamma)$. 
Let $S$ be the torsion part of $\gamma ^{*} \Omega _{X / k[t]}$. 
Then we define the \textit{order of the Jacobian} of $f$ 
$\operatorname{ord}_{\gamma} (\operatorname{jac}_f)$ at $\gamma$ as the length of the $k[[t]]$-module
\[
\operatorname{Coker} \bigl( \gamma ^{'*} \Omega _{Y / k[t]} \to \gamma ^* \Omega _{X / k[t]} /S \bigr).
\]
In particular, if $\operatorname{ord}_{\gamma} (\operatorname{jac}_f) < \infty$, then 
\[
\operatorname{Coker} \bigl( \gamma ^{'*} \Omega _{Y / k[t]} \to \gamma ^* \Omega _{X / k[t]} /S \bigr) 
\simeq 
\bigoplus _i k[t]/(t^{e_i})
\]
holds as $k[[t]]$-modules with some positive integers $e_i$ satisfying $\sum_i e_i = \operatorname{ord}_{\gamma} (\operatorname{jac}_f)$. 

By abuse of notation, we set
\[
\operatorname{Cont}^{e}(\operatorname{jac}_f):=\{\gamma \in X_\infty \mid \operatorname{ord}_\gamma(\operatorname{jac}_f)= e \}
\]
for $e \in \mathbb Z_{\ge 0}$. We note that it is not clear from the definition that $\operatorname{Cont}^{e}(\operatorname{jac}_f)$ is a cylinder. 
\end{defi}

\begin{rmk}
In some papers, the \textit{Jacobian ideal} $\operatorname{Jac}_f$ of $f$ is defined by 
$\operatorname{Jac}_f := \operatorname{Fitt} ^0 (\Omega _{X/Y})$. 
We note that $\operatorname{ord}_{\gamma} (\operatorname{Jac}_f)$ coincides with the length of the $k[[t]]$-module
\[
\operatorname{Coker} \bigl( \gamma ^{'*} \Omega _{Y / k[t]} \to \gamma ^* \Omega _{X / k[t]} \bigr) 
= \gamma ^* \Omega _{X / Y}. 
\]
Therefore, if $\Omega _{X / k[t]}$ is locally free, 
then $\operatorname{ord}_{\gamma} (\operatorname{Jac}_f) = \operatorname{ord}_{\gamma} (\operatorname{jac}_f)$ holds. 
However the equality does not hold in general (cf.\ Example \ref{eg:eg2}). 
\end{rmk}

\begin{eg}\label{eg:eg2}
Let $R=k[t][x,y,z]/(xy+z^2)$ and let $f:R\to R$ be the homomorphism defined by $f(x)=x^2$, $f(y)=xy$ and $f(z)=xz$.
Let $\gamma:R\to k[[t]]$ be the arc defined by $\gamma(x)=t$ and $\gamma(y)=\gamma(z)=0$.
Then $\gamma ^* \Omega_{R/k[t]} = k[[t]] dx \oplus (k[[t]]/(t))dy \oplus k[[t]]dz$.
Note that  $d(f(x))=2xdx$, $d(f(y))=ydx+xdy$ and $d(f(z))=zdx+xdz$.
Therefore we have $\operatorname{ord}_{\gamma} (\operatorname{Jac}_f) = 3$ and  $\operatorname{ord}_{\gamma} (\operatorname{jac}_f)=2$.
\end{eg}

The additivity holds for the orders of the Jacobian of morphisms. 

\begin{lem}\label{lem:additive}
Let $n$ be a non-negative integer, and let $X$, $Y$ and $Z$ be $k[t]$-schemes with the condition $(\star)_n$. 
Let $f : X \to Y$ and $g: Y \to Z$ be morphisms over $k[t]$. 
Let $\gamma \in X_{\infty}$ be an arc and let $\gamma ' := f_{\infty}(\gamma)$. 
Suppose that 
\[
\operatorname{ord}_{\gamma} (\operatorname{Jac}_{X/k[t]}) < \infty, \quad 
\operatorname{ord}_{\gamma'} (\operatorname{Jac}_{Y/k[t]}) < \infty. 
\]
Then we have 
\[
\operatorname{ord}_{\gamma} (\operatorname{jac}_{g \circ f}) = 
\operatorname{ord}_{\gamma} (\operatorname{jac}_{f}) + 
\operatorname{ord}_{\gamma'} (\operatorname{jac}_{g}). 
\]
\end{lem}
\begin{proof}
Set $\gamma '' := g_{\infty}(\gamma')$. 
Let $S$, $T$ and $U$ be the torsion parts of $\gamma ^* \Omega _{X / k[t]}$, ${\gamma'} ^* \Omega _{Y / k[t]}$ and 
${\gamma''} ^* \Omega _{Z / k[t]}$, respectively. 
Since $\operatorname{ord}_{\gamma} (\operatorname{Jac}_{X/k[t]}) < \infty$, $X$ is smooth over $k[t]$ at $\gamma(\eta)$, where $\eta$ is the generic point of $\operatorname{Spec} k[[t]]$. 
Hence we have 
\[
\gamma ^* \Omega _{X / k[t]} /S \simeq k[[t]]^{\oplus n}. 
\]
For the same reason, we have 
\[
{\gamma '} ^* \Omega _{Y / k[t]} /T \simeq k[[t]]^{\oplus n}. 
\]
If 
\[
\operatorname{ord}_{\gamma} (\operatorname{jac}_{f}) = 
\operatorname{length} \Bigl( \operatorname{Coker} \bigl( {\gamma '} ^{*} \Omega _{Y / k[t]} /T \to {\gamma} ^* \Omega _{X / k[t]} /S \bigr) \Bigr) = \infty
\]
holds, then we have 
\[
\operatorname{ord}_{\gamma} (\operatorname{jac}_{g \circ f}) = 
\operatorname{length} \Bigl( \operatorname{Coker} \bigl( {\gamma ''} ^{*} \Omega _{Z / k[t]} /U \to {\gamma } ^* \Omega _{X / k[t]} /S \bigr) \Bigr) = \infty. 
\]
Otherwise, ${\gamma '} ^{*} \Omega _{Y / k[t]} /T \to {\gamma} ^* \Omega _{X / k[t]} /S$ is injective, then the additivity
\[
\operatorname{ord}_{\gamma} (\operatorname{jac}_{g \circ f}) = 
\operatorname{ord}_{\gamma} (\operatorname{jac}_{f}) + 
\operatorname{ord}_{\gamma'} (\operatorname{jac}_{g})
\]
follows from the additivity of the length of modules. 
\end{proof}

The Nash ideals can also be defined in this setting. 

\begin{defi}\label{defi:nash_kt}
Let $X$ be a normal $k[t]$-variety of relative dimension $n$. Suppose that $X$ is smooth over $k[t]$ outside a closed subset of $X$ of codimension two. 
Then the canonical sheaf $\omega _{X/k[t]}$ is defined (cf.\ \cite[Definition 1.6]{Kol13}). 
Suppose that there exists a positive integer $r$ such that $\omega _{X/k[t]} ^{[r]}$ is an invertible sheaf. 

\begin{enumerate}
\item 
Then we have a canonical map
\[
\eta_r \colon (\Omega_{X/k[t]} ^n)^{\otimes r} \to \omega _{X/k[t]} ^{[r]}. 
\] 
Since $\omega _{X/k[t]} ^{[r]}$ is an invertible sheaf, 
an ideal sheaf $\mathfrak{n}_{r,X} \subset \mathcal{O}_X$ is uniquely determined by 
${\rm Im}(\eta_r) = \mathfrak{n}_{r,X} \otimes \omega _{X/k[t]} ^{[r]}$.
The ideal sheaf $\mathfrak n_{r,X}$ is called the $r$-th \textit{Nash ideal} of $X$.

\item Furthermore, let $f:X \to Y$ be a $k[t]$-morphism from a $k[t]$-scheme $Y$. Then
we have a canonical map
\[
\theta_r \colon f^*(\Omega_{Y/k[t]} ^n)^{\otimes r} \to \omega _{X/k[t]} ^{[r]}, 
\] 
and an ideal sheaf $\mathfrak{n}_{r,f} \subset \mathcal{O}_X$ such that 
${\rm Im}(\theta_r)=\mathfrak{n}_{r,f}\otimes \omega _{X/k[t]} ^{[r]}$.
The ideal sheaf $\mathfrak n_{r,f}$ is called the $r$-th \textit{Nash ideal} of $f$.
\end{enumerate}
\end{defi}

\begin{rmk}\label{rmk:nash_kt}
In this paper, we only use this definition for 
$k[t]$-variety $X'$ of the form $X' = X \times _{\operatorname{Spec} k} \operatorname{Spec} k[t]$, where $X$ is a normal $k$-variety. 
In this case, $\omega _{X'/k[t]}$ is just the pulled back of $\omega _X$ to $X'$. 
Therefore $\mathfrak n_{r,X'} = \mathfrak n_{r,X} \mathcal{O}_{X'}$ holds. 
\end{rmk}

\begin{lem}\label{lem:order}
\begin{enumerate}
\item Let $X$ be a $k[t]$-scheme with the condition $(\star)_n$, 
and let $\gamma \in \operatorname{Cont}^e (\operatorname{Jac}_{X/k[t]})$ be an arc. 
Then 
\[
\gamma ^* \Omega _{X / k[t]} \simeq 
k[[t]]^{\oplus n} \oplus \bigoplus _i k[t]/(t^{e_i})
\]
holds as $k[[t]]$-modules with $\sum_i e_i = e$. 

\item Let $r$ be a positive integer and $f:X \to Y$ be a $k[t]$-morphism which satisfy the assumption of Definition \ref{defi:nash_kt}(2). 
Let $\gamma \in X_{\infty}$ be an arc. 
Suppose $\operatorname{ord}_{\gamma} (\operatorname{Jac}_{X/k[t]}) < \infty$. 
Then 
\[
r \operatorname{ord}_{\gamma} (\operatorname{jac}_f) + \operatorname{ord}_{\gamma} (\mathfrak{n}_{r,X}) 
= \operatorname{ord}_{\gamma} (\mathfrak{n}_{r,f})
\]
holds. 
\end{enumerate}
\end{lem}
\begin{proof}
First we shall prove (1). 
Since $k[[t]]$ is PID, the finitely generated module $\gamma ^* \Omega _{X / k[t]}$ is isomorphic 
to a module of the form $k[[t]]^{\oplus a} \oplus \bigoplus _i k[t]/(t^{e_i})$. 
Since $\operatorname{ord}_{\gamma} (\operatorname{Jac}_{X/k[t]}) = e < \infty$, 
$X$ is smooth over $k[t]$ at $\gamma(\eta)$, where $\eta$ is the generic point of $\operatorname{Spec} k[[t]]$. 
Therefore we have $a = n$. 
Then the assertion follows from the definition of $\operatorname{Jac}_{X/k[t]}$ and 
\[
\bigl( t^{\sum _i e_i} \bigr)
=\operatorname{Fitt}^n \bigl( \gamma ^* \Omega _{X / k[t]} \bigr) 
= \gamma ^* \operatorname{Fitt}^n \bigl(\Omega _{X / k[t]} \bigr)
= \gamma ^* \operatorname{Jac}_{X/k[t]}. 
\]

We shall prove (2). 
Let $S$ be the torsion part of $\gamma ^{*} \Omega _{X / k[t]}$. 
Let $S'$ be the torsion part of $\bigl( \gamma ^* \Omega _{X / k[t]} ^{n} \bigr)^{\otimes r}$. 
Since $\gamma^* \omega_{X/k[t]} ^{[r]}$ is torsion-free, the map 
\[
\eta _r: \left( \gamma ^* \Omega ^n _{X / k[t]} \right)^{\otimes r} \to \gamma^*  \omega_{X/k[t]} ^{[r]}
\]
factors through $\bigl( \gamma ^* \Omega _{X / k[t]} ^{n} \bigr)^{\otimes r} / S'$. 
Hence we have
\[
\operatorname{ord}_{\gamma} (\mathfrak{n}_{r,X}) 
= \operatorname{length} (\operatorname{Coker} \eta_r) 
= \operatorname{length} \Bigl( \operatorname{Coker} \bigl( ( \gamma ^* \Omega _{X / k[t]} ^{n} )^{\otimes r} / S' 
\to \gamma^*  \omega_{X/k[t]} ^{[r]} \bigr) \Bigr). 
\]
In the same way, we have 
\[
\operatorname{ord}_{\gamma} (\mathfrak{n}_{r,f}) 
= \operatorname{length} \Bigl( \operatorname{Coker} \bigl( ( {\gamma '}^* \Omega _{Y / k[t]} ^{n} )^{\otimes r} / T' 
\to \gamma^*  \omega_{X/k[t]} ^{[r]} \bigr) \Bigr), 
\]
where $\gamma'=f_\infty(\gamma) $ and $T'$ is the torsion part of $\bigl( {\gamma '}^* \Omega _{Y / k[t]} ^{n} \bigr)^{\otimes r}$. 

We note that 
\[
\bigl( \gamma ^* \Omega _{X / k[t]} ^{n} \bigr)^{\otimes r} / S' 
\simeq \bigl((\gamma ^* \Omega _{X / k[t]} /S)^{\wedge n} \bigr)^{\otimes r}
\simeq k[[t]], 
\]
and we have 
\[
r \operatorname{ord}_{\gamma} (\operatorname{jac}_f) = 
\operatorname{length} \Bigl( \operatorname{Coker} \bigl( ( {\gamma '}^* \Omega _{Y / k[t]} ^{n} )^{\otimes r} / T' 
\to ( \gamma ^* \Omega _{X / k[t]} ^{n} )^{\otimes r} / S' \bigr) \Bigr)
\]
by Definition \ref{defi:jac_kt}. 
Then the desired formula 
\[
r \operatorname{ord}_{\gamma} (\operatorname{jac}_f) + \operatorname{ord}_{\gamma} (\mathfrak{n}_{r,X}) = \operatorname{ord}_{\gamma} (\mathfrak{n}_{r,f})
\]
follows from the additivity of the length of modules. 
\end{proof}

\begin{rmk}\label{rmk:Sebag}
\begin{enumerate} 
\item In \cite{Seb04}, Sebag extends the theory of motivic integration for $k[t]$-schemes to 
the case of formal schemes over $k[[t]]$ with $k$ a perfect field. 
The reader is also referred to \cite{CLNS} to this theory. 

For a scheme $X$ of finite type over $k[t]$, we can associate the formal scheme $\mathcal{X}$ over $k[[t]]$ by 
\[
\mathcal{X} := \mathop{\varinjlim}\limits_{i \ge 0} \mathcal{X}_i, \quad \text{where\ } \mathcal{X}_i := X \times _{\operatorname{Spec} k[t]} \operatorname{Spec} \bigr( k[[t]]/(t^{i+1}) \bigl )
\]
Then the Greenberg schemes $\operatorname{Gr}_{m}(\mathcal{X})$ and $\operatorname{Gr}(\mathcal{X})$ 
defined in \cite{Seb04} are isomorphic to $X_m$ and $X_{\infty}$ respectively (cf.\ \cite[Ch.4.\ Example 3.3.3]{CLNS}). 
Therefore, the theory of the Greenberg schemes developed in \cite{Seb04} and \cite{CLNS} 
can be applied to the arc spaces $X_{\infty}$ of $k[t]$-schemes $X$. 

When $\mathcal{X}$ is a formal scheme of finite type of relative dimension $d$ over $k[[t]]$, 
then the Jacobian ideal $\operatorname{Jac}_{\mathcal{X}}$  is defined by $\operatorname{Jac}_{\mathcal{X}} 
= \operatorname{Fitt}^d \bigl( \Omega_{\mathcal{X}/k[[t]]} \bigr)$ (\cite[Ch.5.\ Definition 1.3.1]{CLNS}). 
When $\mathcal{X}$ is the associated formal scheme 
with a scheme $X$ of finite type of relative dimension $d$ over $k[t]$, 
this definition is compatible with the definition of $\operatorname{Jac}_{X/k[t]}$ in the sense that
\[
\operatorname{Jac}_{\mathcal{X}} \mathcal{O}_{\mathcal{X}_i} = \operatorname{Jac}_{X/k[t]} \mathcal{O}_{\mathcal{X}_i}
\]
for each $i \ge 0$. 
This follows from the base change properties of the sheaves of differntials (cf.\ \cite[Ch.6.\ Proposition 1.8(a)]{Liu}) and 
the Fitting ideal (cf.\ \cite[Corollary 20.5]{Eis95}). 
For a morphism $h: \mathcal{X} \to \mathcal{Y}$ of formal schemes of finite type over $k[[t]]$, 
and for $\gamma \in \operatorname{Gr}(\mathcal{X})$, 
the order $\operatorname{ord}_t (\operatorname{Jac})_h (\gamma)$ is defined in \cite[Section 5]{Seb04}, 
which is denoted by $\operatorname{ordjac}_h(\gamma)$ in \cite[Ch.5.\ 3.1.1]{CLNS}. 
This corresponds to $\operatorname{ord}_{\gamma} (\operatorname{jac}_h)$ in Definition \ref{defi:jac_kt}
when $h:\mathcal{X} \to \mathcal{Y}$ is induced by a $k[t]$-morphism $h: X \to Y$ of $k[t]$-schemes. 

\item 
Theorems in \cite{Seb04} and \cite{CLNS} often assume that the formal schemes are flat over $k[[t]]$. 
However, many of them can be applied to non-flat formal schemes as well. 
In particular, they can be applied to $k[t]$-schemes with the condition $(\star)_n$, 
as we will see in the next subsection. 

\item 
Let $X$ be a $k[t]$-scheme with the condition $(\star)_n$. 
Let $X'$ be the maximal closed subscheme of $X$ which is flat over $\operatorname{Spec} k[t]$. 
Then we have $X'_{\infty} = X_{\infty}$, 
though the inclusion $X'_m \subset X_m$ is not necessarily equal. 
Let $I_{X'} \subset \mathcal{O}_{X}$ be the defining ideal of $X'$. 
Then there exists a non-negative integer $a$ such that 
$t^{a} I_{X'} = 0$ holds in a neighborhood of $t=0$. 
Then it follows that $\pi_{m+a,m}(X_{m+a}) \subset X'_{m}$. 

Furthermore, 
$\operatorname{ord}_{\gamma}(\operatorname{Jac}_{X'/k[t]}) \not = \operatorname{ord}_{\gamma}(\operatorname{Jac}_{X/k[t]})$
holds for $\gamma \in X_{\infty}$ in general (Example \ref{eg:eg1}). 
However, as we will see below, the difference is bounded. 
Since $t^{a} I_{X'} = 0$ holds, there exists a non-negative integer $a'$ such that 
\[
t^{a '} p^{-1} \bigl( \operatorname{Jac}_{X'/k[t]} \bigr) \subset \operatorname{Jac}_{X/k[t]}
\]
in a neighborhood of $t=0$, where we set $p: \mathcal{O}_{X} \to \mathcal{O}_{X'}$. 
Hence we have 
\[
\operatorname{ord}_{\gamma}(\operatorname{Jac}_{X'/k[t]}) \le 
\operatorname{ord}_{\gamma}(\operatorname{Jac}_{X/k[t]}) \le 
a' + \operatorname{ord}_{\gamma}(\operatorname{Jac}_{X'/k[t]})
\]
for any $\gamma \in X_{\infty}$. 

Let $h: X \to Y$ be a $k[t]$-morphism of $k[t]$-schemes with condition $(\star)_n$. 
Let $Y'$ be the maximal closed subscheme of $Y$ which is flat over $\operatorname{Spec} k[t]$. 
Then $h$ induces a $k[t]$-morphism $h': X' \to Y'$. 
Then it follows that $\operatorname{ord}_{\gamma} (\operatorname{jac}_h) = \operatorname{ord}_{\gamma} (\operatorname{jac}_{h'})$
for any $\gamma \in X_{\infty}$. 
This is because
\[
\bigl( (i \circ \gamma) ^* \Omega_{X/k[t]} \bigr) /S \to \bigl( \gamma ^* \Omega_{X'/k[t]} \bigr)/T
\]
is an isomorphism for the inclusion $i: X' \to X$ and an arc $\gamma : \operatorname{Spec} k[[t]] \to X'$, 
where $S$ and $T$ are the torsion parts of $(\gamma \circ i) ^* \Omega_{X/k[t]}$ and 
$\gamma ^* \Omega_{X'/k[t]}$ respectively. 
\end{enumerate}
\end{rmk}

\subsection{Codimension of cylinders in arc spaces}\label{subsection:codim}
In this subsection, we prove Proposition \ref{prop:EM4.1}, which is necessary for defining the codimension of cylinders in the arc spaces of $k[t]$-schemes. 
Proposition \ref{prop:EM4.1} is a generalization of \cite[Lemma 4.1]{DL99} (cf.\ \cite[Proposition 4.1]{EM09}). 
In their proof, they reduce the problem to that for locally complete intersections. 
The same strategy works in our setting. 

First we state Proposition \ref{prop:EM4.1} for the case of complete intersections with a little generalization. 
Replacing $\operatorname{Jac}_{X}$ by the ideal generated by minors of Jacobian matrix, 
their proof works for non-l.c.i.\ varieties. 
\begin{lem}\label{lem:EM4.1CI}
Let $N$ and $r$ be positive integers with $N \ge r$. 
Let $R = k[t][x_1, \ldots , x_N]$ and let $I = (F_1, \ldots , F_r)$ be the ideal generated by elements $F_1, \ldots , F_r \in R$. 
We denote by $M = \operatorname{Spec} (R/I)$ the $k[t]$-scheme corresponding to $R/I$. 
Let $J \subset R$ be the ideal generated by all the $r$-minors of 
the Jacobian matrix $\left( \partial F_i / \partial x_j \right)_{1 \le i \le r, 1 \le j \le N}$, and
let $\overline{J} = (J+I)/I$. 
For non-negative integers $m$ and $e$ with $m \ge e$, the following hold. 
\begin{enumerate}
\item $\psi _{m} \left( \operatorname{Cont}^{e}(\overline{J}) \right) 
= \pi _{m+e, m} \left( \operatorname{Cont}^{e}(\overline{J})_{m+e} \right)$ holds. 

\item $\pi_{m+1, m}:M_{m+1} \to M_m$ induces a piecewise trivial fibration 
\[
\psi _{m+1} \left( \operatorname{Cont}^e(\overline{J}) \right) \to \psi _{m} \left( \operatorname{Cont}^e(\overline{J}) \right)
\]
with fiber $\mathbb{A}^{N-r}$. 
\end{enumerate}
\end{lem}

\begin{proof}
The first statement follows from \cite[Ch.1.\ Lemma 1.3.3]{CLNS}.

The second statement for locally complete intersection varieties is proved in the proof of Lemma 4.1 in \cite{DL99} (cf.\ \cite[Proposition 4.1]{EM09}).
We can apply the same proof as in \cite[Lemma 4.1]{DL99} by replacing Hensel's Lemma by \cite[Ch.1.\ Lemma 1.3.3]{CLNS}.
\end{proof}

\begin{rmk}
In \cite[Proposition 4.1]{EM09}, the l.c.i.\ cases (or more generally only the pure dimensional cases) are treated. 
We treat the non-l.c.i.\ cases in Lemma \ref{lem:EM4.1CI} 
because we will treat such case in Section \ref{section:mld_hyperquot}
(cf.\ Remark \ref{rmk:nonlci}). 
\end{rmk}

When $X$ is flat over $k[t]$, 
Proposition \ref{prop:EM4.1}(1) below is proved in \cite[Ch.5.\ Proposition 2.3.4]{CLNS}, and 
Proposition \ref{prop:EM4.1}(2) is proved in \cite[Lemma 4.5.4]{Seb04} (cf.\ \cite[Ch.5.\ Theorem 2.3.11]{CLNS}). 
We note here that Proposition \ref{prop:EM4.1}(2) can be reduced to the flat case by Remark \ref{rmk:Sebag}(3). 

\begin{prop}\label{prop:EM4.1}
Let $X$ be a $k[t]$-scheme with the condition $(\star)_n$ in Subsection \ref{subsection:ktjet}. 
Then there exists a positive integer $c$ such that the following hold 
for non-negative integers $m$ and $e$ with $m \ge ce$. 
\begin{enumerate}
\item $\psi _{m} \left( \operatorname{Cont}^{e}(\operatorname{Jac}_{X/k[t]}) \right) 
= \pi _{m+e, m} \left( \operatorname{Cont}^{e}(\operatorname{Jac}_{X/k[t]})_{m+e} \right)$ holds. 

\item $\pi_{m+1, m}:X_{m+1} \to X_m$ induces a piecewise trivial fibration 
\[
\psi _{m+1} \left( \operatorname{Cont}^e(\operatorname{Jac}_{X/k[t]}) \right) \to 
\psi _{m} \left( \operatorname{Cont}^e(\operatorname{Jac}_{X/k[t]}) \right)
\]
with fiber $\mathbb{A}^n$. 
\end{enumerate}
\end{prop}
\begin{proof}
In \cite[Lemma 4.5.4]{Seb04} (cf.\ \cite[Ch.5.\ Theorem 2.3.11]{CLNS}), the assertion (2) is proved for flat formal schemes of finite type of pure relative dimension over $k[[t]]$. 
Therefore, The assertion (2) can be reduced to this result by Remark \ref{rmk:Sebag}(1)-(3). 
On the other hand, it seems that the assertion (1) cannot be easily reduced to the flat case (we can only see that 
$\psi _{m} \left( \operatorname{Cont}^{e}(\operatorname{Jac}_{X/k[t]}) \right) 
= \pi _{m+e+a, m} \left( \operatorname{Cont}^{e}(\operatorname{Jac}_{X/k[t]})_{m+e+a} \right)$ by Remark \ref{rmk:Sebag}(3)). 
However, if we assume the condition $(\star)_n$, it turns out that the proof itself is valid for the non-flat case as well. 
For readers' convenience, we give a proof below, following the argument in \cite[Proposition 4.1]{EM09}. 

Since the assertion is local on $X$, we may assume that $X \subset \mathbb{A}_{k[t]} ^N$ is affine. Set $r := N - n$. 
Let $R := k[t][x_1, \ldots , x_N]$ and let $I_X \subset R$ be the defining ideal of $X$. 
Let $f_1, \ldots , f_d$ be generators of $I_X$. 
For $1 \le i \le d$, we set 
\[
F_i = \sum _{j= 1} ^d a_{ij} f_j
\]
for general $a_{ij} \in k$. 
Then for each subset $\Lambda \subset \{1, \ldots , d \}$ with $\# \Lambda = r = N - n$, 
we denote by $M_{\Lambda} \subset \mathbb{A}_{k[t]} ^N$ the subscheme defined by the ideal 
$I_{M_{\Lambda}} := (F_i \mid i \in \Lambda)$ generated by $F_i$ with $i \in \Lambda$. 
We denote by $J_{\Lambda} \subset R$ the ideal generated by the $r$-minors of 
the Jacobian matrix $\left( \partial F_i / \partial x_j \right)_{i \in \Lambda, 1 \le j \le N}$. 
Set $\overline{J}_{\Lambda} := (J_{\Lambda} + I_{M_{\Lambda}})/I_{M_{\Lambda}}$. 

We note that for $\gamma \in X_{\infty}$ we have 
\[
\operatorname{ord}_{\gamma} ( \operatorname{Jac}_{X/k[t]} ) = \min _{\Lambda} \operatorname{ord}_{\gamma} ( J_{\Lambda} ). 
\]
Hence 
\[
U_{\Lambda} := \bigl \{ \gamma \in \operatorname{Cont}^e ( \operatorname{Jac}_{X/k[t]} ) \ \big| \ 
\operatorname{ord}_{\gamma} ( J_{\Lambda} ) = e \bigr \}
\]
is an open subset of $\operatorname{Cont}^e(\operatorname{Jac}_{X/k[t]})$ 
satisfying $\operatorname{Cont}^e(\operatorname{Jac}_{X/k[t]}) = \bigcup _{\Lambda} U_{\Lambda}$. 
Since $X$ is a closed subscheme of $M_{\Lambda}$, we may identify the arc space $X_{\infty}$ 
with a closed subset of $(M_{\Lambda})_{\infty}$. 
Under this identification, we have 
$U_{\Lambda} \subset \operatorname{Cont}^e(\overline{J}_{\Lambda})$. 

Then we claim the following (cf.\ \cite[Lemma 4.2]{EM09}). 
\begin{claim}\label{claim:M}
There exists a positive integer $c_{\Lambda}$ such that 
the following condition holds for any non-negative integers $m$ and $e$ satisfying $m \ge c_{\Lambda} e$. 
\begin{itemize}
\item 
If $\gamma \in \operatorname{Cont}^e(\overline{J}_{\Lambda}) \subset (M_{\Lambda}) _{\infty}$ satisfies 
$\psi _{m} (\gamma) \in X_{m}$, then $\gamma \in X_{\infty}$. 
\end{itemize}
\end{claim}
\begin{proof}
Let $I_{X'_{\Lambda}} := (I_{M_{\Lambda}} : I_X)$ and let $X' _{\Lambda} \subset \mathbb{A}_{k[t]} ^N$ be the corresponding subscheme. 
For a prime ideal $\mathfrak{p}$ of $R$, 
we note that $I_{M_{\Lambda}} \subset \mathfrak{p}$ and $I_X \not \subset \mathfrak{p}$ imply 
$I_{X'_{\Lambda}} = (I_{M_{\Lambda}} : I_X) \subset \mathfrak{p}$. 
Therefore set-theoretically $X'_{\Lambda}$ is the union of the irreducible components of $M_{\Lambda}$ which are not contained in $X$. 
Hence we have $(M_{\Lambda})_{\infty} = X_{\infty} \cup (X'_{\Lambda})_{\infty}$. 
Since any irreducible component of $X$ has dimension at least $n+1$ and
$a_{ij}$ are  general elements  of $k$, for any irreducible component $X_0$ of $X$,
there exists an irreducible component of $M_0$ of  $M_{\Lambda}$ with $X_0=M_0$.
Therefore if $M_{\Lambda}$ is smooth at a point $x\in X$, then $X$ is smooth at $x$ and $\mathcal{O}_{X,x}=\mathcal{O}_{M_{\Lambda},x}$.
Hence, if  $(R/I_{M_{\Lambda}})_{\mathfrak{q}}$ is regular local ring for a prime ideal 
$\mathfrak{q}$ of $R$ with $I_X \subset \mathfrak{q}$, then we have
\[
(I_{X'_{\Lambda}})_{\mathfrak{q}} = (I_{M_{\Lambda}} : I_X)_{\mathfrak{q}} 
= ((I_{M_{\Lambda}})_{\mathfrak{q}} : (I_X)_\mathfrak{q})=R_{\mathfrak{q}}.
\]
This implies that $X'_{\Lambda}\cap ((M_{\Lambda})_{\rm{reg}}\cap X)=\emptyset$.
Hence $M_{\Lambda}$ is singular at any point $x \in X \cap X_{\Lambda}'$. 
Here we claim that 
\begin{itemize}
\item[($\heartsuit$)] $J_{\Lambda} \subset \sqrt{I_X + I_{X'_{\Lambda}}}$ holds. 
\end{itemize}
Let $J_{\Lambda}'$ be the ideal generated by the $r$-minors of 
the Jacobian matrix with respect to $I_{M_{\Lambda}} = (F_i \mid i \in \Lambda)$ and 
derivations $\partial x_j$'s and $\partial t$. 
Then by the definition of $J_{\Lambda}$, we have $J_{\Lambda} \subset J_{\Lambda}'$. 
Let $\mathfrak{p}$ be a prime ideal satisfying $I_X + I_{X'_{\Lambda}} \subset \mathfrak{p}$. 
Since $\operatorname{ht}(I_{M_{\Lambda}} R_{\mathfrak{p}}) \le r$ and the ring 
$R_{\mathfrak{p}}/I_{M_{\Lambda}} R_{\mathfrak{p}}$ is not regular, 
it follows by the Jacobian criterion (cf.\ \cite[Theorem 30.4]{Mat89}) that
$J_{\Lambda}' \subset \mathfrak{p}$, 
which proves the claim ($\heartsuit$).

By ($\heartsuit$), 
\[
J_{\Lambda} ^{c_{\Lambda}} \subset I_X + I_{X'_{\Lambda}}
\]
holds for some $c_{\Lambda}$. 
Suppose $\gamma \in \operatorname{Cont}^e(\overline{J}_{\Lambda}) \subset (M_{\Lambda}) _{\infty}$. 
Since 
\[
\operatorname{ord}_{\gamma} ( J_{\Lambda}^{c_{\Lambda}} ) = c_{\Lambda} e < m+1 \le 
\operatorname{ord}_{\gamma} (I_X), 
\]
we have $\operatorname{ord}_{\gamma} \bigl( I_{X'_{\Lambda}} \bigr) \le c_{\Lambda} e$. 
Hence $\gamma \not \in (X'_{\Lambda})_{\infty}$ and it shows $\gamma \in X_{\infty}$. We complete the proof of the claim. 
\end{proof}
We set $c = \max _{\Lambda} c_{\Lambda}$.
Then the assertions (1) and (2) for $X$ follow from the assertions of Lemma \ref{lem:EM4.1CI} for $M_{\Lambda}$ by Claim \ref{claim:M}. 
\end{proof}

We define cylinders in the arc spaces of $k[t]$-schemes and define their codimension. 
For a $k[t]$-scheme $X$, a subset $C \subset X_{\infty}$ is called a \textit{cylinder} if $C = \psi_{m} ^{-1}(S)$ holds for some $m \ge 0$ and 
a constructible subset $S \subset X_m$. 
\begin{defi}\label{defi:codim_kt}
Let $X$ be a $k[t]$-scheme with the condition $(\star \star)_n$. 
Let $C \subset X_{\infty}$ be a cylinder. 
\begin{enumerate}
\item 
Assume that $C \subset \operatorname{Cont}^e (\operatorname{Jac}_{X/k[t]})$ holds for some $e \in \mathbb{Z}_{\ge 0}$.
Then we define the codimension of $C$ in $X_\infty$ as
\[
\operatorname{codim}(C) := (m+1) n - \operatorname{dim}(\psi_m (C))
\]
for any sufficiently large $m$. This definition is well-defined by Proposition \ref{prop:EM4.1}.

\item 
In general, we define the codimension of $C$ in $X_\infty$ as follows:
\[
\operatorname{codim}(C) := \min_{e \in \mathbb{Z}_{\ge 0}} {\operatorname{codim}(C \cap \operatorname{Cont}^e (\operatorname{Jac}_{X/k[t]}))}. 
\]
\end{enumerate}
\end{defi}

\begin{rmk}
The codimension is well-defined also for $X$ with $(\star)_n$. 
However, in this case, we may have $X_{\infty} \cap \operatorname{Cont}^{e} (\operatorname{Jac}_{X/k[t]}) = \emptyset$ for any $e \ge 0$, 
and $\operatorname{codim}(C) = \infty$ may hold for any cylinder $C$. 
Therefore, we assume $(\star \star)_n$ when we discuss the codimension of cylinder. 
\end{rmk}

\begin{defi}\label{def:thin}
Let $X$ be a $k[t]$-scheme with the condition $(\star \star)_n$. 
A subset $A \subset X_{\infty}$ is called \textit{thin} 
if $A \subset Z_{\infty}$ holds for some closed subscheme $Z$ of $X$ with the condition $(\star \star)_{\ell}$ for some $\ell \le n-1$.
\end{defi}

\begin{rmk}
The arc space $X_{\infty}$ is always not a thin set of $X_{\infty}$ for a $k$-variety $X$. 
However, $X_{\infty}$ can be a thin set of $X_{\infty}$ for a $k[t]$-scheme $X$ even if we assume the condition $(\star \star)_n$. 
See the example in Remark \ref{rmk:LC}. 
\end{rmk}

\begin{lem}[{cf.\ \cite[Ch.6.\ Proposition 2.4.6]{CLNS}}]\label{lem:thin}
Let $X$ be a $k[t]$-scheme with the condition $(\star)_n$, and let $C \subset X_{\infty}$ be a cylinder. 
If $C$ is thin, then $C \cap \operatorname{Cont}^{e} (\operatorname{Jac}_{X/k[t]}) = \emptyset$ holds for any $e \ge 0$. 
\end{lem}
\begin{proof}
This follows from \cite[Ch.6.\ Proposition 2.4.6]{CLNS} and Remark \ref{rmk:Sebag}. 
\end{proof}

\begin{prop}[{\cite[Lemma 4.3.9]{Seb04}}]\label{prop:const}
Let $X$ be a scheme of finite type over $k[t]$, and let $C$ be a cylinder in $X_{\infty}$. 
Then its image $\psi _m (C) \subset X_m$ is a constructible subset for any $m \ge 0$. 
\end{prop}
\begin{proof}
This follows from \cite[Lemma 4.3.9]{Seb04} (cf.\ \cite[Ch.5.\ Corollary 1.5.7]{CLNS}) and Remark \ref{rmk:Sebag}. 
\end{proof}

\begin{prop}[{cf.\ \cite[Th\'{e}or\`{e}me 6.3.5]{Seb04}}]\label{prop:negligible}
Let $X$ be a $k[t]$-scheme with the condition $(\star \star)_n$, 
and let $C$ be a cylinder in $X_{\infty}$. 
Let $\{ C_{\lambda} \} _{\lambda \in \Lambda}$ be a set of countably many disjoint subcylinders $C_{\lambda} \subset C$. 
If $C \setminus (\bigsqcup _{\lambda \in \Lambda} C_{\lambda}) \subset X_{\infty}$ is a thin set, 
then it follows that 
\[
\operatorname{codim}(C) = \min _{\lambda \in \Lambda} \operatorname{codim}(C_{\lambda}). 
\]
\end{prop}
\begin{proof}
This follows from \cite[Ch.6.\ Lemma 3.4.1]{CLNS} and \cite[Ch.6.\ Example 3.5.2]{CLNS}. 
\end{proof}

\begin{lem}\label{lem:smooth}
Let $X$ be a variety over $\operatorname{Spec} k[t]$ which dominates $\operatorname{Spec} k[t]$ and has relative dimension $n$. 
Suppose that $X$ is smooth over $k$. 
Then there exists a non-negative integer $\ell$ such that the following hold. 
\begin{enumerate}
\item $\operatorname{ord}_{\gamma} (\operatorname{Jac}_{X/k[t]}) \le \ell$ holds for any arc $\gamma \in X_{\infty}$. 
\item $\psi_m (X_{\infty}) = \pi _{m + \ell, m} (X_{m + \ell})$ holds for any $m \ge \ell$. 
\item For any $m \ge \ell$, $\pi _{m+1,m}$ induces a piecewise trivial fibration $\psi_{m+1} (X_{\infty}) \to \psi_{m} (X_{\infty})$ with fiber $\mathbb{A}^n$. 
\end{enumerate}
\end{lem}
\begin{proof}
Since $X$ is smooth over $k$, by the generic smoothness, $X$ is smooth over $\operatorname{Spec} k[t]$ outside finite closed points. 
Therefore we have an inclusion of ideals $(t^{\ell}) \subset \operatorname{Jac}_{X/k[t]}$ in a neighborhood of $t=0$ for some $\ell \ge 0$. 
Hence $\operatorname{ord}_{\gamma} (\operatorname{Jac}_{X/k[t]}) \le \ell$ holds for any arc $\gamma \in X_{\infty}$. 
Then the assertions (2) and (3) follow from Lemma \ref{lem:EM4.1CI}. 
\end{proof}

\begin{lem}\label{lem:thin2}
Let $f: Y \to X$ be a proper birational $k[t]$-morphism of $k[t]$-varieties $X$ and $Y$. 
Suppose that $Y$ is smooth over $k$. Let $C \subset X_{\infty}$ be a cylinder. 
If $C$ is a thin set of $X_{\infty}$, then $f_{\infty}^{-1}(C) = \emptyset$. 
\end{lem}
\begin{proof}
We may assume that $X$ dominates $\operatorname{Spec} k[t]$. Let $n$ be the relative dimension of $X$. 
Since $C$ is a thin set, there exists a closed subset $Z \subsetneq X$ such that $C \subset Z_{\infty}$. 
Set $Z' := f^{-1} (Z)$. Then we have $f_{\infty}^{-1}(C) \subset Z' _{\infty}$ (cf.\ Lemma \ref{lem:closed}(2)).
Since dominant components of $Z'$ have relative dimension at most $n-1$, 
the cylinder $f_{\infty}^{-1}(C)$ is also a thin set. 
By Lemma \ref{lem:thin} and Lemma \ref{lem:smooth}(1), we have $f_{\infty}^{-1}(C) = \emptyset$. 
\end{proof}

\begin{lem}\label{lem:closed}
\begin{enumerate}
\item Let $Z \subset X$ be a closed subscheme of a $k[t]$-scheme $X$ of finite type. 
Then the induced map $f_{\infty}: Z_{\infty} \to X_{\infty}$ is a closed immersion. 
\item Moreover, for a $k[t]$-morphism $f: Y \to X$, it follows that 
$\left( f^{-1} (Z) \right)_{\infty} \simeq f^{-1}_{\infty} (Z_{\infty})$. 
\end{enumerate}
\end{lem}
\begin{proof}
The assertions follow from the same argument for $k$-varieties (cf.\ \cite[Remarks 2.7, 2.8]{EM09}). 
\end{proof}

\subsection{Fundamental properties of the arc spaces of $k[t]$-schemes}\label{subsection:fund}
In this subsection, we prove Proposition \ref{prop:EM6.2}, which is a generalization of \cite[Lemma 1.17]{DL02}
to $k[t]$-scheme with the condition $(\star \star)_n$. 
Actually in \cite[Remark 1.19]{DL02}, 
it is mentioned that \cite[Lemma 1.17]{DL02} can be generalized to 
separated reduced schemes of finite type over $k[t]$. 
In \cite[Lemma 10.20]{Yas}, Yasuda proves Proposition \ref{prop:EM6.2} in more general setting. 
For readers' convenience, we give a proof of Proposition \ref{prop:EM6.2} following the argument in \cite{EM09}.

\begin{prop}\label{prop:EM4.4}
Let $X$ be a $k[t]$-scheme with the condition $(\star)_n$. 
Let $p,m$ be non-negative integers with $2p+1 \ge m \ge p$. 

\begin{enumerate}
\item Let $\gamma \in X_p$ with $\pi _{m,p}^{-1} (\gamma) \not = \emptyset$. 
Then scheme-theoretically we have 
\[
\pi _{m,p}^{-1} (\gamma) \simeq \operatorname{Hom}_{k[t]/(t^{p+1})} \left(\gamma ^* \Omega _{X/k[t]}, (t^{p+1})/(t^{m+1}) \right). 
\]

\item Let $\gamma' \in X_{\infty}$ and let $e := \operatorname{ord}_{\gamma '} (\operatorname{Jac}_{X/k[t]})$. 
Let $c$ be a positive integer appearing in Proposition \ref{prop:EM4.1}. 
Let $T$ be the torsion part of $\gamma ^{'*} \Omega_X$. 
Suppose that $2p+1-e \ge m \ge ce$ and $p \ge e$. 
For $\gamma = \psi _p (\gamma ')$, it follows that 
\begin{align*}
&\pi _{m,p}^{-1}  (\gamma)  \cap \psi _m \left( \operatorname{Cont}^e  (\operatorname{Jac}_{X/k[t]}) \right)  =
\pi_{m+e, m} (\pi ^{-1}_{m+e, p} (\gamma)) \\ 
&\simeq 
\operatorname{Hom}_{k[t]/(t^{p+1})} 
\left(   (\gamma ^{'*} \Omega _{X/k[t]}) /T  \otimes _{k[[t]]} k[t]/(t^{p+1}) , 
(t^{p+1})/(t^{m+1}) \right). 
\end{align*}
\end{enumerate}
\end{prop}
\begin{proof}
We shall prove (1). 
We may assume that $X = \operatorname{Spec} A$. Let $\gamma ^* : A \to k[t]/(t^{p+1})$ be the corresponding $k[t]$-ring homomorphism to $\gamma$. 
Suppose that $\alpha \in \pi _{m,p}^{-1} (\gamma)$ and $\alpha ^* : A \to k[t]/(t^{m+1})$ is the corresponding $k[t]$-ring homomorphism to $\alpha$. 
Then we have 
\[
\pi _{m,p}^{-1} (\gamma) \simeq \operatorname{Der}_{k[t]} \left( A, (t^{p+1})/(t^{m+1}) \right); \quad  \beta \mapsto \beta^* - \alpha^*, 
\]
here $(t^{p+1})/(t^{m+1})$ has an $A$-module structure via $\gamma ^*$ (cf.\ \cite[Proposition 4.4]{EM09}). 
Then the assertion follows from the following isomorphisms. 
\begin{align*}
\operatorname{Der}_{k[t]} \left( A, (t^{p+1})/(t^{m+1}) \right) 
& \simeq 
\operatorname{Hom}_{A} \left( \Omega _{A/k[t]}, (t^{p+1})/(t^{m+1}) \right) \\
& \simeq
\operatorname{Hom}_{k[t]/(t^{p+1})} \left(\gamma ^* \Omega _{X/k[t]}, (t^{p+1})/(t^{m+1}) \right). 
\end{align*}

We shall prove (2). 
Note that 
\begin{align*}
\gamma ^* \Omega _{X/k[t]} 
&= \gamma ^{'*} \Omega _{X/k[t]} \otimes _{k[[t]]} k[t]/(t^{p+1}) \\
&\simeq \left( (\gamma ^{'*} \Omega _{X/k[t]})/T \otimes _{k[[t]]} k[t]/(t^{p+1}) \right) \oplus \left( T \otimes _{k[[t]]} k[t]/(t^{p+1}) \right). 
\end{align*}
Since $T$ is the form of $\bigoplus _i k[t]/(t^{e_i})$ with $\sum _i e_i = e$, especially $e_i \le e \le p$, 
it follows that $T \otimes _{k[[t]]} k[t]/(t^{p+1}) \simeq \bigoplus _i k[t]/(t^{e_i})$. 
Note also that $\lambda : \gamma ^* \Omega _{X/k[t]} \to (t^{p+1})/(t^{m+1})$ lifts to 
$\gamma ^* \Omega _{X/k[t]} \to (t^{p+1})/(t^{m+e+1})$ if and only if 
$\lambda (T \otimes k[t]/(t^{p+1})) = 0$ holds (Lemma \ref{lem:keisan}(2)). 
This equivalence and (1) show 
\begin{align*}
& \pi_{m+e, m} (\pi ^{-1}_{m+e, p} (\gamma)) \\
& \simeq \operatorname{Im}\left( \begin{array}{l} 
\operatorname{Hom}_{k[t]/(t^{p+1})} \left(\gamma ^* \Omega _{X/k[t]}, (t^{p+1})/(t^{m+e+1}) \right) \\
\hspace{10mm} \longrightarrow \operatorname{Hom}_{k[t]/(t^{p+1})} \left(\gamma ^* \Omega _{X/k[t]}, (t^{p+1})/(t^{m+1}) \right)  
\end{array} \right) \\
& \simeq \operatorname{Hom}_{k[t]/(t^{p+1})} 
\left(  (\gamma ^{'*}  \Omega _{X/k[t]}) /T   \otimes _{k[[t]]} k[t]/(t^{p+1}), (t^{p+1})/(t^{m+1}) \right). 
\end{align*}
Since $\operatorname{ord}_{\gamma} (\operatorname{Jac}_{X/k[t]}) = e$, it follows from Proposition \ref{prop:EM4.1}(1) that 
\[
\pi_{m+e, m} (\pi ^{-1}_{m+e, p} (\gamma)) =
\pi _{m,p}^{-1} (\gamma) \cap \psi _m (\operatorname{Cont}^e (\operatorname{Jac}_{X/k[t]})). 
\]
We complete the proof. 
\end{proof}

\begin{lem}\label{lem:keisan}
Let $m,p,\ell$ be non-negative integers with $m \ge p$. Then the following hold. 
\begin{enumerate}
\item $\operatorname{Hom}_{k[t]} \left( k[t]/(t^{\ell}), (t^{p+1})/(t^{m+1}) \right)$ is isomorphic to $\mathbb{A}^{\ell}$ 
if $\ell \le m - p$; otherwise, this is isomorphic to $\mathbb{A}^{m-p}$. 

\item If $\ell \le e$, only the zero map $k[t]/(t^{\ell}) \to (t^{p+1})/(t^{m+1})$ can lift to $k[t]/(t^{\ell}) \to (t^{p+1})/(t^{m+e+1})$. 
\end{enumerate}
\end{lem}
\begin{proof}
The proof is straightforward. 
\end{proof}

For an arc $\delta$ we  denote by $\delta_m$ its image in the space of the $m$-th jets.

\begin{lem}\label{lem:EM6.3}
Let $X$ and $Y$ be $k[t]$-schemes with the condition $(\star)_n$, and let $f: Y \to X$ be a morphism over $k[t]$. 
Let $e, e', e'', q \in \mathbb{Z}_{\ge 0}$. 
Let $c_X$ and $c_Y$ be positive integers for $X$ and $Y$ appearing in Proposition \ref{prop:EM4.1}. 
Suppose  $\max \{ e+e', e+e'', c_X e', c_Y e'' \}\le q - e$.
Let $\alpha \in \operatorname{Cont}^{e'}(\operatorname{Jac}_{X/k[t]})$ and 
$\beta \in \operatorname{Cont}^{e''}(\operatorname{Jac}_{Y/k[t]})$ with $f_q(\beta_q)=\alpha_q$ and $\operatorname{ord}_{\beta}(\operatorname{jac}_f) = e$.
Then there is $\delta \in \operatorname{Cont}^{e''}(\operatorname{Jac}_{Y/k[t]})$ 
with $f_{q+1}(\delta_{q+1})=\alpha_{q+1}$ such that $\beta_{q-e}=\delta_{q-e}$ and $\operatorname{ord}_{\delta}(\operatorname{jac}_f) = e$.
\end{lem}

\begin{proof}
Let  $S$ and $T$ be the torsion parts of $\alpha^*\Omega_{X/k[t]}$ and $\beta^*\Omega_{Y/k[t]}$, respectively.
By Lemma \ref{prop:EM4.4}(2),
 we have
\begin{align*}
& \pi_{q+1,q-e}^{-1}  (\alpha_{q-e})\cap \psi_{q+1}(\operatorname{Cont}^{e'}(\operatorname{Jac}_{X/k[t]})) \\
& \simeq 
\operatorname{Hom}_{k[t]/(t^{q-e+1})}  \bigl( ( \alpha^*\Omega_{X/k[t]})/S \otimes_{k[[t]]} k[t]/(t^{q-e+1}),(t^{q-e+1})/(t^{q+2}) \bigr)
\end{align*}
and
\begin{align*}
&\pi_{q+1,q-e}^{-1}(\beta_{q-e})\cap\psi_{q+1}(\operatorname{Cont}^{e''}(\operatorname{Jac}_{Y/k[t]})) \\
&\simeq 
\operatorname{Hom}_{k[t]/(t^{q-e+1})}\bigl( (\beta^*\Omega_{Y/k[t]})/T \otimes_{k[[t]]} k[t]/(t^{q-e+1}),(t^{q-e+1})/(t^{q+2}) \bigr).
\end{align*}
We may assume that $\beta _{q+1}$ corresponds to the zero map via this isomorphisms. 
Let 
\[
w:\alpha^*\Omega_{X/k[t]}/S\otimes_{k[[t]]} k[t]/(t^{q-e+1})\to(t^{q-e+1})/(t^{q+2})
\] 
be the morphism corresponding to $\alpha_{q+1}$ via this isomorphism. 
Then it is sufficient to show the existence of 
\[
u:\beta^*\Omega_{Y/k[t]}/T\otimes_{k[[t]]} k[t]/(t^{q-e+1})\to (t^{q-e+1})/(t^{q+2})
\]
such that $u \circ h_{q-e}=w$, where
\[
h_{q-e}:\alpha^*\Omega_{X/k[t]}/S\otimes_{k[[t]]} k[t]/(t^{q-e+1})\to\beta^*\Omega_{Y/k[t]}/T\otimes_{k[[t]]} k[t]/(t^{q-e+1})
\] 
is the natural morphism.

By definition of $\operatorname{ord}_{\beta}(\operatorname{jac}_f)$, we have 
\[
\operatorname{Coker}({\beta '}^*\Omega_{X/k[t]} \to \beta^*\Omega_{Y/k[t]}/T) \simeq
k[t]/(t^{a_1})\oplus\cdots\oplus k[t]/(t^{a_n})
\] 
with $a_i \ge 0$ and $\sum_i a_i=e$, where $\beta' := f_{\infty} (\beta)$. 
Since $\beta' _{q-e} = \alpha _{q-e}$ and $a_i \le q-e+1$, we have 
\[
\operatorname{Coker} (h_{q-e}) \simeq k[t]/(t^{a_1})\oplus\cdots\oplus k[t]/(t^{a_n}). 
\]
Hence, we can regard $h_{q-e}$ as the morphism given by the diagonal matrix with entries $t^{a_1},\ldots,t^{a_n}$. 
Furthermore, since $\alpha_q=f_q(\beta_q)$, it follows that $\operatorname{Im} (w) \subset (t^{q+1})/(t^{q+2})$. 
Since $q \ge e + a_i$ holds for each $i$, we can find a desired $u$. 
\end{proof}

\begin{lem}\label{lem:EM6.2'}
Let $X$ and $Y$ be $k[t]$-schemes with the condition $(\star)_n$ and let $f: Y\to X$ be a morphism over $k[t]$.
Let $e,e',e'', m \in \mathbb{Z}_{\ge 0}$. 
Let $c_X$ and $c_Y$ be  positive integers for $X$ and $Y$ appearing in Proposition \ref{prop:EM4.1}. 
Suppose $\max \{e+e', e+e'', c_X e', c_Y e'' \} \le m-e$.
Let  $\alpha \in \operatorname{Cont}^{e'}(\operatorname{Jac}_{X/k[t]})$ and 
$\beta \in \operatorname{Cont}^{e''}(\operatorname{Jac}_{Y/k[t]})$ with $f_m(\beta_m)=\alpha_m$ and $\operatorname{ord}_{\beta} (\operatorname{jac}_f) = e$.
Then there is $\delta \in \operatorname{Cont}^{e''}(\operatorname{Jac}_{Y/k[t]})$ 
with $\beta_{m-e}=\delta_{m-e}$ such that $f_{\infty}(\delta)=\alpha$ and $\operatorname{ord}_{\delta} (\operatorname{jac}_f) = e$.
\end{lem}

\begin{proof}
By Lemma \ref{lem:EM6.3},
we can construct recursively $\delta^{(q)} \in \operatorname{Cont}^{e''}(\operatorname{Jac}_{Y/k[t]})$ for $q\ge m$ such that 
$\delta^{(m)}=\beta$,  $\delta_{q-e}^{(q+1)}=\delta_{q-e}^{(q)}$  and
$f_q(\delta_q^{(q)})=\alpha_q$ for every $q \ge m$.
The sequence $\bigl( \delta_{q-e}^{(q)} \bigr) _{q\in\mathbb Z_{\ge m}}$ defines an element 
$\delta \in \operatorname{Cont}^{e''}(\operatorname{Jac}_{Y/k[t]})$ such that $\delta_{q-e}=\delta_{q-e} ^{(q)}$ for every $q\ge m$.
By the construction of $\delta$, it follows that $\beta_{m-e}=\delta_{m-e}$ and $f_\infty(\delta)=\alpha$.
\end{proof}

When $Y$ is smooth over $k[t]$, Proposition \ref{prop:EM6.2} below is proved in \cite[Lemma 7.1.3]{Seb04} (cf.\ \cite[Ch.5.\ Theorem 3.2.2]{CLNS}). 
In \cite[Lemma 10.19, 10.20]{Yas}, Yasuda proves Proposition \ref{prop:EM6.2} in more general setting (for formal Deligne-Mumford stacks of arbitrary characteristic). 
\begin{prop}[{cf.\ \cite[Lemma 1.17, Remark 1.19]{DL02}}]\label{prop:EM6.2}
Let $X$ and $Y$ be $k[t]$-schemes with the condition $(\star \star)_n$ and let $f: Y\to X$ be a morphism over $k[t]$.
Let $e,e',e''\in\mathbb Z_{\ge 0}$.
Let $B$ be a cylinder of $Y_{\infty}$ and let $A = f_\infty(B)$.
Assume that 
\[
B \subset \operatorname{Cont}^{e''}(\operatorname{Jac}_{Y/k[t]}) \cap \operatorname{Cont}^{e}(\operatorname{jac}_f), \quad 
A \subset \operatorname{Cont}^{e'}(\operatorname{Jac}_{X/k[t]}). 
\]
Then $A$ is a cylinder of $X_{\infty}$. 
Moreover, if $f_{\infty}|_{B}$ is injective, then it follows that
\[
\operatorname{codim}(B)+e=\operatorname{codim}(A).
\]
\end{prop}

\begin{proof}
The second statement is obtained by specializing \cite[Lemma 10.20]{Yas} to the case where $\Phi = \Psi = k$, 
and $\mathcal{Y}$ and $\mathcal{X}$ are the formal schemes over $k[[t]]$ associated to $Y$ and $X$ respectively (cf.\ Remark \ref{rmk:Sebag}). 
For readers' convenience, we give a proof in our setting below.

First, we prove that $A$ is a cylinder. 
Let $B_m \subset Y_m$ be a constructible subset such that $B=\psi_m^{-1}(B_m)$. 
By Proposition \ref{prop:const}, we may assume that $B_m = \psi _m (B)$. 
Furthermore, we may assume that $m$ is sufficiently large, and hence $B = \psi^{-1} _{m-e} (\pi_{m,m-e}(B_m))$ also holds.
It is enough to show that $A= \psi_m^{-1}(A_m)$ for $A_m = f_m(B_m)$. 
The inclusion $A \subset \psi_m^{-1}(A_m)$ is obvious. We shall see the opposite inclusion. 
Suppose that $\alpha \in X_{\infty}$ satisfies $\psi _m (\alpha) \in A_m$. 
Then by the definition of $A_m$, there exists $\beta \in Y_{\infty}$ such that its image in $X_m$ coincides with $\alpha _m$. 
Therefore by Lemma \ref{lem:EM6.2'}, there exists $\gamma \in Y_{\infty}$ such that $f_{\infty}(\gamma) = \alpha$ and 
$\psi _{m-e} (\gamma) = \psi _{m-e} (\beta)$. Since $\gamma \in \psi _{m-e} ^{-1} (\pi_{m,m-e} (\beta_{m})) \subset B$, 
it follows that $\alpha \in f_{\infty}(B)=A$. Therefore $A$ is a cylinder.

Next we shall prove that $\operatorname{codim}(B)+e=\operatorname{codim}(A)$. 
For this, it is sufficient to show that $\dim (f^{-1}_m (\alpha _m) \cap B_m) = e$ for each $\alpha _m \in A_m$. 
Let $\alpha \in A$ be a lift of $\alpha _m$ and let $\beta \in B$ be an arc satisfying $f_{\infty}(\beta) = \alpha$. 

We claim that 
\[
\pi_{m,m-e} \bigl(f^{-1}_m(\alpha _m) \cap B_m \bigr) = \{ \beta _{m-e} \}. 
\]
Take $\beta' _{m} \in f^{-1}_m(\alpha _m) \cap B_m$. 
Then by Lemma \ref{lem:EM6.2'}, there exists $\gamma \in B$ such that $\gamma _{m-e} = \pi_{m,m-e}(\beta' _m)$ and 
$f_{\infty} (\gamma) = \alpha$. Since $f_{\infty}|_{B}$ is injective, it follows that $\beta = \gamma$. 
Therefore, $\pi_{m,m-e}(\beta' _m) = \beta _{m-e}$. 

Hence we have 
\begin{align*}
f^{-1}_m(\alpha _m) \cap B_m 
&= f^{-1}_m(\alpha _m) \cap \pi^{-1}_{m,m-e}(\beta _{m-e}) \\
&\subset \pi^{-1}_{m,m-e}(\beta _{m-e}) \cap \psi _m \bigl( \operatorname{Cont}^{e''}(\operatorname{Jac}_{Y/k[t]}) \bigr). 
\end{align*}
Therefore by Proposition \ref{prop:EM4.4}, $f^{-1}_m(\alpha _m) \cap B_m$ is isomorphic to the kernel of 
\begin{align*}
&\operatorname{Hom}_{k[t]/(t^{m-e+1})}\bigl( (\beta^*\Omega_{Y/k[t]})/T\otimes_{k[[t]]} k[t]/(t^{m-e+1}),(t^{m-e+1})/(t^{m+1}) \bigr) \\
&\to \operatorname{Hom}_{k[t]/(t^{m-e+1})}\bigl( (\alpha^*\Omega_{X/k[t]})/S\otimes_{k[[t]]} k[t]/(t^{m-e+1}),(t^{m-e+1})/(t^{m+1}) \bigr), 
\end{align*}
where $S$ and $T$ be the torsion parts of $\alpha^*\Omega_{X/k[t]}$ and $\beta^*\Omega_{Y/k[t]}$, respectively. 
By the definition of $\operatorname{ord}_{\beta} (\operatorname{jac}_f)$, this is isomorphic to 
\[
\operatorname{Hom}_{k[t]/(t^{m-e+1})}\Bigl( \bigoplus _{i} k[t]/(t^{a_i}) \otimes_{k[[t]]} k[t]/(t^{m-e+1}),(t^{m-e+1})/(t^{m+1}) \Bigr), 
\]
with $a_i > 0$ and $\sum _i a_i = e$. This is isomorphic to $\mathbb{A}^e$ by Lemma \ref{lem:keisan}(1) and we complete the proof. 
\end{proof}

The following lemma is a generalization of Lemma 8.4 in \cite{EM09} to $k[t]$-schemes.
This lemma plays an important role in the proof of Theorem \ref{thm:PIA}. 
\begin{lem}\label{lem:EM8.4}
Let $A = \operatorname{Spec} k[t][x_1,\ldots,x_N]$ and let 
$X \subset A$ be a closed subscheme with the condition $(\star \star)_n$. 
Suppose that $c:= N - n \ge 0$.
We denote by $I_X \subset k[t][x_1,\ldots,x_N]$ the defining ideal of $X$ in $A$. 
Suppose that $I_X$ is generated by $c$ elements $f_1, \ldots , f_c \in k[t][x_1,\ldots,x_N]$. 
Let $C \subset A_{\infty}$ be an irreducible locally closed cylinder. 
If 
\begin{itemize}
\item $C \subset \bigcap_{i=1}^c \operatorname{Cont}^{\ge d_i} (f_i)$ and 
\item $C \cap X_{\infty} \cap \operatorname{Cont}^{e}(\operatorname{Jac}_{X/k[t]}) \not = \emptyset$
\end{itemize}
hold for some $d_i \ge 0$ and $e \ge 0$, 
then it follows that
\[
\operatorname{codim}_{X_{\infty}} (C \cap X_{\infty}) \le \operatorname{codim}_{A_{\infty}} (C) + e - \sum _{i=1}^c d_i.
\]
\end{lem}

\begin{proof}
The same proof as in \cite[Lemma 8.4]{EM09} works by replacing $\operatorname{Jac}_{M}$ in \cite{EM09} by $\operatorname{Jac}_{X/k[t]}$. 
We note that \cite[Proposition 4.4(ii)]{EM09}, which is used in the proof, is still true for our $k[t]$-scheme $X$:
\begin{itemize}
\item 
Let $p,m$ and $e$ be non-negative integers with $2p \ge m \ge p+e$. 
Let $\gamma \in X_p$ with $\pi _{m,p}^{-1} (\gamma) \not = \emptyset$ and $\operatorname{ord}_{\gamma} (\operatorname{Jac}_{X/k[t]}) = e$. 
Then it follows that $\pi _{m,p}^{-1} (\gamma) \simeq \mathbb{A}^{e+(m-p)n}$. 
\end{itemize}
In \cite[Proposition 4.4(ii)]{EM09}, the assertion above is proved for locally complete intersection varieties. 
This l.c.i.\ assumption is used only for proving $\operatorname{Fitt}^{n-1} (\Omega _{X/k}) = 0$. 
In our case, $\operatorname{Fitt}^{n-1} (\Omega _{X/k[t]}) = 0$ holds by the assumption that $I_X$ is generated by $c = N -n$ elements. 
\end{proof}

\subsection{Dimension of the arc spaces of quotient varieties}\label{subsection:quot}
In this subsection, we prove Proposition \ref{prop:DL2}, 
which is a generalization of \cite[Lemma 3.5]{DL02} to singular $k[t]$-schemes. 

Let $Y$ be a $k[t]$-scheme with the condition $(\star \star)_n$. 
Suppose that a finite group acts on $Y$ over $k[t]$. 
We denote by $X := Y/G$ its quotient, and by $h:Y \to X$ the quotient map. 
Let $B \subset Y_\infty$ be a $G$-invariant cylinder and $A = h_\infty (B)$. 
Let $e,e',e''\in\mathbb Z_{\ge 0}$.
Assume that 
\[
B \subset \operatorname{Cont}^{e''}(\operatorname{Jac}_{Y/k[t]}) \cap \operatorname{Cont}^{e}(\operatorname{jac}_f), \quad 
A \subset \operatorname{Cont}^{e'}(\operatorname{Jac}_{X/k[t]}). 
\]
We have the following diagram. 
\[
  \xymatrix{
Y_\infty \ar[r] \ar@/^20pt/[rr]^{h_\infty} \ar[d]_{\psi_m} & Y_\infty/G \ar[d] \ar[r] & X_\infty \ar[d]^{\psi_m} \\
Y_m \ar[r]  \ar[d]& Y_m/G \ar[r] \ar[d] & X_m \ar[d]\\
Y_{m-e} \ar[r] \ar@/_20pt/[rr]_{h_{m-e}} & Y_{m-e}/G \ar[r] & X_{m-e}
  }
\]

\begin{prop}[{cf.\ \cite[Lemma 3.5]{DL02}}]\label{prop:DL2}
In the setting above, the following hold. 
\begin{enumerate}
\item $A$ is a cylinder of $X_{\infty}$. 
\item $\operatorname{codim}(B)+e=\operatorname{codim}(A)$.
\end{enumerate}
\end{prop}

\begin{proof}
The second statement is obtained by specializing \cite[Lemma 10.20]{Yas} to the case where $\Phi = \Psi = k$, $g:[\mathcal{Y}/G] \to \mathcal{X}$ and 
$C = B/G$, where $\mathcal{Y}$ and $\mathcal{X}$ are the formal schemes over $k[[t]]$ associated to $Y$ and $X$ respectively (cf.\ Remark \ref{rmk:Sebag}). 
We note that $|\rm{J}_{\infty}([\mathcal{Y}/G])|$ in the notation in \cite{Yas} is equal to $Y_{\infty}/G$, and 
furthermore, $g_{\infty}|_{B/G}$ is injective. 
For readers' convenience, we give a proof in our setting below.

By the same argument of the proof of Proposition \ref{prop:EM6.2}, we can take 
a sufficiently large $m$ and a constructible subset $B_m \subset Y_m$, and $A_m = h_m (B_m)$ such that 
\[
A= \psi_m^{-1}(A_m), \quad B = \psi^{-1} _{m-e} (\pi_{m,m-e}(B_m)). 
\]
In particular $A$ is a cylinder of $X_{\infty}$. 

In order to prove $\operatorname{codim}(B)+e=\operatorname{codim}(A)$, 
it is sufficient to show that $\dim (h^{-1}_m (\alpha _m) \cap B_m) = e$ for each $\alpha _m \in A_m$. 
Let $\alpha \in A$ be a lift of $\alpha _m$ and let $\beta \in B$ be an arc satisfying $h_{\infty}(\beta) = \alpha$. 

We claim that 
\begin{itemize}
\item any arc in $\pi_{m,m-e} \bigl(h^{-1}_m(\alpha _m) \cap B_m \bigr)$ has the same image in $Y_{m-e}/G$. 
\end{itemize}
Take $\beta' _{m} \in h^{-1}_m (\alpha _m) \cap B_m$. 
Then by Lemma \ref{lem:EM6.2'}, there exists $\delta \in B$ such that $\delta _{m-e} = \pi_{m,m-e}(\beta' _m)$ and 
$h_{\infty} (\delta) = \alpha$. Since $Y_\infty/G\to X_\infty$ is injective (cf.\ \cite[Proposition 12.27(2)]{GW10}), 
it follows that $\beta$ and $\delta$ have the same image in $Y_{\infty}/G$.
Therefore the image of $\beta' _{m}$ in $Y_{m-e}/G$ coincides with that of $\beta$.

By the claim above, we have
\begin{align*}
h^{-1}_m (\alpha _m) \cap B_m 
&= h^{-1}_m(\alpha _m) \cap \bigcup _{\gamma \in G} \pi^{-1}_{m,m-e}(\gamma \cdot \beta _{m-e}) \\
&\subset 
\bigcup _{\gamma \in G} 
\left( \pi^{-1}_{m,m-e}(\gamma \cdot \beta _{m-e}) \cap \psi _m 
\bigl( \operatorname{Cont}^{e''}(\operatorname{Jac}_{Y/k[t]}) \bigr) \right). 
\end{align*}
Therefore by Proposition \ref{prop:EM4.4}, $h^{-1}_m (\alpha _m) \cap B_m$ is isomorphic to the union of the kernel $K_{\gamma}$ of 
\begin{align*}
&\operatorname{Hom}_{k[t]/(t^{m-e+1})}\bigl( ((\gamma \cdot \beta)^*\Omega_{Y/k[t]})/T_{\gamma} \otimes_{k[[t]]} k[t]/(t^{m-e+1}),(t^{m-e+1})/(t^{m+1}) \bigr) \\
&\to \operatorname{Hom}_{k[t]/(t^{m-e+1})}\bigl( (\alpha^*\Omega_{X/k[t]})/S\otimes_{k[[t]]} k[t]/(t^{m-e+1}),(t^{m-e+1})/(t^{m+1}) \bigr), 
\end{align*}
where $S$ and $T_{\gamma}$ be the torsion parts of $\alpha^*\Omega_{X/k[t]}$ and $(\gamma \cdot \beta)^*\Omega_{Y/k[t]}$, respectively. 
By definition of $\operatorname{ord}_{\beta}(\operatorname{jac}_h)$, $K_{\gamma}$ is isomorphic to 
\[
\operatorname{Hom}_{k[t]/(t^{m-e+1})}\Bigl( \bigoplus _{i} k[t]/(t^{a_i}) \otimes_{k[[t]]} k[t]/(t^{m-e+1}),(t^{m-e+1})/(t^{m+1}) \Bigr), 
\]
with $a_i > 0$ and $\sum _i a_i = e$. 
Therefore $K_{\gamma} \simeq \mathbb{A}^e$ and hence 
$\dim (h^{-1}_m(\alpha _m) \cap B_m) = e$, which completes the proof. 
\end{proof}

\section{Denef and Loeser's theory for quotient varieties}\label{section:DL2}
In this section, first we review the theory of the arc space of quotient varieties established 
by Denef and Loeser \cite{DL02} with more detail (Proposition \ref{prop:lift} and \ref{prop:quotsing}). 
In Subsection \ref{subsection:singDL}, we study quotients of singular varieties and 
state an analogous statement of Proposition \ref{prop:quotsing} for this setting (Proposition \ref{prop:hyperquot}). 

\subsection{Lifting property of arcs on quotient varieties}\label{subsection:lift}
Let $\overline{X}$ be a variety over $k$, and let $G$ be a finite group with order $d$ acting on $\overline{X}$. 
Let $q: \overline{X} \to X := \overline{X}/G$ be the quotient morphism. 
Let $Z \subset X$ be the minimal closed subset such that $q$ is \'{e}tale outside $Z$. 
Set 
\[
X_{\infty} ^{\rm{g}} := X_{\infty} \setminus Z_{\infty}, \quad
\overline{X}_{\infty} ^{\frac{1}{d}} := 
\operatorname{Hom}_{k} \bigl(\operatorname{Spec} k[[t^{\frac{1}{d}}]], \overline{X} \bigr). 
\]

\begin{lem}\label{lem:lift}
Any $\varphi \in X_{\infty} ^{\rm{g}}$ lifts to $\overline{X}_{\infty} ^{\frac{1}{d}}$. 
That is, the composition $\operatorname{Spec} k[[t^{\frac{1}{d}}]] \to \operatorname{Spec} k[[t]] \xrightarrow{\varphi} X$
factors through $\overline{X}$. 
Moreover, $\varphi$ has exactly $d$ lifts and $G$ acts on them transitively. 
\end{lem}
\begin{proof}
We denote $\varphi ' : \operatorname{Spec} k((t)) \to \operatorname{Spec} k[[t]] \overset{\varphi}{\to} X$ the decomposition. 
First we see that there are exactly $d$ lifts $\operatorname{Spec} k((t^{\frac{1}{d}})) \to \overline{X}$ of $\varphi '$. 
Let $\operatorname{Spec} L$ be the fiber product of $q: \overline{X} \to X$ and $\varphi': \operatorname{Spec} k((t)) \to X$. 
Since $q: \overline{X} \to X$ is \'{e}tale at the image of the generic point of $\varphi$, 
the extension $L/k((t))$ is \'{e}tale. Furthermore, we have $L^ G = k((t))$. 
Note that $k((t^{\frac{1}{d'}}))$ is the unique finite field extension of 
$k((t))$ of degree $d'$ (cf.\ \cite[Theorem 1.94]{Kol07}). 
Hence $L \simeq  \prod_{i=1}^c k((t^{\frac{1}{a_i}}))$ for some $a_i$ and $c$.
Note that if $a_i\neq a_j$, then $k((t^{\frac{1}{a_i}}))$ is not isomorphic to $k((t^{\frac{1}{a_j}}))$.
This implies that $L^ G$ is not a field if $a_i \neq a_j$ for some $i$ and $j$.
Therefore $L$ is decomposed as 
the product of $c$ copies of $k((t^{\frac{1}{a}}))$ for some $a$ and $c$ with $ac=d$. 
Hence we have $\# \operatorname{Hom}_{k((t))} \bigl( L, k((t^{\frac{1}{d}})) \bigr) = d$ and $G$ acts on 
$\operatorname{Hom}_{k((t))} \bigl( L, k((t^{\frac{1}{d}}))  \bigr)$ transitively. 
\[
  \xymatrix{
  \operatorname{Spec} k ((t^{\frac{1}{d}})) \ar@{..>}[rr]^{\exists} \ar[d] & & \overline{X} \ar[d]  \\
  \operatorname{Spec} k((t)) \ar[r] & \operatorname{Spec} k[[t]] \ar[r]^{\hspace{6mm}\varphi} & X
  }
\]
By the valuative criterion of properness on $\overline{X} \to X$, 
each $\operatorname{Spec} k((t^{\frac{1}{d}})) \to \overline{X}$ factors through
$\operatorname{Spec} k [[t^{\frac{1}{d}}]] \to \overline{X}$. 
We complete the proof. 
\end{proof}

We have two group actions on $\overline{X}_{\infty} ^{\frac{1}{d}}$. 
\begin{itemize}
\item Let $\xi = \xi _{d} \in k$ be a $d$-th primitive root of unity in $k$. 
$\mathbb{Z}/d \mathbb{Z} = \langle \xi \rangle$ acts on $\overline{X}_{\infty} ^{\frac{1}{d}}$ as follows. 
For $\overline{\varphi} = \overline{\varphi} (t^{\frac{1}{d}}) \in \overline{X}_{\infty} ^{\frac{1}{d}}$, 
we define $\overline{\varphi} (\xi t^{\frac{1}{d}} )$ by the composition
$\operatorname{Spec} k [[t^{\frac{1}{d}}]] \to \operatorname{Spec} k [[t^{\frac{1}{d}}]] \xrightarrow{\overline{\varphi}} \overline{X}$, 
where the first map $\operatorname{Spec} k [[t^{\frac{1}{d}}]] \to \operatorname{Spec} k [[t^{\frac{1}{d}}]]$ is a map induced by 
the ring homomorphism $k [[t^{\frac{1}{d}}]] \to k [[t^{\frac{1}{d}}]]; t^{\frac{1}{d}} \mapsto \xi t^{\frac{1}{d}}$. 

\item $G$ acts on $\overline{X}_{\infty} ^{\frac{1}{d}}$ as follows. 
For $\gamma \in G$ and $\overline{\varphi} \in \overline{X}_{\infty} ^{\frac{1}{d}}$, 
we define $\gamma \overline{\varphi}$ to be the composition 
$\operatorname{Spec} k [[t^{\frac{1}{d}}]] \xrightarrow{\overline{\varphi}} \overline{X} \xrightarrow{\gamma} \overline{X}$. 
\end{itemize}

\begin{lem}\label{lem:conj}
Let $\varphi \in X_{\infty} ^{\rm{g}}$ and let $\overline{\varphi} \in \overline{X}_{\infty}^{\frac{1}{d}}$ be its lift. 
Then the following hold: 
\begin{enumerate}
\item There exists the unique $\gamma \in G$ such that 
$\overline{\varphi}(\xi t^{\frac{1}{d}}) = \gamma \overline{\varphi}(t^{\frac{1}{d}})$. 
\item If $\overline{\varphi}'$ is another lift, and $\gamma '$ satisfies $\overline{\varphi}'(\xi t^{\frac{1}{d}}) = \gamma' \overline{\varphi}'(t^{\frac{1}{d}})$, 
then $\gamma$ and $\gamma'$ are in the same conjugacy class. 
\end{enumerate}
\end{lem}
\begin{proof}
Since $\overline{\varphi}, \overline{\varphi}'$, and $\overline{\varphi}(\xi t^{\frac{1}{d}})$ are lifts of $\varphi$, 
(1) and (2) follow from Lemma \ref{lem:lift}. 
\end{proof}

For $\gamma \in G$, we define $\overline{X}_{\infty} ^{\frac{1}{d}, (\gamma)}$ and 
$X_{\infty} ^{\rm{g}, (\gamma)}$ as follows. 
\begin{itemize}
\item $\overline{X}_{\infty} ^{\frac{1}{d}, (\gamma)} := 
\left \{ \overline{\varphi} \in \overline{X}_{\infty} ^{\frac{1}{d}} 
\ \middle | \  \overline{\varphi}(\xi t^{\frac{1}{d}}) = \gamma \overline{\varphi}(t^{\frac{1}{d}}) \right\}$. 

\item $X_{\infty} ^{\rm{g}, (\gamma)} := \left \{ \varphi \in X_{\infty} ^{\rm{g}} 
\ \middle | \  \text{$\varphi$ lifts to an arc in $\overline{X}_{\infty} ^{\frac{1}{d}, (\gamma)}$} \right \}$. 
\end{itemize}

\begin{lem}\label{lem:conj2}
\begin{enumerate}
\item $X_{\infty} ^{\rm{g}, (\gamma)} = X_{\infty} ^{\rm{g}, (\gamma ')}$ holds 
if $\gamma$ and $\gamma '$ are in the same conjugacy class. 

\item There is a natural map $\rho _{\gamma} : \overline{X}_{\infty} ^{\frac{1}{d}, (\gamma)} \to X_{\infty}$. 

\item $\overline{X}_{\infty} ^{\frac{1}{d}, (\gamma)}$ and $\rho _{\gamma}$ are $C_{\gamma}$-invariant, where $C_{\gamma}$ is the centralizer of $\gamma$. 

\item The $C_{\gamma}$-action on each fiber over $X_{\infty}^{\rm{g}, (\gamma)}$ of $\rho _{\gamma}$ is transitive. 
\end{enumerate}
\end{lem}
\begin{proof}
Prove (1). If $\gamma' = \beta \gamma \beta^{-1}$ for some $\beta \in G$ and $\overline{\varphi} (\xi t^{\frac{1}{d}}) = \gamma \overline{\varphi} (t^{\frac{1}{d}})$, 
then $\overline{\varphi}' := \beta \overline{\varphi}$ satisfies 
$\gamma ' \overline{\varphi}' (t^{\frac{1}{d}}) = \overline{\varphi}' (\xi t^{\frac{1}{d}})$.

Prove (2). We may assume that $\overline{X}$ is an affine variety $\operatorname{Spec} R$. 
Let $\overline{\varphi} \in \overline{X}_{\infty} ^{\frac{1}{d}, (\gamma)}$ and 
let $\overline{\varphi}^* : R \to k [[t^{\frac{1}{d}}]]$ be the corresponding ring homomorphism. 
Suppose $a \in R^G$. It is sufficient to show that $\overline{\varphi}^* (a)\in k [[t]]$. 
Since $a = \gamma \cdot a$, it follows that 
\[
(\overline{\varphi}^* (a))(t^{\frac{1}{d}}) = (\overline{\varphi}^* (\gamma \cdot a))(t^{\frac{1}{d}}) = (\overline{\varphi}^* (a))(\xi t^{\frac{1}{d}}), 
\]
which shows $\overline{\varphi}^* (a) \in k [[t]]$. 

The proof of (3) is straightforward. 

Prove (4). Suppose that $\overline{\varphi} _1, \overline{\varphi} _2 \in \overline{X}_{\infty} ^{\frac{1}{d}, (\gamma)}$ 
are lifts of $\varphi \in X_{\infty}^{\rm{g}, (\gamma)}$. 
Then by Lemma \ref{lem:lift}, $\overline{\varphi} _1 = \alpha \overline{\varphi} _2$ holds for some $\alpha \in G$. 
Since $\overline{\varphi}_i(\xi t^{\frac{1}{d}}) = \gamma \overline{\varphi}_i(t^{\frac{1}{d}})$ holds for each $i$, 
it follows that $\alpha \gamma = \gamma \alpha$. 
\end{proof}

\begin{prop}[{\cite[2.1]{DL02}, cf.\ \cite[Section 3]{Yas16}}]\label{prop:lift}
\begin{enumerate}
\item $X_{\infty}^{\rm{g}} = \bigsqcup _{\langle \gamma \rangle \in \operatorname{Conj}(G)} X_{\infty} ^{\rm{g}, (\gamma)}$ holds. 
\item $\rho _{\gamma}$ induces two maps 
\begin{align*}
\rho &: \bigsqcup _{\langle \gamma \rangle \in \operatorname{Conj}(G)} \overline{X}_{\infty} ^{\frac{1}{d}, (\gamma)} \to X_{\infty}, \\
\rho' &: \bigsqcup _{\langle \gamma \rangle \in \operatorname{Conj}(G)} \left( \overline{X}_{\infty} ^{\frac{1}{d}, (\gamma)}/C_{\gamma} \right) \to X_{\infty}, 
\end{align*}
and $\rho'$ is bijective over $X_{\infty}^{\rm{g}}$. 
\end{enumerate}
\end{prop}
\begin{proof}
(1) follows from Lemma \ref{lem:lift}, Lemma \ref{lem:conj}(2) and Lemma \ref{lem:conj2}(1). 
(2) follows from Lemma \ref{lem:conj2}. 
\end{proof}

\subsection{Arc spaces of quotient singularities}\label{subsection:DLquot}
Let $d$ be a positive integer and let $\xi$ be a primitive $d$-th root of unity in $k$. 
Let $G \subset \operatorname{GL}_N( k)$ be a finite group with order $d$ which acts on 
$\overline{A} = \mathbb{A}^N _k = \operatorname{Spec} k[x_1,\ldots,x_N]$. 
Set $A = \overline{A}/G = \operatorname{Spec} k[x_1, \ldots , x_n]^G$. 
Suppose that an element $\gamma \in G$ is the diagonal matrix with entries $\xi ^{e_1}, \ldots , \xi ^{e_n}$ ($0 \le e_i \le d-1$). 

\begin{lem}[{\cite[2.3]{DL02}}]
The $k[t]$-ring homomorphism 
\[
\epsilon_{\gamma} ^*: k [t][x_1, \ldots, x_n] \to k [t^{\frac{1}{d}}][x_1, \ldots, x_n]; \quad x_i \mapsto t^{\frac{e_i}{d}}x_i
\]
induces a bijective map $\epsilon_{\gamma}: \overline{A} _{\infty} \to \overline{A} _{\infty}^{\frac{1}{d}, (\gamma)}$. 
\end{lem}
\begin{proof}
First we shall see that $\epsilon_{\gamma} ^*$ induces an injective map 
$\epsilon_{\gamma}: \overline{A} _{\infty} \to \overline{A} _{\infty}^{\frac{1}{d}}$. 
Let $\overline{\varphi} \in \overline{A} _{\infty}$ and let 
$\overline{\varphi}^* : k [t][x_1, \ldots , x_{n}] \to k [[t]]$ be the corresponding $k[t]$-ring homomorphism. 
Then we define $\epsilon_{\gamma} (\overline{\varphi}) \in  \overline{A} _{\infty}^{\frac{1}{d}}$ to be the arc corresponding to 
the $k[t^{\frac{1}{d}}]$-ring homomorphism 
\[
k[t^{\frac{1}{d}}][x_1, \ldots , x_{n}] \to k [[t^{\frac{1}{d}}]]; \quad x_i \mapsto t^{\frac{e_i}{d}} \overline{\varphi}^*(x_i). 
\]
Since the map $\epsilon_{\gamma}: \overline{A} _{\infty} \to \overline{A} _{\infty}^{\frac{1}{d}}$ is injective, 
it is sufficient to show that its image is $\overline{A} _{\infty}^{\frac{1}{d}, (\gamma)}$. 

Let $\overline{\varphi} \in \overline{A} _{\infty}^{\frac{1}{d}}$ and let 
$\overline{\varphi}^* : k[t^{\frac{1}{d}}][x_1, \ldots , x_{n}] \to k [[t^{\frac{1}{d}}]]$ be the corresponding $k[t^{\frac{1}{d}}]$-ring homomorphism. 
Set $f_i := \overline{\varphi}^*(x_i) \in k [[t^{\frac{1}{d}}]]$. 
Then the condition $\overline{\varphi} \in \overline{A} _{\infty}^{\frac{1}{d}, (\gamma)}$ is equivalent to the condition 
that 
\[
\overline{\varphi}^*(x_i) (\xi t^{\frac{1}{d}}) = \overline{\varphi}^*(\gamma \cdot x_i) (t^{\frac{1}{d}})
\]
holds for each $i$. 
This condition is equivalent to $f_i(t^{\frac{1}{d}}) \in t^{\frac{e_i}{d}} k [[t]]$ since we have 
\begin{itemize}
\item $\overline{\varphi}^*(x_i) (\xi t^{\frac{1}{d}}) = f_i (\xi t^{\frac{1}{d}})$, and
\item $\overline{\varphi}^*(\gamma \cdot x_i) (t^{\frac{1}{d}}) 
= \overline{\varphi}^*(\xi ^{e_i} x_i) (t^{\frac{1}{d}}) = \xi ^{e_i} \overline{\varphi}^*(x_i) (t ^{\frac{1}{d}}) = \xi ^{e_i} f_i(t^{\frac{1}{d}})$. 
\end{itemize}
This equivalence shows that the image of $\epsilon_{\gamma}$ is $\overline{A} _{\infty}^{\frac{1}{d}, (\gamma)}$. 
\end{proof}

\begin{lem}\label{lem:geomquot}
\begin{enumerate}
\item $\epsilon_{\gamma}$ is $G$-equivariant. 

\item There is a natural inclusion $\overline{A} _{\infty}/C_{\gamma} \hookrightarrow (\overline{A} / C_{\gamma}) _{\infty}$. 
\end{enumerate}
\end{lem}
\begin{proof}
(1) easily follows from the definition of the $G$-actions on $\overline{A} _{\infty}$ and $\overline{A} _{\infty}^{\frac{1}{d}, (\gamma)}$. 
(2) follows from \cite[Proposition 12.27(2)]{GW10}. 
\end{proof}

\begin{prop}[{\cite[2.7]{DL02}}]\label{prop:quotsing}
$\epsilon _{\gamma} ^*$ induces a $k[t]$-ring homomorphism
\[
\lambda^* _{\gamma}: k[t][x_1, \ldots , x_n]^{G} \to k[t][x_1, \ldots , x_n]^{C_{\gamma}}; \quad x_i \mapsto t^{\frac{e_i}{d}} x_i, 
\]
and a morphism $\lambda_{\gamma}: (\overline{A} / C_{\gamma}) _{\infty} \to A _{\infty}$, and the following diagram commutes. 
\[
  \xymatrix{
 \overline{A}_{\infty} \ar@{->>}[d] \ar[r]^-{\epsilon _{\gamma}\rm{(bij.)}} & \overline{A}_{\infty} ^{\frac{1}{d},(\gamma)} \ar@{->>}[d] \ar[r]^-{\rho _{\gamma}} & A_{\infty} \\
\overline{A}_{\infty}/C_{\gamma}  \ar[r]^-{\rm{(bij.)}} \ar@{^{(}-{>}}[d] & \overline{A}_{\infty} ^{\frac{1}{d},(\gamma)} /C_{\gamma} \ar[ru] & \\
(\overline{A}/C_{\gamma})_{\infty}  \ar@/_30pt/[rruu] _{\lambda _{\gamma}} & & \\
  	}
\]
Moreover, the composite map $\overline{A} _{\infty}/C_{\gamma} \to A _{\infty}$ is bijective over $A _{\infty}^{\rm{g}, (\gamma)}$. 
\end{prop}
\begin{proof}
Note that $\epsilon _{\gamma} ^*$ induces a $k[t]$-ring homomorphism 
$k[t][x_1, \ldots , x_n]^{\langle \gamma \rangle} \to k[t][x_1, \ldots , x_n]$, 
where $\langle \gamma \rangle \subset G$ is the subgroup generated by $\gamma$. 
Then $\lambda^* _{\gamma}$ is its restriction to $k[t][x_1, \ldots , x_n]^{G} \subset k[t][x_1, \ldots , x_n]^{\langle \gamma \rangle}$. 
Then the second assertion on the bijectivity follows from Proposition \ref{prop:lift}(2). 
\end{proof}

\subsection{Arc spaces of quotient varieties}\label{subsection:singDL}
We take over the notations in Subsection \ref{subsection:DLquot}. 
Suppose that $\overline{X} \subset \overline{A}$ is a $G$-invariant subvariety. 
In this subsection, we study the arc space of the quotient variety $X := \overline{X} /G$. 

Let $I_X \subset k[x_1, \ldots , x_n] ^G$ be the defining ideal of $X$ in $A = \operatorname{Spec} k[x_1, \ldots, x_n] ^G$. 
We denote by the same character $I_X$ the ideal of $k[t][x_1, \ldots , x_n] ^G$ generated by the original $I_X$. 
We denote by 
\[
\widetilde{I}_X ^{(\gamma)} \subset k[t][x_1, \ldots , x_n] ^{C_{\gamma}}, \quad 
\overline{I}_X ^{(\gamma)} \subset k[t][x_1, \ldots , x_n]
\]
the ideals generated by $\lambda _{\gamma} ^*(I_X)$ and $\overline{\lambda} _{\gamma} ^*(I_X)$, respectively, 
where we set $\overline{\lambda} _{\gamma} ^*$ as the composition of $\lambda _{\gamma} ^*$ and 
the inclusion $k [t][x_1, \ldots, x_n]^{C_{\gamma}} \to k[t][x_1, \ldots, x_n]$. 
Then we have the following commutative diagram. 
\[
  \xymatrix{
k[t][x_1, \ldots, x_n]^{G} \ar[r] ^{\lambda _{\gamma} ^*} _{x_i \mapsto t^{\frac{e_i}{d}} x_i} \ar@/^20pt/[rr]^{\overline{\lambda} _{\gamma} ^*} \ar@{->>}[d] & k[t][x_1, \ldots, x_n]^{C_{\gamma}} \ar@{->>}[d] \ar[r] & k[t][x_1, \ldots, x_n] \ar@{->>}[d] \\
k[t][x_1, \ldots, x_n]^{G}/I_X \ar[r] ^{\lambda _{\gamma} ^{X*}} & k[t][x_1, \ldots, x_n]^{C_{\gamma}} / \widetilde{I}_X ^{(\gamma)} \ar[r] & k[t][x_1, \ldots, x_n] / \overline{I}_X ^{(\gamma)}
  }
\]

\noindent
We define the arc spaces $\widetilde{X}_{\infty} ^{(\gamma)}$ and $\overline{X}_{\infty} ^{(\gamma)}$ as follows 
(see Subsection \ref{subsection:ktjet} for the definition of the arc spaces for $k[t]$-schemes). 
\begin{align*}
\widetilde{X}_{\infty} ^{(\gamma)} &:= \left( \operatorname{Spec} k[t][x_1, \ldots, x_n]^{C_{\gamma}} / \widetilde{I}_X ^{(\gamma)} \right)_{\infty}, \\
\overline{X}_{\infty} ^{(\gamma)} &:= \left( \operatorname{Spec} k[t][x_1, \ldots, x_n] / \overline{I}_X ^{(\gamma)} \right)_{\infty}. 
\end{align*}

\noindent
Then we have the following diagram of arc paces: 
\[
\xymatrix{
A_{\infty} & (\overline{A}/C_{\gamma})_{\infty} \ar[l]_{\lambda _{\gamma}} & \overline{A}_{\infty} \ar[l] \ar@/_20pt/[ll]_{\overline{\lambda} _{\gamma}} \\
X_{\infty} \ar@{^{(}-{>}}[u] & \widetilde{X}_{\infty} ^{(\gamma)} \ar[l] \ar@{^{(}-{>}}[u] & \overline{X}_{\infty} ^{(\gamma)} \ar[l] \ar@{^{(}-{>}}[u]
}
\]
Here, the vertical arrows are closed immersions by Lemma \ref{lem:closed}(1). 
Moreover there is a natural injective map $\overline{X}_{\infty} ^{(\gamma)} / C_{\gamma} \hookrightarrow \widetilde{X}_{\infty} ^{(\gamma)}$ 
(cf.\ Lemma \ref{lem:geomquot}(2)). 

\begin{prop}\label{prop:hyperquot}
The ring homomorphism $\lambda _{\gamma} ^*$ induces a morphism $\lambda _{\gamma} ^X : \widetilde{X}_{\infty} ^{(\gamma)} \to X_{\infty}$. 
Moreover, the composition $\overline{X}_{\infty} ^{(\gamma)} / C_{\gamma} \hookrightarrow \widetilde{X}_{\infty} ^{(\gamma)} \to X_{\infty}$ 
is bijective over $X_{\infty} \cap A_{\infty}^{\rm{g}, (\gamma)}$. 
\end{prop}
\begin{proof}
The first assertion is straightforward. 

We denote $\lambda _{\gamma}: (\overline{A}/C_{\gamma})_{\infty} \to A_{\infty}$ and 
$\overline{\lambda} _{\gamma} : \overline{A}_{\infty} \to A_{\infty}$. 
We can identify $\overline{X}_{\infty} ^{(\gamma)}$, $\widetilde{X}_{\infty} ^{(\gamma)}$ and $X_{\infty}$ 
with the closed subspaces of $\overline{A}_{\infty}$, $(\overline{A}/C_{\gamma})_{\infty}$ and $A_{\infty}$, respectively, 
and under these identifications we have
\[
\lambda _{\gamma} ^{-1} (X_{\infty}) = \widetilde{X}_{\infty} ^{(\gamma)}, \qquad 
\overline{\lambda} _{\gamma} ^{-1} (X_{\infty}) = \overline{X}_{\infty} ^{(\gamma)}
\] 
by Lemma \ref{lem:closed}(2). Therefore the second assertion follows from Proposition \ref{prop:quotsing}. 
\end{proof}

\begin{rmk}\label{rmk:Yas16}
In \cite{Yas16}, Yasuda also generalizes the theory of Denef and Loeser to singular varieties. 
The construction in \cite{Yas16} is intrinsic and more general, and it works even in positive characteristics. 
The propositions in this section are covered in the paper \cite{Yas16}. 
The correspondence between the notations in this section and \cite{Yas16} is described below.

Let $\gamma \in G$, and let $E$ be the $G$-cover of $D = \operatorname{Spec} k[[t]]$ corresponding to $\gamma$. 
Then $\overline{X}^{\frac{1}{d},(\gamma)} _{\infty}$ in Subsection \ref{subsection:lift} corresponds to
the set $\operatorname{Hom}_D ^G (E,V)$ of $G$-equivariant $D$-homomorphisms in \cite[Section 3]{Yas16} 
when $V = \overline{X} \times _{k} D$. 
Furthermore, $\overline{X}^{\frac{1}{d},(\gamma)} _{\infty}/C_{\gamma}$ corresponds to
$J_{\infty}^{G,E} V := \operatorname{Hom}_D ^G (E,V)/C_G(H)$ in \cite{Yas16}
when $H = \langle \gamma \rangle$ is the subgroup of $G$ generated by $\gamma$. 

The following diagram in Subsection \ref{subsection:DLquot}
\[
  \xymatrix{
 \overline{A}_{\infty} \ar@{->>}[d] \ar[r]^-{\epsilon _{\gamma}\rm{(bij.)}} & \overline{A}_{\infty} ^{\frac{1}{d},(\gamma)} \ar@{->>}[d] \ar[r]^-{\rho _{\gamma}} & A_{\infty} \\
\overline{A}_{\infty}/C_{\gamma}  \ar[r]^-{\rm{(bij.)}}  & \overline{A}_{\infty} ^{\frac{1}{d},(\gamma)} /C_{\gamma} \ar[ru] & \\
  	}
\]
corresponds to the diagram
\[
  \xymatrix{
J_{\infty} V^{|F|} \ar@{->>}[d] \ar[r]^-{\rm{(bij.)}} & \Xi _F \ar@{->>}[d] \ar[r] & J_{\infty} X \\
J_{\infty} V^{|F|}/C_G(H)  \ar[r]^-{\rm{(bij.)}}  & J_{\infty} ^{G,E} V \ar[ru]_{p_{\infty}} & \\
  	}
\]
in \cite[Section 4]{Yas16} when $V = \overline{A} \times _{k} D$ and $X=V/G$. 

The following diagram in Subsection \ref{subsection:singDL}
\[
\xymatrix{
\overline{A}_{\infty} \ar[r] & \overline{A}_{\infty}/C_{\gamma} \ar[r] & A_{\infty} \\
\overline{X}_{\infty} ^{(\gamma)} \ar[r] \ar@{^{(}-{>}}[u] & \overline{X}_{\infty} ^{(\gamma)}/C_{\gamma} \ar[r] \ar@{^{(}-{>}}[u] & X_{\infty} \ar@{^{(}-{>}}[u]
}
\]
corresponds to 
\[
\xymatrix{
J_{\infty}V^{|F|} \ar[r] & J_{\infty}V^{|F|}/C_{G}(H) \ar[r] & J_{\infty} X \\
J_{\infty}{\sf v}^{|F|} \ar[r] \ar@{^{(}-{>}}[u] & J_{\infty}{\sf v}^{|F|}/C_{G}(H) \ar[r] \ar@{^{(}-{>}}[u] & J_{\infty} {\sf x} \ar@{^{(}-{>}}[u]
}
\]
in \cite[Section 7]{Yas16} when ${\sf v} = \overline{X} \times _k D$ and ${\sf x} = {\sf v}/G$.

We note that ${\sf v}^{|F|}$ in \cite[Section 7]{Yas16} corresponds to 
the dominant component $\overline{X}^{(\gamma)} _{\rm dom}$ of 
$\overline{X}^{(\gamma)} := \operatorname{Spec} \bigr( k[t][x_1, \ldots, x_n] / \overline{I}_X ^{(\gamma)} \bigl)$. 
In Section \ref{section:mld_hyperquot} and later, we will work on $\overline{X}^{(\gamma)}$ itself
instead of ${\sf v}^{|F|} = \overline{X}^{(\gamma)} _{\rm dom}$, although their arc spaces are equal (cf.\ Remark \ref{rmk:Sebag}(3)). 
One of the advantages of working on $\overline{X}^{(\gamma)}$ instead of $\overline{X}^{(\gamma)} _{\rm dom}$ is 
that Lemma \ref{lem:EM8.4} can be applied to $\overline{X}^{(\gamma)}$ in the proof of Theorem \ref{thm:PIA}.
\end{rmk}

\section{Arc space of hyperquotient singularities}\label{section:mld_hyperquot}
In this section, we investigate the minimal log discrepancies of hyperquotient singularities in terms of the arc spaces of $k[t]$-schemes 
(Theorem \ref{thm:mld_hyperquot}). 

\subsection{Arc spaces of quotient singularities}\label{subsection:arc_quot}
In this subsection, we study the arc spaces of quotient singularities. 

Let $d$ be a positive integer and let $\xi$ be a primitive $d$-th root of unity in $k$. 
Let $G \subset \operatorname{GL}_N( k)$ be a finite group with order $d$ which acts on 
$\overline{A} = \mathbb{A}^N _k = \operatorname{Spec} k[x_1,\ldots,x_N]$. 
We denote by 
\[
A := \overline{A}/G
\]
the quotient varieties. Let $Z \subset A$ be the minimal closed subset such that $\overline{A} \to A$ is \'{e}tale outside $Z$. 
We assume that $\operatorname{codim} Z \ge 2$, and hence the quotient map $\overline{A} \to A$ is \'{e}tale in codimension one. 
We fix a positive integer $r$ such that $\omega _A ^{[r]}$ is invertible (cf.\ \cite[2.40]{Kol13}). 

Let $\gamma\in G$ and let $C_\gamma$ be the centralizer of $\gamma$ in $G$. 
We denote by 
\[
A^{(\gamma)} := \overline{A}/C_{\gamma}
\]
the quotient varieties. 
Since $G$ is a finite group, $\gamma$ can be diagonalized 
$\xi^{e_1},\ldots,\xi^{e_N}$ $(0\le e_i\le d-1)$ with a suitable basis $x_1,\ldots,x_N$. 
We set $\operatorname{age}(\gamma) = \frac{1}{d} \sum _{i=1} ^N e_i$ (cf.\ \cite[Definition 3.20]{Kol13}).

Let 
\[
\lambda_{\gamma} ^*: k[t][x_1, \ldots, x_N]^{G}  \to  k[t][x_1, \ldots, x_N]^{C_{\gamma}};\quad 
x_i \mapsto t^{\frac{e_i}{d}}x_i
\]
be the $k[t]$-morphism as in Section \ref{section:DL2}. 
Then, we have the following maps
\[
  \xymatrix{
 k[t][x_1, \ldots, x_N]^{G} \ar[r] ^{\lambda _{\gamma} ^*} _{x_i \mapsto t^{\frac{e_i}{d}}x_i} \ar@/^20pt/[rr]^{\overline{\lambda} _{\gamma} ^*}  &  k[t][x_1, \ldots, x_N]^{C_{\gamma}}  \ar[r]^i &  k[t][x_1, \ldots, x_N],   \\
  }
\]
here $i$ is the inclusion map and $\overline{\lambda}_{\gamma}^*=i\circ\lambda_{\gamma}^*$ is the composite map. 
Then, we have the following morphisms between the corresponding $k[t]$-varieties. 
\[
  \xymatrix{
A '     & {A^{(\gamma)}}' = (\overline{A}/C_{\gamma})'  \ar[l]^{\lambda _{\gamma} \hspace{11mm}} &  \overline{A}' \ar[l]^{\hspace{12mm}q} \ar@/_15pt/[ll]_{\overline{\lambda} _{\gamma} }, 
  }
\]
where we denote the base change $Y' := Y \times _{\operatorname{Spec} k} \operatorname{Spec} k[t]$ for a $k$-variety $Y$.

\begin{lem}[{cf.\ \cite[Lemma 6.5]{Yas16}}]\label{lem:age1}
Let $\alpha \in \overline{A}_{\infty}$ be an arc. 
Set $\alpha ':= (\lambda _{\gamma} \circ q)_{\infty}(\alpha)$. 
Then it follows that 
\[
\operatorname{ord}_{\alpha} (\operatorname{jac}_{\overline{\lambda} _{\gamma}}) 
= \frac{1}{r} \operatorname{ord}_{\alpha '} (\mathfrak n_{r,A}) +
\operatorname{age} (\gamma). 
\]
\end{lem}

\begin{proof}
This follows from \cite[Lemma 6.5]{Yas16} (cf.\ Remark \ref{rmk:Yas16}). 
We note that ${\bf v}_V(E)$ in \cite[Lemma 6.5]{Yas16} is equal to $\operatorname{age}(\gamma)$ 
when $E$ is the $G$-cover of $\operatorname{Spec} k[[t]]$ corresponding to the conjugacy class of $\gamma$ 
(cf.\ \cite[Lemma 4.3]{WY15}). 

For readers' convenience, we give a proof in our setting below. 

Let 
\[
C = \operatorname{Spec} k[t^\frac{1}{d}][x_1,\ldots,x_N]^G, \quad
\overline{C} = \operatorname{Spec} k[t^\frac{1}{d}][x_1,\ldots,x_N]. 
\]
We denote by the same character $\overline{\lambda} _{\gamma}$ for the $k[t^{\frac{1}{d}}]$-morphism 
$\overline{C} \to C$ induced by the original $\overline{\lambda} _{\gamma}: \overline{A}' \to A'$. 
Then $\overline{\lambda} _{\gamma}: \overline{C} \to C$ is decomposed as 
\[
\xymatrix{
\overline{C} \ar[r]^s & \overline{C} \ar[r]^t & C 
}
\]
which correspond to $k[t^{\frac{1}{d}}]$-ring homomorphisms
\[
  \xymatrix{
k[t^\frac{1}{d}][x_1, \ldots, x_N] & k[t^\frac{1}{d}][x_1, \ldots, x_N] \ar[l]^{t^{\frac{e_i}{d}}x_i \leftarrow\!\shortmid x_i} & k[t^\frac{1}{d}][x_1, \ldots, x_N]^{G}. \ar@{^{(}-{>}}[l] 
  }
\]
Let $\beta : \operatorname{Spec} k[[t^\frac{1}{d}]] \to \overline{C}$ be the lift of $\alpha$. 
Then we have the following diagram. 
\[
\xymatrix{
\operatorname{Spec} k[[t]] \ar[r]^(.6){\alpha} & \overline{A}' \ar[r]^q \ar@/^20pt/[rr]^{\overline{\lambda}_{\gamma}} & A^{(\gamma)'} \ar[r]^{\lambda _{\gamma}} & A' \ar[r] & \operatorname{Spec} k[t] \\
\operatorname{Spec} k[[t^{\frac{1}{d}}]] \ar[u] \ar[r]_(.6){\beta} & \overline{C} \ar[u] \ar[r]_s \ar@/_20pt/[rr]_{\overline{\lambda}_{\gamma}} & \overline{C} \ar[r]_t & C \ar[u] \ar[r] & \operatorname{Spec} k[t^{\frac{1}{d}}] \ar[u]
}\]
By abuse of notation, we define the order $\operatorname{ord}_{\beta} ( \operatorname{jac}_{\overline{\lambda} _{\gamma}} )$ by
\[
\operatorname{ord}_{\beta} ( \operatorname{jac}_{\overline{\lambda} _{\gamma}} ) := 
\frac{1}{d} \operatorname{length} \Bigl( \operatorname{Coker} \bigl( (\overline{\lambda}_{\gamma} \circ \beta)^* \Omega _{C/k[t^{1/d}]} \to 
\beta ^* \Omega _{\overline{C}/k[t^{1/d}]} \bigr) \Bigr), 
\]
where the length is considered as a $k[[t^{\frac{1}{d}}]]$-module. 
Here, we note that $\Omega _{\overline{C}/k[t^{1/d}]}$ is locally free because $\overline{C}$ is smooth over $k[t^{\frac{1}{d}}]$. 
Since 
\begin{align*}
&\operatorname{Coker} \Bigl( (\overline{\lambda}_{\gamma} \circ \beta)^* \Omega _{C/k[t^{1/d}]} \to \beta ^* \Omega _{\overline{C}/k[t^{1/d}]} \Bigr) \\ 
&\simeq 
\operatorname{Coker} \Bigl( (\overline{\lambda}_{\gamma} \circ \alpha)^* \Omega _{A'/k[t]} \to \alpha ^* \Omega _{\overline{A}'/k[t]} \Bigr) \otimes_{k[[t]]} k[[t^{\frac{1}{d}}]], 
\end{align*}
we have
\[
\operatorname{ord}_{\alpha} (\operatorname{jac}_{\overline{\lambda} _{\gamma}}) = \operatorname{ord}_{\beta} (\operatorname{jac}_{\overline{\lambda} _{\gamma}}). 
\]
We also define $\operatorname{ord}_{\beta} (\operatorname{jac}_{s})$ and $\operatorname{ord}_{\beta '} (\operatorname{jac}_{t})$ for $\beta ' := s_{\infty} (\beta)$ by the same way. 
Then we have
\[
\operatorname{ord}_{\beta} (\operatorname{jac}_{\overline{\lambda} _{\gamma}}) =
\operatorname{ord}_{\beta} (\operatorname{jac}_{s}) + \operatorname{ord}_{\beta '} (\operatorname{jac}_{t}), \quad 
\operatorname{ord}_{\beta} (\operatorname{jac}_{s}) = \operatorname{age} (\gamma)
\]
by Lemma \ref{lem:additive} and an easy calculation. 
Then it is sufficient to show that $\operatorname{ord}_{\beta '} (\operatorname{jac}_{t}) = \frac{1}{r} \operatorname{ord}_{\alpha '} (\mathfrak n_{r,A})$. 

In the following commutative diagram
\[
\xymatrix{
\bigl( \Omega ^N _{\overline{C}/k[t^{1/d}]} \bigr)^{\otimes r} \ar[d]_{\simeq} & t^* \bigl( \Omega ^N _{C/k[t^{1/d}]} \bigr)^{\otimes r} \ar[d]\ar[l] \\
\bigl( \omega _{\overline{C}/k[t^{1/d}]} \bigr) ^{[r]} & t^* \bigl( \omega _{C/k[t^{1/d}]} \bigr)^{[r]} \ar[l]^{\simeq}, 
}
\]
the vertical map $\bigl( \Omega ^N _{\overline{C}/k[t^{1/d}]} \bigr)^{\otimes r} \to \bigl( \omega _{\overline{C}/k[t^{1/d}]} \bigr) ^{[r]}$ 
is an isomorphism because $\overline{C}$ is smooth over $k[t^{\frac{1}{d}}]$. 
Furthermore, $t^* \bigl( \omega _{C/k[t^{1/d}]} \bigr)^{[r]} \to \bigl( \omega _{\overline{C}/k[t^{1/d}]} \bigr) ^{[r]}$ 
is an isomorphism since $t$ is \'{e}tale in codimension one. 
By combining with Lemma \ref{lem:order}(2), we conclude that 
\[
r \operatorname{ord}_{\beta '} (\operatorname{jac}_{t}) = \operatorname{ord}_{\alpha '} (\mathfrak n_{r,A}). 
\]
We complete the proof. 
\end{proof}

\subsection{Arc spaces of hyperquotient singularities}\label{subsection:arc_hyperquot}
In this subsection, we study the arc spaces of hyperquotient singularities. 

Let $\xi$, $G$, $\gamma$, $C_\gamma$, 
$\overline{A}$, $A$, $Z$, $A^{(\gamma)}$, $r$, 
$\lambda_{\gamma}^*$, $i$, 
$\overline{\lambda}_{\gamma}^*$,
$\lambda_{\gamma}$, $q$, 
$\overline{\lambda}_{\gamma}$
be as in Subsection \ref{subsection:arc_quot}.
Let $f_1,\ldots,f_c \in k[x_1, \ldots, x_N]^{G}$ be a regular sequence which is contained in the maximal ideal at the origin. 
We set 
\[
B := \operatorname{Spec} k[x_1, \ldots, x_N]^{G}/(f_1, \ldots , f_c), \quad 
\overline{B} := \operatorname{Spec} k[x_1, \ldots, x_N]/(f_1, \ldots , f_c). 
\]
Suppose that $B$ is normal. 
Note that $\omega _{B} ^{[r]}$ is invertible since $\omega _A ^{[r]}$ is invertible. 
We define ideals $I$, $\widetilde{I}^{(\gamma)}$ and $\overline{I}^{(\gamma)}$ as 
\begin{align*}
I &:= (f_1,\ldots,f_c)\subset k[t][x_1, \ldots, x_N]^{G}, \\
\widetilde{I}^{(\gamma)} &:= \bigl( \lambda^* _{\gamma} (f_1),\ldots,\lambda^* _{\gamma} (f_c) \bigr) \subset k[t][x_1, \ldots, x_N]^{C_{\gamma}}, \\
\overline{I}^{(\gamma)} &:= \bigl( \overline{\lambda}^* _{\gamma} (f_1),\ldots,\overline{\lambda}^* _{\gamma} (f_c) \bigr) \subset k[t][x_1, \ldots, x_N]. 
\end{align*}
We define $k[t]$-schemes $B'$, $\widetilde{B}^{(\gamma)}$ and $\overline{B}^{(\gamma)}$ as follows
\begin{align*}
B' &= \operatorname{Spec}  k[t][x_1,\ldots,x_N]^G/I, \\
\widetilde{B}^{(\gamma)} &= \operatorname{Spec}  k[t][x_1,\ldots,x_N]^{C_{\gamma}}/ \widetilde{I}^{(\gamma)}, \\
\overline{B}^{(\gamma)} &= \operatorname{Spec}  k[t][x_1,\ldots,x_N]/ \overline{I}^{(\gamma)}. 
\end{align*}
Then we have the following diagram with the induced $k[t]$-morphisms $\mu_{\gamma} ^*$, $\overline{\mu}_{\gamma} ^*$ and $j$. 
\[
  \xymatrix{
 k[t][x_1, \ldots, x_N]^{G} \ar[r]_{\lambda _{\gamma} ^*} \ar@/^20pt/[rr]^{\overline{\lambda} _{\gamma} ^*} \ar@{->>}[d] &  k[t][x_1, \ldots, x_N]^{C_{\gamma}} \ar@{->>}[d] \ar[r]_i &  k[t][x_1, \ldots, x_N] \ar@{->>}[d] \\
 k[t][x_1, \ldots, x_N]^{G}/I \ar[r] ^{\mu _{\gamma} ^*} \ar@/_20pt/[rr]_{\overline{\mu} _{\gamma}^* }&  k[t][x_1, \ldots, x_N]^{C_{\gamma}} / \widetilde{I}^{(\gamma)} \ar[r]^j &  k[t][x_1, \ldots, x_N] / \overline{I}^{(\gamma)}
  }
\]
Furthermore, we have the following morphisms between the corresponding $k[t]$-schemes. 
\[
  \xymatrix{
A' & {A^{(\gamma)}}' \ar[l]^{\lambda _{\gamma}} & \overline{A}' \ar[l]^-q   \ar@/_18pt/[ll]_{\overline{\lambda} _{\gamma}} \\
B' \ar@{^{(}->}[u]^{\sigma} & \widetilde{B}^{(\gamma)} \ar@{^{(}->}[u] \ar[l]_{\mu _{\gamma}} & \overline{B}^{(\gamma)} \ar[l]_-p \ar@/^18pt/[ll]^{\overline{\mu} _{\gamma}} \ar@{^{(}->}[u]_{\tau}
  }
\]

First we give an easy observation on the intersection of $B$ and $Z$. 
\begin{lem}\label{lem:purity}
It follows that $\operatorname{codim}_B (Z \cap B) \ge 2$. 
In particular $\overline{B}\to B$ is  \'{e}tale in codimension one, and $\overline{B}$ is normal.
\end{lem}

\begin{proof}
Since $\operatorname{codim}_A (Z) \ge 2$, it follows that $A_{\rm sing} = Z$ by the purity of the branch locus (cf.\ \cite{Nag59}). 
Since $f_1, \ldots, f_{c} \in A$ is a regular sequence, it follows that $Z \cap B \subset B_{\rm sing}$ (cf.\ \cite[tag 00NU]{Sta}). 
By the normality of $B$,  we have $\operatorname{codim}_B (Z \cap B) \ge 2$. 
Therefore  $\overline{B} \to B$ is \'{e}tale in codimension one. 
Furthermore, we have $\operatorname{codim}_{\overline{B}} (\overline{B}_{\rm sing}) \ge 2$. 
Since $\overline{B}$ is Cohen-Macaulay, the normality follows from Serre's criterion. 
\end{proof}

\begin{rmk}\label{rmk:isotrivial}
Let $\overline{B}^{(\gamma)}_{t \not= 0}$ be the open subscheme of $\overline{B}^{(\gamma)}$ defined by $t \not = 0$. 
We shall see in this remark that we have a surjective \'{e}tale morphism 
\[
(\mathbb{A}^1 \setminus \{ 0 \}) \times  \overline{B} \to \overline{B}^{(\gamma)}_{t \not= 0}. 
\]

First, we have the following natural morphisms. 
\[
\xymatrix{
\overline{B}^{(\gamma)} = \operatorname{Spec} \left( k[t][x_1, \ldots , x_N]/ \bigl( \overline{\lambda}^* _{\gamma} (f_1),\ldots,\overline{\lambda}^* _{\gamma} (f_c) \bigr) \right) \\
\overline{B}^{(\gamma)}_{t \not= 0} := \operatorname{Spec} \left( k[t, t^{-1}][x_1, \ldots , x_N]/ \bigl( \overline{\lambda}^* _{\gamma} (f_1),\ldots,\overline{\lambda}^* _{\gamma} (f_c) \bigr) \right) \ar@{^{(}->}[u] \\
\operatorname{Spec} \left( k[t^{1/d}, t^{-1/d}][x_1, \ldots , x_N]/ \bigl( \overline{\lambda}^* _{\gamma} (f_1),\ldots,\overline{\lambda}^* _{\gamma} (f_c) \bigr) \right) \ar[u]_{\text{\'{e}tale}} \ar[d]^{\simeq} \\
(\mathbb{A}^1 \setminus \{ 0 \}) \times  \overline{B} \simeq \operatorname{Spec} \left( k[t^{1/d}, t^{-1/d}][x_1, \ldots , x_N]/(f_1, \ldots, f_c) \right) 
}
\]
Here the second morphism is \'{e}tale since it is induced by the \'{e}tale homomorphism $k[t, t^{-1}] \to k[t^{\frac{1}{d}}, t^{-\frac{1}{d}}]$. 
The third morphism, which is actually an isomorphism, is induced by the ring isomorphism
\[
k[t^{1/d}, t^{-1/d}][x_1, \ldots , x_N] \to k[t^{1/d}, t^{-1/d}][x_1, \ldots , x_N]; \quad x_i \mapsto t^{\frac{e_i}{d}} x_i, 
\]
whose inverse map is $x_i \mapsto t^{- \frac{e_i}{d}} x_i$. 
Therefore, we have a surjective \'{e}tale morphism 
\[
(\mathbb{A}^1 \setminus \{ 0 \}) \times  \overline{B} \to \overline{B}^{(\gamma)}_{t \not= 0}. 
\]
\end{rmk}

By Remark \ref{rmk:isotrivial}, the $k[t]$-scheme $\overline{B}^{(\gamma)}$ satisfies the condition $(\star \star)_n$ for $n := N - c$. 
Furthermore, $B'$ and $\widetilde{B}^{(\gamma)}$ also satisfy the condition $(\star \star)_n$. 

\begin{rmk}\label{rmk:nonlci}
Note that $\overline{B}^{(\gamma)}$ is not a complete intersection in $\overline{A}'$ in general because 
$\overline{\lambda}^* _{\gamma} (f_1),\ldots,\overline{\lambda}^* _{\gamma} (f_c)$ is not necessarily a regular sequence. 
In fact, $\overline{B}^{(\gamma)}$ is not pure dimensional in general (cf.\ Example \ref{eg:eg1}). Hence we do not have the standard definition of 
the canonical sheaf on $\overline{B}^{(\gamma)}$ (cf.\ Definition \ref{defi:nash_kt}). 
\end{rmk}
We define the following invertible sheaf $L_{\overline{B}^{(\gamma)}}$ on $\overline{B}^{(\gamma)}$ instead. 
\[
L_{\overline{B}^{(\gamma)}} := \overline{\mu} _{\gamma}^* \bigl( \operatorname{det} ^{-1} (I/I^2) \bigr) \otimes _{\mathcal{O}_{\overline{B}^{(\gamma)}}} \tau ^* \omega _{\overline{A}'/k[t]}. 
\]
Here $\operatorname{det} ^{-1} (I/I^2) = \bigl( \bigwedge ^{c} (I/I^2) \bigr)^*$ is an invertible sheaf on $B'$. 
Then $L_{\overline{B}^{(\gamma)}}$ fills the role of the canonical sheaf as follows. 
\begin{lem}\label{lem:L1}
There is a canonical morphism $\eta : \Omega ^n _{\overline{B}^{(\gamma)}/k[t]} \to L_{\overline{B}^{(\gamma)}}$ with the following conditions. 
\begin{enumerate}
\item It follows that $\operatorname{Im} (\eta) = \operatorname{Jac}_{\overline{B}^{(\gamma)}/k[t]} \otimes L_{\overline{B}^{(\gamma)}}$. 

\item There exists a following commutative diagram. 
\[
  \xymatrix{
 \overline{\mu} _{\gamma}^* ( \Omega^n _{B'/k[t]}) \ar[r] \ar[d] & \Omega^n _{\overline{B}^{(\gamma)}/k[t]} \ar[d] \\
 \overline{\mu} _{\gamma}^* \omega _{B'/k[t]} \ar[r] & L _{\overline{B}^{(\gamma)}}.
  }
\]
\end{enumerate}
\end{lem}
\begin{proof}
First, we shall see that there exists a canonical morphism 
\[
\overline{\mu} _{\gamma}^* \bigl( \operatorname{det} (I/I^2) \bigr) \otimes _{\mathcal{O}_{\overline{B}^{(\gamma)}}} \Omega ^n _{\overline{B}^{(\gamma)}/k[t]} 
\to \tau ^* \omega _{\overline{A}'/k[t]}
\] 
whose image is $\operatorname{Jac}_{\overline{B}^{(\gamma)}/k[t]} \otimes \tau ^* \omega _{\overline{A}'/k[t]}$.
We set 
\[
R := k[t][x_1, \ldots , x_N], \quad \overline{S}:= k[t][x_1, \ldots , x_N] / \overline{I}^{(\gamma)}, \quad S:= k[t][x_1, \ldots , x_N]^G / I. 
\]
Let $N \subset \Omega _{R/k[t]}$ be the $R$-submodule which is generated by $df$ for $f \in \overline{I}^{(\gamma)}$. 
Then we have $\Omega _{\overline{S}/k[t]} \simeq (\Omega _{R/k[t]}/N) \otimes _{R} \overline{S}$. 
Since $\overline{I}^{(\gamma)}$ is generated by $c$ elements, we have $\bigwedge ^{c+1} N = 0$. 
Hence a canonical map $\bigwedge ^c N \otimes _R  \bigwedge ^{N-c} \Omega _{R/k[t]} \to \Omega ^N _{R/k[t]}$ induces 
\[
\bigwedge ^c N \otimes _R \bigwedge ^{N-c} (\Omega _{R/k[t]}/N) \to \Omega ^N _{R/k[t]}. 
\]
By taking $\mathchar`- \otimes _R \overline{S}$ and composing with 
\[
I/I^2 \otimes _S \overline{S} \to \overline{I}^{(\gamma)}/(\overline{I}^{(\gamma)})^2 \xrightarrow{\overline{h} \mapsto d(h) \otimes 1} N \otimes _R \overline{S}, 
\]
we have a canonical map
\[
\Bigl( \bigwedge ^{c} (I/I^2) \otimes _S \overline{S} \Bigr) \otimes _{\overline{S}} \Omega^{N-c} _{\overline{S}/k[t]} \to \Omega ^N _{R/k[t]} \otimes _{R} \overline{S}. 
\]
The $\overline{S}$-module $\bigl( \bigwedge ^{c} (I/I^2) \otimes _S \overline{S} \bigr) \otimes _{\overline{S}} \Omega^{N-c} _{\overline{S}/k[t]}$ is generated 
by the elements of the form 
\[
(\overline{f}_1 \wedge \cdots \wedge \overline{f}_c) \otimes (dx_{i_1} \wedge \cdots \wedge dx_{i_{N-c}}), 
\]
and its image in $\Omega ^N _{R/k[t]} \otimes _{R} \overline{S}$ is 
\[
d \bigl( \overline{\lambda} _{\gamma}^* (f_1) \bigr) \wedge \cdots \wedge d \bigl( \overline{\lambda} _{\gamma}^* (f_c) \bigr) \wedge dx_{i_1} \wedge \cdots \wedge dx_{i_{N-c}}. 
\]
We note that this is equal to 
\[
\pm \Delta \cdot dx_1 \wedge \cdots \wedge dx_N, 
\]
where $\Delta$ is the determinant of the Jacobian matrix respect to 
$\overline{\lambda} _{\gamma}^* (f_i)$ with $1 \le i \le c$ and $\partial x_j$ with $j \in \{1, \ldots , N \} \setminus \{ i_1, \ldots i_{N-c} \}$. 
Hence we have a canonical morphism 
$\overline{\mu} _{\gamma}^* \bigl( \operatorname{det} (I/I^2) \bigr) \otimes _{\mathcal{O}_{\overline{B}^{(\gamma)}}} \Omega ^n _{\overline{B}^{(\gamma)}/k[t]} 
\to \tau ^* \omega _{\overline{A}'/k[t]}$ whose image is $\operatorname{Jac}_{\overline{B}^{(\gamma)}/k[t]} \otimes \tau ^* \omega _{\overline{A}'/k[t]}$. 
We have proved (1). 

For (2), we note that 
a canonical map $\overline{\mu} _{\gamma}^* \omega _{B'/k[t]} \to L _{\overline{B}^{(\gamma)}}$ is induced by 
the isomorphism
\[
\omega _{B'/k[t]} \simeq \operatorname{det}^{-1} ( I / I^2 ) 
\otimes _{\mathcal{O}_{B'}} \sigma ^* \omega _{A'/k[t]} 
\]
obtained by adjunction. Then the commutativity of the diagram is obvious. 
\end{proof}

In the same way as in Definition \ref{defi:nash_kt}, we denote by $\mathfrak{n}'_{1,p}$ and $\mathfrak{n}'_{1,\overline{\mu} _{\gamma}}$ the ideal sheaves on $\overline{B}^{(\gamma)}$ satisfying
\[
\operatorname{Im} \bigl( p^* \Omega^n _{\widetilde{B}^{(\gamma)}/k[t]} \to L_{\overline{B}^{(\gamma)}} \bigr) = \mathfrak{n}'_{1,p} \otimes L_{\overline{B}^{(\gamma)}}, \quad 
\operatorname{Im} \bigl( \overline{\mu} _{\gamma}^* \Omega^n _{B'/k[t]} \to L_{\overline{B}^{(\gamma)}} \bigr) = \mathfrak{n}'_{1,\overline{\mu} _{\gamma}} \otimes L_{\overline{B}^{(\gamma)}}. 
\]
We prove Lemma \ref{lem:L2} and Lemma \ref{lem:age2} on relations on orders which will be used in the proof of Theorem \ref{thm:mld_hyperquot}. 
\begin{lem}\label{lem:L2}
Let $\alpha \in \overline{B}^{(\gamma)} _{\infty}$ be an arc with 
$\operatorname{ord}_{\alpha} \bigl( \operatorname{Jac}_{\overline{B}^{(\gamma)}/k[t]} \bigr) < \infty$. 
Then the following hold. 
\begin{enumerate}
\item $\operatorname{ord}_{\alpha} (\operatorname{jac}_p) + \operatorname{ord}_{\alpha} (\operatorname{Jac}_{\overline{B}^{(\gamma)}/k[t]})
= \operatorname{ord}_{\alpha} (\mathfrak{n}'_{1,p})$. 
\item $\operatorname{ord}_{\alpha} (\operatorname{jac}_{\overline{\mu} _{\gamma}}) + \operatorname{ord}_{\alpha} (\operatorname{Jac}_{\overline{B}^{(\gamma)}/k[t]})
= \operatorname{ord}_{\alpha} (\mathfrak{n}'_{1,\overline{\mu} _{\gamma}})$. 
\end{enumerate}
\end{lem}
\begin{proof}
The same proof of Lemma \ref{lem:order}(2) works by Lemma \ref{lem:L1}(1). 
\end{proof}

\begin{lem}\label{lem:age2}
Let $\alpha \in \overline{B}^{(\gamma)}_{\infty}$ be an arc. 
Set $\alpha ' := \overline{\mu} _{\gamma \infty} (\alpha)$. 
Suppose that $\alpha ' \not \in Z_{\infty}$. 
Then it follows that 
\[
\operatorname{ord}_{\alpha} (\mathfrak{n}'_{1,\overline{\mu}_{\gamma}}) = \frac{1}{r} \operatorname{ord}_{\alpha '} (\mathfrak{n}_{r,B'})
+ 
\operatorname{age}(\gamma).
\]
\end{lem}
\begin{proof}
We have two diagrams (cf.\ Lemma \ref{lem:L1}(2)). 
\[
  \xymatrix{
A' & \overline{A}' \ar[l]_-{\overline{\lambda} _{\gamma}} && \overline{\mu} _{\gamma}^* ( \Omega^n _{B'/k[t]})^{\otimes r} \ar[r] \ar[d] & (\Omega^n _{\overline{B}^{(\gamma)}/k[t]} )^{\otimes r} \ar[d] \\
B' \ar@{^{(}->}[u]^{\sigma} & \overline{B}^{(\gamma)}, \ar[l]^-{\overline{\mu} _{\gamma}} \ar@{^{(}->}[u]_{\tau} && \overline{\mu} _{\gamma}^* \omega _{B'/k[t]}^{[r]} \ar[r] & L _{\overline{B}^{(\gamma)}}^{[r]}.
  }
\]
We denote by $\mathfrak{r}_{\overline{B}^{(\gamma)}}$ the ideal sheaf on $\overline{B}^{(\gamma)}$ satisfying
\[
\operatorname{Im} \Bigl( \overline{\mu} _{\gamma}^* \omega _{B'/k[t]}^{[r]} \to L _{\overline{B}^{(\gamma)}}^{[r]} \Bigr) = \mathfrak{r}_{\overline{B}^{(\gamma)}} \otimes L _{\overline{B}^{(\gamma)}}^{[r]}. 
\]
We also denote by $\mathfrak{r}_{\overline{A}'}$ the ideal sheaf on $\overline{A}'$ satisfying
\[
\operatorname{Im} \Bigl( \overline{\lambda} _{\gamma}^* \omega _{A'/k[t]}^{[r]} \to \omega _{\overline{A}'/k[t]}^{[r]} \Bigr) = \mathfrak{r}_{\overline{A}'} \otimes \omega _{\overline{A}'/k[t]}^{[r]}. 
\]
By definition of the Nash ideals (Definition \ref{defi:nash_kt}), 
we have 
\[
r \operatorname{ord}_{\alpha} (\mathfrak{n}'_{1,\overline{\mu}_{\gamma}}) 
= \operatorname{ord}_{\alpha} (\mathfrak{r}_{\overline{B}^{(\gamma)}}) + \operatorname{ord}_{\alpha '} (\mathfrak{n}_{r,B'}). 
\]
On the other hand, by Lemma \ref{lem:age1}, we have 
\begin{align*}
r \operatorname{age} (\gamma) 
&= r \operatorname{ord}_{\alpha} (\operatorname{jac}_{\overline{\lambda}_{\gamma}}) - \operatorname{ord}_{\alpha '} (\mathfrak{n}_{r,A'}) \\
&= r \operatorname{ord}_{\alpha} (\mathfrak{n}_{1,\overline{\lambda}_{\gamma}}) - \operatorname{ord}_{\alpha '} (\mathfrak{n}_{r,A'}) \\
&= \operatorname{ord}_{\alpha} (\mathfrak{r}_{\overline{A}'}). 
\end{align*}
Here the second equality follows from Lemma \ref{lem:order}(2) and the fact $\mathfrak{n}_{1,\overline{A}'} = \mathcal{O}_{\overline{A}'}$. 
The third equality follows from the definition of 
the ideals $\mathfrak{n}_{1,\overline{\lambda}_{\gamma}}$, $\mathfrak{n}_{r,A'}$, and $\mathfrak{r}_{\overline{A}'}$, 
and the fact $\operatorname{ord}_{\alpha '} (\mathfrak{n}_{r,A'}) < \infty$, which follows from the assumption $\alpha ' \not \in Z_{\infty}$. 
Hence it is sufficient to show that $\operatorname{ord}_{\alpha} (\mathfrak{r}_{\overline{B}^{(\gamma)}}) = \operatorname{ord}_{\alpha} (\mathfrak{r}_{\overline{A}'})$. 

Since 
\[
\omega _{B'/k[t]} \simeq \operatorname{det}^{-1} ( I / I^2 ) 
\otimes _{\mathcal{O}_{B'}} \sigma ^* \omega _{A'/k[t]} 
\]
holds by the adjunction formula, we have $\mathfrak{r}_{\overline{B}^{(\gamma)}} = \mathfrak{r}_{\overline{A}'} \mathcal{O}_{\overline{B}^{(\gamma)}}$ by definition of $L _{\overline{B}^{(\gamma)}}$.
We complete the proof. 
\end{proof}

\subsection{Minimal log discrepancies of hyperquotient singularities via jet schemes}
In this subsection, we investigate the minimal log discrepancies of hyperquotient singularities in terms of the arc spaces. 
We take over the notation from Subsection \ref{subsection:arc_hyperquot}. 

\begin{thm}\label{thm:mld_hyperquot}
Let $x = 0 \in B$ be the origin, and 
let $\mathfrak{a} \subset  k[x_1,\ldots,x_N]^G/I$ be a non-zero ideal and let $\delta$ be a non-negative real number. 
Then 
\[
\operatorname{mld}_x(B,\mathfrak{a}^\delta) = 
\inf _{w, b_1 \in \mathbb{Z}_{\ge 0}, \gamma \in G}
\bigl \{ 
\operatorname{codim} (C_{w, \gamma, b_1}) + \operatorname{age}(\gamma)  - b_1-\delta w
\bigr \}
\]
holds for 
\[
C_{w, \gamma, b_1} = 
\operatorname{Cont}^w  (\mathfrak{a} \mathcal{O}_{\overline{B}^{(\gamma)}}) \cap 
\operatorname{Cont}^{\ge 1} (\mathfrak{m}_x \mathcal{O}_{\overline{B}^{(\gamma)}})\cap \operatorname{Cont}^{b_1}(\operatorname{Jac}_{\overline{B}^{(\gamma)}/k[t]}), 
\]
where $\mathfrak{m}_x \subset \mathcal{O}_B$ is the maximal ideal corresponding to the closed point $x \in B$.
\end{thm}

\begin{proof}
By \cite[Theorem 7.4]{EM09}, we have
\[
\operatorname{mld}_x(B,\mathfrak{a}^\delta)=
\inf _{w, m_3 \in \mathbb{Z}_{\ge 0}} 
\left\{\hspace{-1mm}  \begin{array}{c} 
\operatorname{codim}\left( \operatorname{Cont}^w(\mathfrak a) \cap \operatorname{Cont}^{m_3}(\mathfrak n_{r,B}) 
	\cap \operatorname{Cont}^{\ge 1}(\mathfrak m_x) \right) \\
-\frac{m_3}{r}-\delta w
\end{array}
\right\}.
\]
Note that  $\mathfrak{n}_{r, B'} = \mathfrak{n}_{r,B} \mathcal{O}_{B'}$ holds by Remark \ref{rmk:nash_kt} and 
$B'_m = B_m$ holds for $m\in \mathbb Z_{\ge 0} \cup \{ \infty \}$ by Remark \ref{rmk:bc}.

Let $w,b_1,b_2,b_3,m_1,m_2,m_3 \in \mathbb{Z}_{\ge 0}$ and $\gamma \in G$. 
We denote by $D_{w,\gamma,b_1,b_2,b_3,m_1,m_2,m_3} \subset \overline{B}^{(\gamma)}_{\infty}$ the cylinder
\begin{align*}
\operatorname{Cont}^w ( \mathfrak{a}\mathcal{O}_{\overline{B}^{(\gamma)}} ) \cap 
\operatorname{Cont}^{\ge 1} ( \mathfrak{m}_x \mathcal{O}_{\overline{B}^{(\gamma)}}) 
\cap \operatorname{Cont}^{b_1}(\operatorname{Jac}_{\overline{B}^{(\gamma)}/k[t]}) 
\cap \operatorname{Cont}^{b_2}( \operatorname{Jac}_{\widetilde{B}^{(\gamma)}/k[t]} \mathcal{O}_{\overline{B}^{(\gamma)}}) \\
\cap \operatorname{Cont}^{b_3}  (\operatorname{Jac}_{B'/k[t]} \mathcal{O}_{\overline{B}^{(\gamma)}})
\cap \operatorname{Cont}^{m_1}(\mathfrak{n}'_{1,p})
\cap \operatorname{Cont}^{m_2}(\mathfrak{n}'_{1,\overline{\mu}_{\gamma}})
\cap \operatorname{Cont}^{m_3}(\mathfrak{n}_{r,B'} \mathcal{O}_{\overline{B}^{(\gamma)}}).
\end{align*}

By Propositions \ref{prop:lift} and \ref{prop:hyperquot}, we have
\begin{align*}
&\operatorname{Cont}^w(\mathfrak a) \cap 
\operatorname{Cont}^{m_3}(\mathfrak n_{r,B}) \cap 
\operatorname{Cont}^{\ge 1}(\mathfrak m_x) \setminus Z_\infty \\
&{}= \bigsqcup _{\langle \gamma \rangle \in \operatorname{Conj}(G)} 
\mu_{\gamma \infty} \circ p_{\infty} \bigl( \operatorname{Cont}^w (\mathfrak{a} \mathcal{O}_{\overline{B}^{(\gamma)}}) \cap 
\operatorname{Cont}^{m_3} ( \mathfrak{n}_{r,B} \mathcal{O}_{\overline{B}^{(\gamma)}}) \cap 
\operatorname{Cont}^{\ge 1} (\mathfrak{m}_x \mathcal{O}_{\overline{B}^{(\gamma)}}) \bigr) \setminus Z_\infty. 
\end{align*}
Note that $\mu_{\gamma \infty} \circ p_{\infty}(C)$ is a thin set of $B'_{\infty}$ 
for any thin set $C$ of $\overline{B}^{(\gamma)}_{\infty}$.
Hence we have 
\begin{align*}
& \operatorname{codim} \left( \operatorname{Cont}^w(\mathfrak a) \cap 
\operatorname{Cont}^{m_3}(\mathfrak n_{r,B}) \cap 
\operatorname{Cont}^{\ge 1}(\mathfrak m_x) \right) \\
&{}=
\min_{\gamma,b_1,b_2,b_3,m_1,m_2} \operatorname{codim} \bigl( \mu_{\gamma \infty} \circ p_{\infty} ( D_{w,\gamma,b_1,b_2,b_3,m_1,m_2,m_3} ) \bigr) 
\end{align*}
by Proposition \ref{prop:negligible}. 
On the other hand, again by Proposition \ref{prop:negligible}, we have 
\[
\operatorname{codim} (C_{w, \gamma, b_1}) = 
\min_{b_2,b_3,m_1,m_2,m_3} \operatorname{codim} \bigl( D_{w,\gamma,b_1,b_2,b_3,m_1,m_2,m_3} \bigr).
\]

By Lemma \ref{lem:L2}, we have 
\[
\operatorname{ord}_{\alpha} (\operatorname{jac}_p) = m_1 - b_1
\]
for any $\alpha \in D_{w, \gamma, b_1, b_2, b_3,m_1,m_2,m_3}$. 
Furthermore, $\operatorname{ord}(\operatorname{Jac}_{\overline{B}^{(\gamma)}/k[t]})$ and $\operatorname{ord}(\operatorname{Jac}_{\widetilde{B}^{(\gamma)}/k[t]})$ 
take constant values $b_1$ and $b_2$ on $D_{w,\gamma, b_1, b_2, b_3,m_1,m_2,m_3}$ and $p_{\infty}(D_{w,\gamma, b_1, b_2, b_3,m_1,m_2,m_3})$, respectively. 
Hence by applying Proposition \ref{prop:DL2} to $p$, we have 
\[
\operatorname{codim} \bigl( p_{\infty}(D_{w,\gamma, b_1, b_2, b_3,m_1,m_2,m_3}) \bigr) = 
\operatorname{codim} \bigl( D_{w,\gamma, b_1, b_2, b_3,m_1,m_2,m_3} \bigr) + m_1 - b_1. 
\]

Note that $\operatorname{codim}_B (B \cap Z) \ge 2$ by Lemma \ref{lem:purity} and that $\mu_{\gamma \infty}|_{p_{\infty} (D_{w,\gamma, b_1, b_2, b_3,m_1,m_2,m_3})}$ 
is injective outside $Z_{\infty}$ by Proposition \ref{prop:hyperquot}. 
By Lemma \ref{lem:additive} and Lemma \ref{lem:L2}, we have 
\[
\operatorname{ord}_{\alpha} (\operatorname{jac}_{\mu _{\gamma}}) = (m_2 - b_1) - (m_1 - b_1) = m_2 - m_1
\]
for any $\alpha \in p_{\infty}(D_{w,\gamma, b_1, b_2, b_3,m_1,m_2,m_3})$. 
Furthermore, $\operatorname{ord}(\operatorname{Jac}_{B'/k[t]})$ takes a constant value $b_3$ on 
$(\mu _{\gamma} \circ p)_{\infty} (D_{w,\gamma, b_1, b_2, b_3,m_1,m_2,m_3})$. 
Hence by applying Proposition \ref{prop:EM6.2} to $\mu _{\gamma}$, we have 
\begin{align*}
&\operatorname{codim} \bigl( (\mu _{\gamma} \circ p)_{\infty}(D_{w,\gamma, b_1, b_2, b_3,m_1,m_2,m_3}) \bigr) \\
&{}= \operatorname{codim} \bigl( p_{\infty} (D_{w,\gamma, b_1, b_2, b_3,m_1,m_2,m_3}) \bigr) + m_2 - m_1. 
\end{align*}

By Lemma \ref{lem:age2}, we have 
\begin{align*}
m_2 - \frac{m_3}{r} = \operatorname{age}(\gamma) 
\end{align*}
if $(\mu _{\gamma} \circ p)_{\infty} (D_{w, \gamma, b_1, b_2, b_3,m_1,m_2,m_3}) \setminus Z_{\infty}$ is non-empty. 
Thus we obtain 
\begin{align*}
&\operatorname{codim} (D_{w, \gamma, b_1, b_2, b_3,m_1,m_2,m_3}) + \operatorname{age}(\gamma)  - b_1 \\
&=
\operatorname{codim} \bigl( (\mu _{\gamma} \circ p)_{\infty} (D_{w, \gamma, b_1, b_2, b_3,m_1,m_2,m_3}) \bigr) - \frac{m_3}{r}. 
\end{align*}

Therefore we obtain the desired formula: 
\begin{align*}
&\operatorname{mld}_x(B,\mathfrak{a}^\delta)\\
&=
\inf _{w, m_3} \left\{ 
\operatorname{codim}\left( \operatorname{Cont}^w(\mathfrak a) \cap \operatorname{Cont}^{m_3}(\mathfrak n_{r,B})
\cap \operatorname{Cont}^{\ge 1}(\mathfrak m_x) \right)
-\frac{m_3}{r}-\delta w
 \right\} \\
&= \inf _{\gamma, w, b_1, b_2, b_3,m_1,m_2,m_3}
\left\{ 
\operatorname{codim}\bigl( (\mu _{\gamma} \circ p)_{\infty} (D_{w,\gamma, b_1, b_2, b_3,m_1,m_2,m_3}) \bigr) 
-\frac{m_3}{r}-\delta w
 \right\} \\
&= \inf _{\gamma, w, b_1, b_2, b_3,m_1,m_2,m_3}
\bigl \{ 
\operatorname{codim}( D_{w,\gamma, b_1, b_2, b_3,m_1,m_2,m_3} ) 
+ \operatorname{age}(\gamma)  - b_1-\delta w
 \bigr \} \\ 
&=\inf _{\gamma, w, b_1}
\bigl \{
\operatorname{codim} (C_{w, \gamma, b_1}) + \operatorname{age}(\gamma)  - b_1-\delta w
\bigr \}
\end{align*}
We complete the proof. 
\end{proof}

\begin{cor}\label{cor:mld_hyperquot2}
In the same setting as in Theorem \ref{thm:mld_hyperquot}, it follows that
\[
\operatorname{mld}_x(B,\mathfrak{a}^\delta) = 
\inf _{w, b_1 \in \mathbb{Z}_{\ge 0}, \gamma \in G}
\bigl \{ 
\operatorname{codim} (C'_{w, \gamma, b_1}) + \operatorname{age}(\gamma)  - b_1-\delta w
\bigr \}
\]
for 
\[
C' _{w, \gamma, b_1} = 
\operatorname{Cont}^{\ge w}  (\mathfrak{a} \mathcal{O}_{\overline{B}^{(\gamma)}}) \cap 
\operatorname{Cont}^{\ge 1} (\mathfrak{m}_x \mathcal{O}_{\overline{B}^{(\gamma)}}) \cap 
\operatorname{Cont}^{b_1}(\operatorname{Jac}_{\overline{B}^{(\gamma)}/k[t]}). 
\]
\end{cor}
\begin{proof}
Let $C _{w, \gamma, b_1}$ be the cylinder in Theorem \ref{thm:mld_hyperquot}. 
We fix $\gamma \in G$ and $b_1 \in \mathbb{Z}_{\ge 0}$. 
Since we have $C_{w, \gamma, b_1} \subset C'_{w, \gamma, b_1}$, it follows that 
\[
\inf _w \left \{ \operatorname{codim} \left( C_{w, \gamma, b_1} \right) - \delta w \right \} \ge 
\inf _w \left \{ \operatorname{codim} \left( C'_{w, \gamma, b_1} \right) - \delta w \right \}. 
\]
We fix $w' \in \mathbb{Z}_{\ge 0}$. 
Then it follows that
\[
\operatorname{codim} \left( C' _{w', \gamma, b_1} \right) - \delta w' 
= \min_{w \ge w'} \{ \operatorname{codim} \left( C_{w, \gamma, b_1} \right) \} - \delta w'
\ge \inf _w \left \{ \operatorname{codim} \left( C_{w, \gamma, b_1} \right) - \delta w \right \}. 
\]
Here the first equality follows from Proposition \ref{prop:negligible}, and the last inequality follows from $\delta \ge 0$. 
Therefore we have the opposite inequality
\[
\inf _w \left \{ \operatorname{codim} \left( C_{w, \gamma, b_1} \right) - \delta w \right \} \le 
\inf _w \left \{ \operatorname{codim} \left( C'_{w, \gamma, b_1} \right) - \delta w \right \}. 
\]
We complete the proof. 
\end{proof}

\begin{rmk}\label{rmk:multi_index}
Our formula can be easily extended to $\mathbb{R}$-ideals $\mathfrak{a} = \prod _{i=1} ^r \mathfrak{a}_i ^{\delta_i}$, 
where $\mathfrak a_1, \ldots, \mathfrak a_r$ are ideals and $\delta_1, \ldots ,\delta_r$ are non-negative real numbers. 
In this setting, we have 
\begin{align*}
&\operatorname{mld}_x \bigl( B,\prod _{i=1} ^r \mathfrak{a}_i ^{\delta_i} \bigr) \\
&= 
\inf _{w_1, \ldots, w_r, b_1 \in \mathbb{Z}_{\ge 0}, \gamma \in G}
\left \{ 
\operatorname{codim} (C_{w_1, \ldots , w_r, \gamma, b_1}) + \operatorname{age}(\gamma)  - b_1 - \sum _{i=1} ^r \delta_i w_i
\right \} \\
&=
\inf _{w_1, \ldots, w_r, b_1 \in \mathbb{Z}_{\ge 0}, \gamma \in G}
\left \{ 
\operatorname{codim} (C'_{w_1, \ldots , w_r, \gamma, b_1}) + \operatorname{age}(\gamma)  - b_1 - \sum _{i=1} ^r \delta_i w_i
\right \}
\end{align*}
for 
\begin{align*}
C_{w_1, \ldots , w_r, \gamma, b_1} &= 
\Bigl( \bigcap _{i=1} ^{r} \operatorname{Cont}^{w_i}  (\mathfrak{a}_i \mathcal{O}_{\overline{B}^{(\gamma)}}) \Bigr) \cap 
\operatorname{Cont}^{\ge 1} (\mathfrak{m}_x \mathcal{O}_{\overline{B}^{(\gamma)}}) \cap \operatorname{Cont}^{b_1}(\operatorname{Jac}_{\overline{B}^{(\gamma)}/k[t]}), \\
C' _{w_1, \ldots , w_r, \gamma, b_1} &= 
\Bigl( \bigcap _{i=1} ^{r} \operatorname{Cont}^{\ge w_i}  (\mathfrak{a}_i \mathcal{O}_{\overline{B}^{(\gamma)}}) \Bigr) \cap 
\operatorname{Cont}^{\ge 1} (\mathfrak{m}_x \mathcal{O}_{\overline{B}^{(\gamma)}}) \cap \operatorname{Cont}^{b_1}(\operatorname{Jac}_{\overline{B}^{(\gamma)}/k[t]}). 
\end{align*}
\end{rmk}

We state Theorem \ref{thm:mld_hyperquot} in the case of quotient singularities. 

\begin{cor}\label{cor:mld_quot}
Let $x = 0 \in A$ be the origin, and 
let $\mathfrak{a} \subset  k[x_1,\ldots,x_N]^G$ be a non-zero ideal and $\delta$ be a non-negative real number. 
Then 
\[
\operatorname{mld}_x(A, \mathfrak{a}^\delta) = 
\inf _{w \in \mathbb{Z}_{\ge 0}, \gamma \in G}
\bigl \{ 
\operatorname{codim} (C_{w, \gamma}) + \operatorname{age}(\gamma) - \delta w
\bigr \}
\]
holds for 
\[
C_{w, \gamma} = 
\operatorname{Cont}^w  (\mathfrak{a} \mathcal{O}_{\overline{A}'}) \cap 
\operatorname{Cont}^{\ge 1} (\mathfrak{m}_x \mathcal{O}_{\overline{A}'}), 
\]
where $\mathfrak{m}_x \subset \mathcal{O}_A$ is the maximal ideal corresponding to the closed point $x \in A$. 
Furthermore, the same statements of Corollary \ref{cor:mld_hyperquot2} and Remark \ref{rmk:multi_index} also hold. 
\end{cor}
\begin{proof}
Since $\overline{A}'$ is smooth over $k[t]$, we have $\operatorname{Jac}_{\overline{A}'/k[t]} = \mathcal{O}_{\overline{A}'}$. 
Therefore the assertion follows from Theorem \ref{thm:mld_hyperquot}. 
\end{proof}

\noindent
As a corollary, we obtain the following Reid-Tai type formula on minimal log discrepancies. 
See also Remark \ref{rmk:RT} for a proof using the resolution of singularities.

\begin{cor}[{cf.\ \cite[Question 2]{Bor97}}]\label{cor:qmld=cqmld}
For $\gamma \in G$, we denote by $\langle \gamma \rangle$ the subgroup of $G$ generated by $\gamma$. 
Let $A^{\langle \gamma \rangle} = \mathbb{A}^N/\langle \gamma \rangle$, and 
let $x_{\gamma} \in A^{\langle \gamma \rangle}$ be the image of the origin of $\mathbb{A}^N$. 
Let $\mathfrak{a}$ be an $\mathbb{R}$-ideal sheaf on $A$. 
Then it follows that
\[
\operatorname{mld}_x (A, \mathfrak{a}) = 
\min_{\gamma \in G} \operatorname{mld}_{x_\gamma}
\bigl( A^{\langle \gamma \rangle}, \mathfrak{a} \mathcal{O}_{A^{\langle \gamma \rangle}} \bigr). 
\]
\end{cor}
\begin{proof}
We fix $\gamma \in G$ and $\gamma ' \in \langle \gamma \rangle$. 
Then we have a $k[t]$-ring homomorphism
\[
\overline{\lambda}_{\gamma '}^*: k[t][x_1, \ldots , x_N]^G \longrightarrow k[t][x_1,\dots,x_N]
\]
applying the explanation in Subsection \ref{subsection:arc_quot} to $G$ and $\gamma '$. 
We also have a $k[t]$-ring homomorphism
\[
\overline{\lambda}_{\gamma '}^*: k[t][x_1, \ldots , x_N]^{\langle \gamma \rangle} \longrightarrow k[t][x_1,\dots,x_N]
\]
for $\langle \gamma \rangle$ and $\gamma '$, where we use the same symbol 
$\overline{\lambda}_{\gamma '}^*$ by abuse of notation. 
We note by definition that these two maps are compatible with the inclusion
\[
k[t][x_1, \ldots , x_N]^G \hookrightarrow k[t][x_1, \ldots , x_N]^{\langle \gamma \rangle}, 
\]
which is induced by $k[x_1, \ldots , x_N]^G \hookrightarrow k[x_1, \ldots , x_N]^{\langle \gamma \rangle}$. 

Let 
\[
\mathfrak{m}_x \subset k[x_1,\dots,x_N]^G = \mathcal{O}_A, \quad
\mathfrak{m}_{x_{\gamma}} \subset k[x_1, \ldots , x_N]^{\langle \gamma \rangle} = \mathcal{O}_{A^{\langle \gamma \rangle}}
\] 
be the maximal ideals corresponding to $x \in A$ and $x_{\gamma} \in A^{\langle \gamma \rangle}$, respectively.
Since we have 
\[
\sqrt{\mathfrak{m}_x \mathcal{O}_{A^{\langle \gamma \rangle}}} = \mathfrak{m}_{x_{\gamma}}, 
\]
it follows that
\[
\operatorname{Cont}^{\ge 1} (\mathfrak{m}_x \mathcal{O}_{\overline{A}'})= \operatorname{Cont}^{\ge 1} (\mathfrak{m}_{x_\gamma} \mathcal{O}_{\overline{A}'}).
\]
Therefore $C_{w, \gamma'}$ defined in Corollary \ref{cor:mld_quot} for $A$, $\gamma' \in G$ and $\mathfrak{a}$ 
is exactly same as $C_{w, \gamma'}$ defined for $A^{\langle \gamma \rangle}$, 
$\gamma ' \in \langle \gamma \rangle$ and $\mathfrak{a} \mathcal{O}_{A^{\langle \gamma \rangle}}$. 
Therefore the assertion 
\[
\operatorname{mld}_x (A, \mathfrak{a}) = 
\min_{\gamma \in G} \operatorname{mld}_{x_\gamma}
\bigl( A^{\langle \gamma \rangle}, \mathfrak{a} \mathcal{O}_{A^{\langle \gamma \rangle}} \bigr)
\]
follows from Corollary \ref{cor:mld_quot}. 
\end{proof}

\begin{rmk}\label{rmk:RT}
We can prove more general statement following the argument in \cite{Rei80} (cf.\ \cite[2.42, Theorem 3.21]{Kol13}). 
\begin{itemize}
\item Let $V$ be a $\mathbb{Q}$-Gorenstein normal variety and let $G$ be a finite group acting on $V$. 
Let $x \in V$ be a closed point and let $x' \in V/G$ be its image. 
Let $V^{\langle \gamma \rangle} = V/\langle \gamma \rangle$, and 
let $x_{\gamma} \in V^{\langle \gamma \rangle}$ be the image of $x$. 
Let $\mathfrak{a}$ be an $\mathbb{R}$-ideal sheaf on $V/G$. 
Then it follows that
\[
\operatorname{mld}_{x'} (V/G, \mathfrak{a}) = 
\min_{\gamma \in G} \operatorname{mld}_{x_\gamma}
\bigl( V^{\langle \gamma \rangle}, \mathfrak{a} \mathcal{O}_{V^{\langle \gamma \rangle}} \bigr). 
\]
\end{itemize}
By \cite[2.42.4]{Kol13}, the inequality 
\[
\operatorname{mld}_{x'} (V/G, \mathfrak{a}) \le 
\min_{\gamma \in G} \operatorname{mld}_{x_\gamma}
\bigl( V^{\langle \gamma \rangle}, \mathfrak{a} \mathcal{O}_{V^{\langle \gamma \rangle}} \bigr)
\]
is easy. 
Let $X' \to X := V/G$ be a resolution of singularities. 
We may assume that some divisor $E$ on $X'$ computes $\operatorname{mld}_{x'}(X, \mathfrak{a})$, that is, 
$E$ satisfies 
\begin{itemize} 
\item $c_X(E)=\{x'\}$ and 
\item $\operatorname{mld}_{x'}(X, \mathfrak{a}) = a_E(X, \mathfrak{a})$ if $\operatorname{mld}_{x'}(X, \mathfrak{a}) \ge 0$, and 
$a_E(X, \mathfrak{a}) < 0$ otherwise. 
\end{itemize}
Let $V'$ be the normalization of $X'$ in the field of fraction $k(V)$ of $V$. 
Let $F$ be a divisor on $V'$ which dominates $E$. 
We note that $G$ acts on $V'$ and we have $X' = V'/G$. 
Let $G_F$ be the subgroup of $G$ that consists of an element $g \in G$ which fixes $F$ point-wise. 
Then by \cite[2.42.4]{Kol13} we have 
\[
a_E (V/G, \mathfrak{a}) = a_{F'} (V/G_F, \mathfrak{a} \mathcal{O}_{V/G_F}), 
\]
where $F'$ is the image of $F$ on $V/G_F$. 
Since $G_F$ is a cyclic group (cf.\ \cite[2.5, 2.6]{IR96}), we get the opposite inequality. 
\end{rmk}

\noindent
As an application of Corollary \ref{cor:qmld=cqmld}, we can prove the ACC conjecture for quotient singularities (cf.\ \cite{Mor}). 

\begin{thm}\label{thm:acc_quot}
Let $n$ be a positive integer.
The set
\[
A_{\rm quot}(n) := 
\{ \operatorname{mld}_x(X) \mid \text{$X$ has a quotient singularity at $x$ and $\dim X= n$} \}
\]
satisfies the ascending chain condition.
\end{thm}
\begin{proof}
By Corollary \ref{cor:qmld=cqmld} (or Remark \ref{rmk:RT}), the set $A_{\rm quot}(n)$ is equal to 
\[
A_{\rm cquot}(n):= \{ \operatorname{mld}_x(X) \mid \text{$X$ has a cyclic quotient singularity at $x$ and $\dim X= n$} \}.
\]
Since a cyclic quotient singularity is a toric singularity (cf.\ \cite[(4.3)]{Rei85}), 
$A_{\rm cquot}$ satisfies the ascending chain condition by \cite[Theorem 1]{Amb06}. 
We complete the proof. 
\end{proof}

\section{Precise inversion of adjunction formula for quotient singularities}\label{section:PIA}
In this section, we prove the precise inversion of adjunction formula for the quotient of a complete intersection singularity 
by a finite linear group action (Theorem \ref{thm:PIA}). 

Let $\xi$, $G$, $\gamma$, $C_\gamma$, $Z$, 
$\overline{A}$, $A$, $A^{(\gamma)}$, $r$, 
$\overline{\lambda}_{\gamma}^*$,
$\lambda_{\gamma}$, $q$, 
$\overline{\lambda}_{\gamma}$, 
$f_i$, $I$, $\widetilde{I}^{(\gamma)}$, $\overline{I}^{(\gamma)}$, 
$B$, $\widetilde{B}^{(\gamma)}$, $\overline{B}^{(\gamma)}$, $p$, $\tau$, 
$L_{\overline{B}^{(\gamma)}}$, $\mathfrak{n}'_{1,p}$ and $\mathfrak{n}'_{1,\overline{\mu} _{\gamma}}$ 
be as in Section \ref{section:mld_hyperquot}.

\begin{thm}\label{thm:PIA}
Let $\mathfrak{a} \subset  k[x_1,\ldots,x_N]^G$ be a non-zero ideal and let $\delta$ be a non-negative real number. 
We define $\mathfrak{b} = \mathfrak{a} \bigl( k[x_1,\ldots,x_N]^G/(f_1,\ldots,f_c) \bigr)$. 
Suppose that $\mathfrak{b} \not = 0$. 
Let $x = 0 \in A$ be the origin. 
Suppose that $B$ is klt. Then 
\[
\operatorname{mld}_x \bigl(A,(f_1\cdots f_c) \mathfrak{a} ^{\delta} \bigr) 
= \operatorname{mld}_x (B,\mathfrak{b}^{\delta})
\]
holds. 
\end{thm}

\begin{proof}
Since 
\[
\operatorname{mld}_x \bigl( A,(f_1\cdots f_c)\mathfrak{a} ^{\delta} \bigr) \le \operatorname{mld} _x (B,\mathfrak{b}^{\delta})
\]
is true in general by adjunction, 
it is enough to prove the opposite inequality
\[
\operatorname{mld}_x \bigl( A,(f_1\cdots f_c)\mathfrak{a}^{\delta} \bigr) \ge \operatorname{mld}_x (B,\mathfrak{b}^{\delta}).
\]

By Corollary \ref{cor:mld_hyperquot2} (cf.\ Remark \ref{rmk:multi_index}), we have
\[
\operatorname{mld}_x \bigl( A,(f_1\cdots f_c)\mathfrak{a}^{\delta} \bigr) =  \inf_{w,v\in\mathbb Z_{\ge 0},\gamma\in G} 
\bigl \{ \operatorname{codim}_{\overline{A}'_\infty} (C_{w,v,\gamma}) + \operatorname{age}(\gamma) - w - \delta v  \bigr \}, 
\]
where we set $C_{w,v,\gamma} \subset \overline{A}'_{\infty}$ as 
\[
C_{w,v,\gamma} := \operatorname{Cont}^{\ge w} (f_1 \cdots f_c \mathcal{O}_{\overline{A}'}) \cap 
\operatorname{Cont}^{\ge v} (\mathfrak{a} \mathcal{O}_{\overline{A}'}) 
\cap \operatorname{Cont}^{\ge 1} (\mathfrak{m}_x \mathcal{O}_{\overline{A}'}). 
\]
We fix $w,v \in \mathbb{Z}_{\ge 0}$ and $\gamma \in G$. 

First we note that the following claim holds. The claim will be proved in the end of the proof. 
\begin{claim}\label{claim:not_thin}
$C' \cap \overline{B}^{(\gamma)}_{\infty}$ is not a thin set of $\overline{B}^{(\gamma)}_{\infty}$ for any irreducible component $C'$ of $C_{w,v,\gamma}$. 
\end{claim}

Let $C'_{w,v,\gamma} \subset C_{w,v,\gamma}$ be an irreducible component satisfying 
\[
\operatorname{codim}_{\overline{A}'_\infty} (C_{w,v,\gamma}) = \operatorname{codim}_{\overline{A}'_\infty} (C'_{w,v,\gamma}). 
\]
Set 
\[
b_1 := \min _{\alpha \in C' _{w,v,\gamma} \cap \overline{B}^{(\gamma)}_{\infty}} \operatorname{ord}_{\alpha} (\operatorname{Jac}_{\overline{B}^{(\gamma)}/k[t]}). 
\]
Note that $b_1 < \infty$ by Claim \ref{claim:not_thin}. 
Then by Lemma \ref{lem:EM8.4}, we have 
\begin{align*}
&\operatorname{codim}_{\overline{A}'_\infty} \bigl( C'_{w,v,\gamma} \cap \operatorname{Cont}^{\le b_1}\bigl( (\tau^*)^{-1} \operatorname{Jac}_{\overline{B}^{(\gamma)}/k[t]} \bigr) \bigr) - w \\
&\ge \operatorname{codim}_{\overline{B}^{(\gamma)}_{\infty}} \bigl( C' _{w,v,\gamma} \cap \overline{B}^{(\gamma)}_{\infty} \cap \operatorname{Cont}^{b_1}(\operatorname{Jac}_{\overline{B}^{(\gamma)}/k[t]}) \bigr) - b_1. 
\end{align*}
Since $C'_{w,v,\gamma}$ is an irreducible closed cylinder, 
its non-empty open subcylinder has the same codimension. Therefore 
\[
\operatorname{codim}_{\overline{A}'_\infty} \bigl( C'_{w,v,\gamma} \cap \operatorname{Cont}^{\le b_1} \bigl( (\tau^*)^{-1} \operatorname{Jac}_{\overline{B}^{(\gamma)}/k[t]} \bigr) \bigr)
= \operatorname{codim}_{\overline{A}'_\infty} ( C'_{w,v,\gamma}). 
\]
Hence we have 
\begin{align*}
\operatorname{codim}_{\overline{A}'_\infty} (C_{w,v,\gamma}) - w 
\ge 
\operatorname{codim}_{\overline{B}^{(\gamma)}_{\infty}} \bigl( C _{w,v,\gamma} \cap \overline{B}^{(\gamma)}_{\infty} \cap \operatorname{Cont}^{b_1}(\operatorname{Jac}_{\overline{B}^{(\gamma)}/k[t]}) \bigr) - b_1. 
\end{align*}
Since 
\begin{align*}
& C _{w,v,\gamma} \cap \overline{B}^{(\gamma)}_{\infty} \cap \operatorname{Cont}^{b_1}(\operatorname{Jac}_{\overline{B}^{(\gamma)}/k[t]}) \\ 
& \subset \operatorname{Cont}^{\ge v} (\mathfrak{b} \mathcal{O}_{\overline{B}^{(\gamma)}}) 
\cap \operatorname{Cont}^{\ge 1} (\mathfrak{m}_x \mathcal{O}_{\overline{B}^{(\gamma)}}) 
\cap \operatorname{Cont}^{b_1}(\operatorname{Jac}_{\overline{B}^{(\gamma)}/k[t]}), 
\end{align*}
by applying Corollary \ref{cor:mld_hyperquot2} (cf.\ Remark \ref{rmk:multi_index}) to $B$, we obtain the desired formula 
\begin{align*}
&\operatorname{mld}_x \bigl( A,(f_1\cdots f_c)\mathfrak{a}^{\delta} \bigr) \\
&= \inf_{w,v\in\mathbb Z_{\ge 0},\gamma\in G} 
\bigl \{ \operatorname{codim}_{\overline{A}'_\infty} (C_{w,v,\gamma}) + \operatorname{age}(\gamma) - w - \delta v \bigr \} \\
&\ge 
\inf_{v,b \in\mathbb Z_{\ge 0},\gamma\in G} 
\left \{ \operatorname{codim}_{\overline{B}^{(\gamma)}_{\infty}} 
\left( 
\begin{array}{c}
\operatorname{Cont}^{\ge v} (\mathfrak{b} \mathcal{O}_{\overline{B}^{(\gamma)}}) 
\cap \operatorname{Cont}^{\ge 1} (\mathfrak{m}_x \mathcal{O}_{\overline{B}^{(\gamma)}}) \\
\cap \operatorname{Cont}^{b} \bigl( \operatorname{Jac}_{\overline{B}^{(\gamma)}/k[t]} \bigr) 
\end{array}
\right) 
- b - \delta v  \right \} \\
&= \operatorname{mld}_x (B,\mathfrak{b}^{\delta}). 
\end{align*}
Therefore, it is sufficient to show Claim \ref{claim:not_thin}. 
\begin{proof}[Proof of Claim \ref{claim:not_thin}]
First we introduce some notation. 
For an arc $\alpha \in \overline{A}'_{\infty}$, we denote $g_{i}^{\alpha} := \alpha ^* (x_i) \in k[[t]]$, 
where $\alpha ^* : k[t][x_1, \ldots , x_N] \to k[[t]]$ be the corresponding ring homomorphism. 
We denote by $\beta \in \overline{A}'_{\infty}$ the trivial arc which is determined by $g_{i}^{\beta} := 0$. 
Let $T \subset \overline{A}'$ be the closed subscheme defined by the ideal $(x_1, \ldots , x_N) \subset k[t][x_1, \ldots , x_N]$. 
Then we have $T_{\infty} = \{ \beta \}$. 

Since $\overline{I}^{(\gamma)} \subset (x_1, \ldots , x_N)$, it follows that $\beta \in T_{\infty} \subset \overline{B}^{(\gamma)}_{\infty}$. 
Let $W \subset \overline{B}^{(\gamma)}$ be the irreducible component dominating $\operatorname{Spec} k[t]$, and 
let $h: W' \to W$ be a resolution of singularities of $W$. Let $T' := h^{-1} (T)$. 
Since $B$ is klt by assumption, we note that $\overline{B}$ is also klt by Lemma \ref{lem:purity}. 
Therefore $W$ is klt outside $t = 0$ by Remark \ref{rmk:isotrivial}, 
and hence $T' \to T$ has a section by \cite[Corollary 1.7(2)]{HM07}. 
Hence there exists an arc $\beta ' \in T'_{\infty} \subset W' _{\infty}$ such that $h _{\infty} (\beta ') = \beta$. 

Then by Lemma \ref{lem:thin2}, it is suffices to show that $\beta \in C'$. 
In order to prove it, we introduce a $k$-action on the arc space $\overline{A}'_{\infty}$ as follows. 
Let $\alpha \in \overline{A}'_{\infty}$ and $a \in k$. Then we define $a \cdot \alpha$ by 
\[
g_{i}^{a \cdot \alpha} (x_i) := a^{e_i} g_i ^{\alpha} (a^d t). 
\]
We adopt the convention that $a^{e_i} = 1$ when $a = 0$ and $e_i = 0$. 
Then for $f \in k[x_1, \ldots , x_N]^G$, it is easy to see that $v(t) = u(a^d t)$ when we set
\[
u(t) := \alpha^* \bigl( \overline{\lambda}^*_{\gamma}(f) \bigr), \quad 
v(t) := (a \cdot \alpha)^* \bigl( \overline{\lambda}^*_{\gamma}(f) \bigr) \in k[[t]]
\]
for $\alpha \in \overline{A}'_{\infty}$ and $a \in k$. 
Therefore 
\[
\operatorname{ord}_{\alpha} \bigl( \overline{\lambda}^*_{\gamma}(f) \bigr) = \operatorname{ord}_{a \cdot \alpha} \bigl( \overline{\lambda}^*_{\gamma}(f) \bigr)
\]
holds if $a \in k^{\times}$. 
Hence any cylinder of form $\operatorname{Cont}^{\ge c} (\mathfrak{c} \mathcal{O}_{\overline{B}^{(\gamma)}})$ with 
an ideal $\mathfrak{c} \subset k[x_1, \ldots , x_N]^G$ is invariant under the $k$-action. 
Therefore $C_{w,v,\gamma}$ and its irreducible component $C'$ are invariant under the $k$-action. 
Then the assertion $\beta \in C'$ follows from the observation that $\beta = 0 \cdot \alpha$ holds 
for any $\alpha \in \operatorname{Cont}^{\ge 1} (\mathfrak{m}_x \mathcal{O}_{\overline{B}^{(\gamma)}})$. 
\end{proof}
We complete the proof of the theorem. 
\end{proof}

\begin{rmk}\label{rmk:LC}
In Theorem \ref{thm:PIA}, we assume that $B$ is klt. This assumption is essentially used in the proof of Claim \ref{claim:not_thin}. 

If $B$ is not klt, then the arc space $\overline{B}^{(\gamma)}_{\infty}$ itself might be a thin set of $\overline{B}^{(\gamma)}_{\infty}$. 
Let 
\[
\overline{B} = \operatorname{Spec} k[x_1, x_2, x_3]/(x_1^3 + x_2^3 + x_3^3), 
\quad (d, e_1, e_2, e_3) = (3,0,1,2). 
\]
Then we have 
\[
\overline{B}^{(\gamma)} = \operatorname{Spec} k[t][x_1, x_2, x_3]/(x_1^3 + t x_2^3 + t^2 x_3^3), 
\]
and it follows that $\overline{B}^{(\gamma)}_{\infty} = \{ \beta \}$, where $\beta$ is the trivial arc corresponding to the origin. 
Therefore $\overline{B}^{(\gamma)}_{\infty}$ turns out to be a thin set of $\overline{B}^{(\gamma)}_{\infty}$. 
\end{rmk}

\begin{rmk}\label{rmk:multi_index2}
Theorem \ref{thm:PIA} can be generalized to $\mathbb{R}$-ideals due to Remark \ref{rmk:multi_index}. 
Let $\mathfrak{a}$ be an $\mathbb{R}$-ideal on $A$. 
Then 
\[
\operatorname{mld} _x \bigl( A, (f_1\cdots f_c) \mathfrak{a} \bigr) 
= \operatorname{mld} _x (B , \mathfrak{a} \mathcal{O}_B)
\]
holds. 
\end{rmk}

\begin{rmk}\label{rmk:local}
All the statements in Section \ref{section:mld_hyperquot} are still true 
when we replace by 
$k[x_1, \ldots, x_N]$ by its localization $k[x_1, \ldots, x_N]_g$ at a
$G$-invariant element $g$ which does not vanish at the origin. 
Therefore Theorem \ref{thm:PIA} is also true for the local setting, that is, when 
$B$ is locally defined by a regular sequence $f_1, \ldots, f_c$ at the origin and 
has only klt singularity at the origin. 
\end{rmk}

\section{Proof of the main theorems}\label{section:maintheorem}

As a corollary of Theorem \ref{thm:PIA}, 
we prove the PIA conjecture for hyperquotient singularities. 
\begin{cor}\label{cor:PIA_general}
Suppose that a finite subgroup $G \subset {\rm GL}_N(k)$ acts on $\mathbb{A}_k^N$ freely in codimension one. 
Let $X := \mathbb{A}_k^N / G$ be the quotient variety and let $x \in X$ be the image of the origin of $\mathbb{A}_k^N$. 
Let $Y$ be a subvariety of $X$ through $x$ of codimension $c$, and 
let $\mathfrak{a}$ be an $\mathbb{R}$-ideal sheaf on $Y$. 
Suppose that $Y$ is locally defined by $c$ equations at $x$ in $X$. 
Let $D$ be a Cartier prime divisor on $Y$ through $x$ with a klt singularity at $x \in D$. 
Suppose that $D$ is not contained in the cosupport of the $\mathbb{R}$-ideal sheaf $\mathfrak{a}$. 
Then it follows that
\[
\operatorname{mld}_x \bigl( Y, \mathfrak{a} \mathcal{O}_Y (-D) \bigr) = 
\operatorname{mld}_x (D, \mathfrak{a} \mathcal{O}_D). 
\]
\end{cor}
\begin{proof}
Let $R := k[x_1, \ldots , x_N]^G$ be the invariant ring. 
Take an $\mathbb{R}$-ideal sheaf $\mathfrak{b}$ on $X$ such that $\mathfrak{a} = \mathfrak{b} \mathcal{O}_Y$, 
and take local equations $f_1, \ldots , f_c \in \mathcal{O}_{X,x}$ of $Y$ in $X$. 
Furthermore, take $g \in \mathcal{O}_{X,x}$ such that its image $\overline{g} \in \mathcal{O}_{Y,x}$ defines $D$. 
We note that $Y$ has a klt singularity at $x$ by inversion of adjunction (cf.\ \cite[Theorem 5.50]{KM98}). 
Then it follows that
\[
\operatorname{mld}_x \bigl( Y, \mathfrak{a} \mathcal{O}_Y (-D) \bigr) = 
\operatorname{mld}_x \bigl( X, (f_1 \cdots f_{c} \cdot g)\mathfrak{b} \bigr) =
\operatorname{mld}_x ( D, \mathfrak{a} \mathcal{O}_D )
\]
by applying Theorem \ref{thm:PIA} twice (cf.\ Remark \ref{rmk:local}). 
\end{proof}

\begin{thm}\label{thm:LSC_general}
Suppose that a finite subgroup $G \subset {\rm GL}_N(k)$ acts on $\mathbb{A}_k^N$ freely in codimension one. 
Let $X := \mathbb{A}_k^N / G$ be the quotient variety. 
Let $Y$ be a subvariety of $X$ of codimension $c$ which has only klt singularities, and 
let $\mathfrak{a}$ be an $\mathbb{R}$-ideal sheaf on $Y$. 
Suppose that $Y$ is locally defined by $c$ equations in $X$. 
Then the function 
\[
|Y| \to \mathbb{R}_{\ge 0} \cup \{ - \infty \}; \quad y \mapsto \operatorname{mld}_y(Y,\mathfrak{a})
\]
is lower semi-continuous, where we denote by $|Y|$ the set of all closed points of $Y$ with the Zariski topology. 
\end{thm}
\begin{proof}
We take over the notation of the proof of Corollary \ref{cor:PIA_general}. 
Since the lower semi-continuity holds for $X$ by \cite[Corollary 1.3]{Nak16}, 
it is sufficient to show that 
\[
\operatorname{mld}_y \bigl( Y, \mathfrak{a} \bigr) = 
\operatorname{mld}_y \bigl( X, (f_1 \cdots f_{c})\mathfrak{b} \bigr)
\]
for any $y \in Y$. 

We fix a closed point $y \in Y$. 
Take a closed point $y' \in \mathbb{A}_k^N$ whose image in $Y$ is $y$. 
Let $G_{y'} := \{ g \in G \mid g(y') = y'\}$ be the stabilizer group of $y'$. 
Then $\mathbb{A}^N / G_{y'} \to \mathbb{A}^N / G$ is \'{e}tale at $y$. 
We note that the minimal log discrepancy is preserved under an \'{e}tale map. 
Hence by replacing $X = \mathbb{A}_k^N / G$ by $\mathbb{A}^N / G_{y'}$ and changing the coordinate of $\mathbb{A}^N$, 
we may assume that $y'$ is the origin and the group action is still linear. 
Then we have 
\[
\operatorname{mld}_y \bigl( Y, \mathfrak{a} \bigr) = 
\operatorname{mld}_y \bigl( X, (f_1 \cdots f_{c})\mathfrak{b} \bigr)
\]
by Theorem \ref{thm:PIA}, which completes the proof. 
\end{proof}


\end{document}